\documentclass[oneside,english,10pt]{amsart}
\usepackage{color}
\usepackage{amsthm}
\usepackage{amstext}
\usepackage{graphicx}
\usepackage{amscd}
\usepackage{amsfonts}
\usepackage[colorlinks,linkcolor=blue, pdfstartview=FitH]{hyperref}
\usepackage{amssymb}
\usepackage{esint}
\usepackage{hyperref}
\usepackage{amsmath, amssymb, amsthm, mathrsfs, graphicx}
\usepackage{hyperref, enumitem, xspace, ifthen,comment}
\usepackage[all]{xy}
\usepackage{tikz}
\usetikzlibrary{arrows,shapes,trees}

\hypersetup{citecolor=red}

\begin{document}

\title[\null]
{ $L^2$-extension theorems for jet sections of nef holomorphic vector bundles on compact K\"ahler manifolds and  rational homogeneous manifolds, I.}

\author[\null]{Qilin Yang}

\address{ Department of Mathematics,
Harvard University,  ~Cambridge,~MA ~02138,~U. S. A.} %
\email{qlyang@math.harvard.edu}
\address{ Department of Mathematics,
Sun Yat-Sen University, ~510275, ~ Guangzhou, ~P. R.  CHINA.} %
\email{yqil@mail.sysu.edu.cn}


\begin{abstract}
In this paper we study holomorphic vector bundles with singular Hermitian metrics whose curvature are Hermitian matrix currents. We obtain an extension theorem for
holomorphic jet sections of nef holomorphic vector bundle on compact K\"ahler manifolds. Using it we prove that Fano manifolds with strong Griffiths nef tangent bundles are rational homogeneous spaces.
\end{abstract}

\maketitle
\baselineskip=10pt


\newcommand{\C}{\Bbb C}
\newcommand{\kd}{\mathcal{D}}
\newcommand{\ko}{\mathscr{O}}
\newcommand{\ki}{\mathcal{I}}
\newcommand{\kj}{\mathscr{J}}
\newcommand{\kf}{\mathcal{F}}
\newcommand{\Pj}{\Bbb P}
\newcommand{\Z}{\Bbb Z}
\newcommand{\R}{\Bbb R}
\newcommand{\Q}{\ensuremath{\mathbb Q}}
\newcommand{\lnorm}{\left|\left|}
\newcommand{\rnorm}{\right|\right|}
\newcommand{\fk}{\mathfrak{k}}
\newcommand{\fp}{\mathfrak{p}}
\newcommand{\fm}{\mathfrak{m}}
\newcommand{\cf}{\mathrm{cf.}}
\newcommand\grad{\rm grad}
\newcommand\lgr{\longrightarrow}
\newcommand\rw{\rightarrow}
\newcommand\lmp{\longmapsto}
\newcommand\lie{{\cal L}{\rm ie}}
\newcommand\Lm{\Lambda}
\newcommand\lmd{\lambda}
\newcommand\lc{\lceil}
\newcommand\rc{\rceil}
\newcommand\al{\alpha}
\newcommand\bt{\beta}
\newcommand\el{\epsilon}
\newcommand\gm{\gamma}
\newcommand\dt{\delta}
\newcommand\G{\Gamma}
\newcommand\eg{e.g.}
\newcommand\ie{i.e.}
\newcommand\resp{\rm resp.}
\newcommand\Her{\rm Her}
\newcommand\Aut{\rm Aut}
\newcommand\Hom{\rm Hom}
\newcommand\ov{\overline}
\newcommand\be{\beta}
\newcommand\om{\omega}
\newcommand\df{\rm d}
\newcommand\bi{\bar{i}}
\newcommand\bj{\bar{j}}
\newcommand\bl{\bar{l}}
\newcommand\bk{\bar{k}}
\newcommand\bs{\bar{\xi}}
\newcommand\bp{\bar{p}}
\newcommand\gd{\rm grad}
\newcommand\lge{\langle}
\newcommand\rg{\rangle}
\newcommand\sta{\Theta_H }
\newcommand\vk{{\varkappa}_{\tau}}
\newcommand\wh{\widetilde{H}}
\newcommand\wf{\widetilde{F}}
\newcommand\pqm{\Lambda^{p,q}T^*_M }
\newcommand\ppm{\Lambda^{p,p}T^*_M }
\newcommand\pqme{\Lambda^{p,q}T^*_M\otimes E }
\newcommand\dpq{C^k_0(M,\Lambda^{p,q}T^*_M) }
\newcommand\dfg{C^k_0(M,\Lambda^{f,g}T^*_M\otimes \C^r) }
\newcommand\dst{C^k_0(M,\Lambda^{s,t}T^*_M\otimes \C^r) }
\newcommand\dnpnq{C^k_0(M,\Lambda^{n-p,n-q}T^*_M) }
\newcommand\dnqi{C^{\infty}_0(M,\Lambda^{p,q}T^*_M) }
\newcommand\dnqie{C^{\infty}_0(M,\Lambda^{n,q}T^*_M\otimes E) }
\newcommand\dpqd{\mathcal{D}'^k (M,\Lambda^{p,q}T^*_M) }
\newcommand\dpqe{\mathcal{D}(M,\Lambda^{p,q}T^*_M\otimes E) }
\newcommand\lpqe{L^2 (M,\Lambda^{p,q}T^*_M\otimes E) }
\newcommand\lnqe{L^2 (M,\Lambda^{n,q}T^*_M\otimes E) }
\newcommand\End{\mbox{{\rm End}}\, }
\newcommand\im{\mbox{{\rm im}}\, }
\newcommand\Ima{\mbox{{\rm Im}}\, }
\newcommand\Ker{\mbox{{\rm Ker}}\, }

\newtheorem{theorem}{Theorem}[section]
\newtheorem{proposition}[theorem]{Proposition}
\newtheorem{corollary}[theorem]{Corollary}
\newtheorem{lemma}[theorem]{Lemma}
\newtheorem{example}[theorem]{Example}
\newtheorem{remark}[theorem]{Remark}
\newtheorem{definition}[theorem]{Definition}

\tableofcontents

\section{Introduction}
In \cite{mor} and \cite{sy}, Mori, Siu and Yau, using different methods, proved the Hartshorne-Frankel  conjecture, which says that a compact K\"ahler manifold with positive bisectional curvature is biholomorphic to the complex projective space. A beautiful theorem established later by  Mok in \cite{mok}, which states that a simply connected irreducible compact K\"ahler manifold with non-negative bisectional curvature such that the Ricci curvature is positive at one point is isometric to the complex projective space or the irreducible compact Hermitian symmetric space of rank $\geq 2.$  Using the splitting theorem of Howard-Smyth-Wu, we get a complete classification of all compact simply connected K\"ahler manifolds of nonnegative holomorphic bisectional curvature.

In \cite{cp}, Campana and Peternell initiate the classification of projective manifolds
whose tangent bundles are nef, a more degenerate curvature condition than non-negative bisectional curvature.  Recall that a holomorphic line bundle $L$ on a projective manifold 
$M$ is
called nef if the intersection number $L\cdot C\geq 0$ for every irreducible curve $C\subset M,$ and a holomorphic vector bundle on $M$ is called nef if the associated tautological line bundle 
$\ko_{\Pj (E^*)}(1)$ is a nef line bundle. If we denote the first Chern class of $L$ by $c_1(L)$ then  $L\cdot C=\int_C c_1(L)$ when $C$ is smooth.
 Hence a holomorphic vector bundle 
over a projective manifold with nonnegative bisectional curvature  is a nef vector bundle, but the converse is not always  true.
Later in \cite{dps},  Demailly, Campana and Peternell generalized the definition of nefness for holomorphic vector bundles on projective manifolds to arbitrary compact complex manifolds.
 They proved 
that the Albanese morphism of  a compact  K\"ahler manifold with nef tangent bundle is a submersion and whose fibers are  Fano manifolds with nef tangent bundles. Since by Mok's theorem any Fano manifold with nonnegative
holomorphic bisectional curvature is a Hermitian symmetric space, in view of this,
 Campana and Peternell \cite{cp} formulated a corresponding  conjecture claims that a Fano manifold with nef tangent bundle is  a rational homogeneous space, i.e., quotient of a semisimple Lie group by a parabolic subgroup. They showed that their conjecture is  true up to dimension four. Later many special cases are checked to be true, we will not report them here even there are some very important ideals and techniques found during recent years.
 We refer readers to the  survey \cite{sur} for a detailed report of related works.

In proving Hartshorne-Frankel  conjecture, Mori, Siu and Yau used the characterization  of projective space established by Kobayashi and  Ochiai. Mok's proof of his famous theorem also used the Berger-Simons theorem on the characterization of  irreducible symmetric spaces of rank $\geq 2.$ 
However, until now we haven't a unified theorem which characterize   all of the  rational homogeneous manifolds at the same time.  This is a main difficult  when we are
in search of a proof of the Campana-Peternell conjecture. In  \cite{br}, Borel and Remmert  established a beautiful structure  theorem of compact homogeneous K\"ahler manifolds which says 
that they are direct product of complex tori and rational homogeneous manifolds.  Since the Fano manifolds are simply connected K\"ahler manifolds, if we could prove they are holomorphical homogeneous then they are
 rational homogeneous.  The homogeneity,  a relaxed characterization, is  much easier to deal with than giving a straight way  and a precise characterization of the rational homogeneous manifolds. 

How to prove a compact complex manifold is homogeneous? A natural way is to  decide its holomorphic automorphism group (which is a complex Lie group by Bochner-Montgoemery's theorem)  and   prove  the automorphism group's  action is homogeneous. However it is usually very difficult to calculate  the holomorphic automorphism group of a complex manifold. passing to  infinitesimal level, if we could prove it has many holomorphic vector fields, then its holomorphic automorphism group, whose Lie algebra  could be identified with the set of global holomorphic vector fields, would be large enough to act transitively on it.  

The action of the diffeomorphism group of a real manifold is alway transitive since we could glue the real local vector fields to global vector fields in a smooth way. Every complex manifold has a Stein open covering  and every Stein open subset have many holomorphic vector fields.  However it is very difficult to patch up analytically the local holomorphic vector fields to get global holomorphic vector fields.
This is  due to the rigidity of holomorphic objects. But the rigidity may have an advantage in some special cases,  where  it could  propagate form local to global.
We also called this phenomenon  the extension or analytical continuation of a holomorphic object. The corresponding theorems are called extension theorems,
and was extensively studied for many years  and has a long history.

Siu obtained many important extension theorems and explained related techniques in the monograph  \cite{sbook} ,   see also his ICM lecture \cite{siu80} for  a report.
 It is well known (\cite{siu74}) that we could extend a $L^2$-bounded holomorphic function  form a sub manifold to an open neighborhood, however, we did not know how large  the scale of the open neighborhood is until
 the  work of Ohsawa-Takegoshi appeared  in \cite{ot}. Different form the older extension theorems,  the Ohsawa-Takegoshi's extension theorem  gains more information including controlling the $L^2$ norm of the extended holomorphic functions in a precise way,  hence it found many applications.  A surprise applications is to prove   the deformational invariance of plurigenera of the algebraic variety of general type by Siu in \cite{siu98},   he successfully extended holomorphic object  from an open submanifold to a family of compact complex manifold. This is usually impossible as we know even an open complex manifold have many holomorphic holomorphic functions but none  of them could extend to the compactified complex manifold. 
  
 A careful check the assumptions in the Campana-Peternell conjecture, we meet the following difficulties. Firstly,  we have no extension theorems of the Ohsawa-Takegoshi type for holomorphic tangent bundle of high dimensional complex manifold. Though the  extension theorems of the Ohsawa-Takegoshi type are extensively studied in recent years, almost all are concentrate in line bundles case. Secondly, the nefness assumption of the tangent bundle as a curvature condition is imposed in fact on the tautological line bundle of projective bundle of the dual of the tangent bundle, rather directly on the tangent bundle itself. Hence we need to construct Hermitian metrics on the tangent bundle  only under the assumption that the tautological line bundle have a singular Hermitian metric whose curvature is semi positive  in the sense of distribution. It is not so difficult to construct Hermitian metric on the tautological line 
 bundle form those on the tangential bundle and the inverse process is not known up to present.  In this paper we will study a weaker version  of Campana-Peternell conjecture by changing the assumption of curvature conditions, imposing them  directly on the tangential bundle.  Accordingly, as a solution of the first difficult, 
 we will study extension theorems for higher rank holomorphic vector bundles with singular Hermitian metric. 
 
What curvature condition imposed on a holomorphic vector bundle which is close to that the tautological line  bundle is nef?  We know nefness is a generalization of  ampleness: a holomorphic vector bundle is called ample if the corresponding tautological line bundle is ample. A vector bundle is ample is equivalent to that  it is Griffith positive.  Hence we have a  weaker version  of Campana-Peternell conjecture that a Fano manifold with 
``Griffiths nef" (Definition \ref{nefff}) tangential bundle is rational homogeneous, here the new definition is a nef version of Griffiths positivity, just like nef line bundle is a generalization a Griffiths  positive line bundle or equivalently, an ample line bundle.  In this paper we add more condition on the positivity on the tangent bundle, call the Griffiths strongly nef positivity and prove the following main result:       
   
  \vskip0.5cm

\noindent{\bf Main Theorem.}{\it A Fano manifold whose tangent bundle is strongly Griffiths nef is a rational homogeneous space}.

\vskip0.5cm

The  main tool used in proving our main theorem is a $L^2$-extension theorem on compact K\"ahler manifold,  which we will give in Section  \ref{subs4} after a series preliminary work
from Section  \ref{subs2} to  Section  \ref{subs3}.
We will study holomorphic vector bundles with singular Hermitian metrics, to define their curvatures we need to study Hermitian matrix current. In Section \ref{subs2} we establish a representation theorem for  Hermitian matrix current, and study the positivity of  the Hermitian matrix currents. In Section \ref{subs3} we introduce the singular Hermitian metrics  via an example and give the definitions of Griffiths nef and Nakano nef holomorhic vector bundles, we clarify the relations among Griffith nefness, Nakano nefness, and nefness defined in \cite{dps}(called nef in usual sense). 
Finally we gives  Bochner-Kodaira identity for holomorphic vector bundles with singular Hermitian metrics
and existence theorem for $\bar{\partial}$ and approximate $\bar{\partial}$-equations which are extensively used in the next section for studying extension theorem. 
In Section \ref{subs4} we  establish Ohsawa-Takegoshi type extension theorems for holomorphic jet sections with finite isolated support of Nakano semi positive  vector bundles and strong Nakano nef vector bundles over bounded Stein domains and compact K\"ahler manifolds respectively. Finally in Section \ref{subs5} we prove that a  compact K\"ahler manifold with strong Griffiths  nef tangent bundle is  homogeneous, and as a corollary we prove the Main theorem.

\section{Hermitian matrix current}\label{subs2}

Singular Hermitian metric on a holomorphic line bundle $L$ over a complex manifold was systematically studied by Siu  and Demailly, we may refer to  the ICM lectures \cite{siu02} and \cite{de06} for a report of related works.  If we write the  metric of $L$ locally as $h=e^{-\varphi},$ where $\varphi$ is only a locally integrable function. It is locally a PSH function if $L$ is a nef or more generally a pseudoeffective line bundle. The Chern  curvature of $L$ is $\Theta_h=id'd''\varphi.$ Here the derivatives are in the sense of distribution and $\Theta_h$ is a closed positive current.
In this paper we will consider similar constructions for higher rank holomorphic vector bundles, i.e.,
study holomorphic vector bundles with measurable metrics. The curvatures are  the weak differentials of metrics with measurable coefficients, we need to consider them in the sense of distribution. We refer readers to \cite{siu74}, \cite{bt} and \cite{debo} for the theory  of closed positive currents. In this section we will give the definitions of weak derivatives of a Hermitian matrix function and study the holomorphic vector bundle with singular metrics whose curvatures are positive Hermitian matrix currents.

Let $M$ be a complex manifold of complex dimension $n$ and $E$ a smooth complex vector bundle of complex rank $r$ on $M.$ In this paper we use $C^k_0 (M,E)$ to denote the set of $E$-valued sections with compact supports and whose derivatives up to $k$-order are continuous; and  $L^{p}(M,E)$({\rm resp.} $L^{p}_{loc}(M,E)$) to denote the set of  $E$-valued sections which are $L^p$ ({\rm resp.} locally $L^{p}$) integrable; and $L^{\infty}(M,E)$({\rm resp.} $L^{\infty}_{loc}(M,E)$) to denote the set of  $E$-valued sections which are bounded ({\rm resp.} locally bounded); and $W^{k,p}(M,E)$
to denote the set of sections $s\in   L^{p}(M,E)$ whose weak derivatives up to $k$-order are also in   $L^{p}(M,E).$ 

If $E$ is Hermitian vector bundle, we use $\Her(E)$ to denote the sub-bundle of $\End(E),$ consisting of Hermitian transformations on each fibre of $E,$ and  $\Her^+(E)$ the subset of $\Her(E)$ whose elements are everywhere non-negative definite Hermitian  transformations. We use $C^k(M,\Her(E))$ to denote the set of Hermitian transformation with matrix representation $H=(h_{\al\bar{\bt}})_{r\times r}$ such that every entry $h_{\al\bar{\bt}}$ is continuous differentiable up to $k$-order. Similarly we can define $L^p(M,\Her(E))$ and  $L^p_{loc}(M,\Her(E))$ for any $1\leq p\leq \infty.$ If $E$ is a trivial bundle we denote, for example,
$C^k(M,\Her(E))$ by $C^k(M,\Her(\C^r)).$

 Let $T_M$ be the holomorphic tangent bundle of $M,$  $T^*_M$ and $\overline{T}^*_M$  denote  the holomorphic and antiholomorphic cotangent bundles of $M$  respectively. Let $\pqm$ be the tensor product bundle $\wedge^p T^*_M \otimes \wedge^p \overline{T}^*_M.$ Equip the space
 $\dpq$  with the following topology induced by semi-norms: write $u\in \dpq$ in local holomorphic coordinate $(z_1,\cdots,z_n)$ as $$u=\sum_{|I|=p,|J|=q} u_{I\bar{J}}(z)dz_I\wedge d\bar{z}_{J},$$
  where $dz_I$ denotes $dz_{i_1}\wedge\cdots\wedge dz_{i_p},$ while multi-index $I$ is a short writing  for $(i_1,\cdots,i_p)$
  and $|I|=i_1+\cdots+i_p.$ The semi-norms $p^k_K$ on $u$ for any compact subset $K\subset M,$   are defined by
 \begin{align}p^k_{K}(u)=\sup_{x\in K}\max_{|\al|\leq k}\max_{|I|=p,|J|=q}|\partial^{\al}u_{I\bar{J}}(x)|,\end{align}
 where $\al=(\al_1,\cdots,\al_{2n}),$ and  $x=(x_1,\cdots,x_{2n})$ are real coordinate of $M,$ and $\partial^{\al}=\partial^{|\al|}\slash \partial x_1^{\al_1}\cdots \partial x_{2n}^{\al_{2n}}.$
 With respect to these semi-norms $\dpq$ is a complete locally convex topological space.
   Let $\dpqd$  be the space of continuous linear functionals on $\dnpnq,$ its elements are called the $k$-order currents of bidegree $(p,q).$

 The space ${C^k_0}(M,\Lambda^{p,q}T^*_M\otimes\C^r)$ of $\C^r$-valued differential forms has a naturally defined  topology induced by the topology on
 ${C^k_0}(M,\Lambda^{p,q}T^*_M),$ defined by the seminorms $$p^k_K(u)=\max_{1\leq j\leq r} p^k_K(u_j)$$ for any $u=(u_1,\cdots,u_r)\in {C^k_0}(M,\Lambda^{p,q}T^*_M\otimes\C^r)$ with $u_j\in\dpq.$

A {\it matrix current of order $k$} and bidegree $(p,q)$ on $M$ is a  bilinear function $\dfg\times \dst\rightarrow \C,(u,v)\mapsto U(u,v),$ where $f+s=n-p$ and $g+t=n-q,$ satisfying (I). $U$ is bilinear linear, i.e.,   $U(au,bv)=abU(u,v)$ for any $a,b\in\C;$ and (II).  $U$ is continuous, which means for any compact subset $K\subset M,$ there exists a positive constant $C$ such that
\begin{align}|U(u,v)|\leq C\sum_{r+s=k} p^r_K(u)p^s_K(v).\label{norm}\end{align}
  Given a matrix current $U$ we may get a sesquelinear linear function $\hat{U}$ defined by $\hat{U}(u,v)=U(u,\bar{v}).$
   We called $\hat{U}$  {\it a sesquelinear linear current}. We can also inverse the process to obtain a matrix current form a sesquelinear linear current.
 Since the set of  order $k$ and bidegree $(p,q)$- matrix currents is the same as the set of order $k$ and bidegree $(p,q)$- sesquelinear linear currents, we denote them by the same symbol $D'^k_{p,q}(M,\End(\C^r)).$

  Given a sesquelinear linear current $H\in D'^k_{p,p}(M,\End(\C^r)),$ its conjugate $\bar{H}:\dfg\times \dst\rightarrow \C$ is defined  by $\bar{H}(\xi,\eta)=\overline{H({\xi},{\eta})};$ its transpose $H^t:\dst\times \dfg\rightarrow \C$ is defined by $H^t(\xi,\eta)=H(\eta,\xi).$ A sesquelinear linear current $H$ of order $k$ and bidegree $(p,p)$
 is called a {\it  Hermitian matrix current}  of order $k$ and bidegree $(p,p)$ if  it is Hermitian symmetric, i.e., $\bar{H}=H^t.$
 The set of  order $k$ and bidegree $(p,p)$- Hermitian matrix currents is denoted by $D'^k_{p,p}(M,\Her(\C^r)).$

\begin{example}{\rm  A  Hermitian matrix current of bidegree $(0,0)$ is called a {\it generalized Hermitian matrix function}.
 Let $L^1_{loc}(M, \Her ^+(T_M))$ be the set of Hermitian metrics $H=(h_{\al\bar{\beta}})_{n\times n}$ on $M,$ where $h_{\al\bar{\beta}}$ is locally defined function and $h_{\al\bar{\beta}}\in L^1_{loc}(\Omega)$ is locally integrable for any open subset  $\Omega\subset M.$ Then $H$ defines a generalized Hermitian matrix function. For example for $\xi,\eta\in {C^k_0}(M,T_M )$ with notion $\eta^t=(\eta_1,\cdots,\eta_n)$ the transpose of the column  vector $\eta$, let $ d\mu$ denote the Lebesgue measure then $\xi d\mu \in {C^k_0}(M,\Lambda^{n,n}T^*_M\otimes T_M ),$
   The function $\int_{M}{\eta}^t H\bar{\xi} d\mu$  is clearly sesquelinear, continuous, and symmetric.} \end{example}

   \begin{example}{\rm  Let $U=(U_{\al \bt})_{r\times r}$ be a matrix of currents of of $k$-order and bidegree $(p,q),$ and each entry $U_{\al \bt}\in D'^k(M,\Lambda^{p,q}T^*_M)$ is a $(p,q)$-form with $k$-order distribution coefficients. Define
   \begin{equation}U(u,v)= \int_M  u^t Uv,\label{matr}\end{equation}
   where
   $$\begin{array}{rcl}
   u=(u_1,\cdots,u_r)^t &\in& \dfg\\
   v=(v_1,\cdots,v_r)^t &\in& \dst\\
    \end{array}$$
    with $n=f+r+p=g+s+q,$ and $u_{\al}\in C^k_0(M,\Lambda^{f,g}T^*_M)$ and $v_{\al}\in C^k_0(M,\Lambda^{\al,\bt}T^*_M)$ for $\al=1,\cdots,r,$  and for brevity we don't write out the wedges in (\ref{matr}) and
    $u^t U v=\sum_{\al,\bt}   u_{\al}\wedge U_{\al\bt}\wedge v_{\bt}.$
       It is easy to check that $U$ is a matrix current of order $k$ and bidegree $(p,q).$ If $(H_{\al\bt})$ is Hermitian symmetric matrix of currents  of $k$-order and bidegree $(p,p),$ that is if $\bar{H}_{\al \bt}=H_{\bt\al}$ for all $\al,\bt,$ then
   \begin{equation}H(u,v)=\int_M u^t H \bar{v},\label{matr2}\end{equation}
   defines a Hermitian matrix current $H.$
   }
   \end{example}

\begin{proposition}A matrix current $U\in {D'}^k_{p,q}(M,\End(\C^r))$  is a matrix of currents of order $k$ and bidegree $(p,q).$
\end{proposition}
\begin{proof} Fix $u\in {C^k_0}(M,\Lambda^{f,g}T^*_M\otimes\C^r)$ then $U(u,\cdot): {C^k_0}(M,\Lambda^{s,t}T^*_M\otimes\C^r)\rw \C$ is a linear function satisfying  $$|U(u,v)|\leq C_u p^k_K(v)$$ for any compact subset $K\subset M.$
By (\ref{norm}), $C_u$ is a fixed constant depending only $u.$ Hence $U(u,\cdot)$ is a $k$-order continuous function on  ${C^k_0}(M,\Lambda^{p,q}T^*_M\otimes\C^r).$ Write $v=(v_1,\cdots,v_r)$ with $v_j\in {C^k_0}(M,\Lambda^{p,q}T^*_M).$ Identify the vectors
$$(v_1,0,\cdots,0),\cdots,(0,\cdots,0,v_r)$$  with the forms $v_1,\cdots,v_r$ then  $U(u,v)=\sum_{\bt=1}^r U(u,v_{\bt})$ and  we may consider $U(u,v_{\bt})$ as a linear function on ${C^k_0}(M,\Lambda^{s,t}T^*_M).$
 Hence
\begin{align}U(u,v) =\sum_{\bt=1}^r\int_M U_{\bt}(u)\wedge v_{\bt},\label{current1}\end{align}
where $U_{\bt}(u)$ are $k$-order currents of degree $(n-r,n-s).$ Write $u=(u_1,\cdots,u_r)$ we may get
\begin{align}U(u,v) =\sum_{\al,\bt=1}^r U(u_{\al},v_{\bt})=\sum_{\bt=1}^r\int_M U_{\bt}(u_{\al})\wedge v_{\bt},\label{current2}\end{align}
Now fix $v_{\bt}$ we consider each term $F_{\bt}(u_{\al}):=\int_M U_{\bt}(u_{\al})\wedge v_{\bt}$ as a linear function on ${C^k_0}(M,\Lambda^{f,g}T^*_M).$ Then $$|F_{\bt}(u_{\al})|\leq C_{v_{\bt}} p^k_K(u_{\al}),$$ where  $C_{v_{\bt}}$ is a constant depending only on ${v_{\bt}} .$ 
For any multi-indices $I,J$ with $|I|=p$ and $|J|=q,$ take $u_{\al}=h_1 dz_P\wedge d\bar{z}_Q\in {C^k_0}(M,\Lambda^{f,g}T^*_M)$ and $v_{\bt}=h_2dz_R\wedge d\bar{z}_S\in {C^k_0}(M,\Lambda^{s,t}T^*_M)$ such that $(P\cup R)\cup I=(Q\cup S)\cup J =\{1,\cdots,n\}.$
 Denote $h=h_1 h_2 dz_1\wedge\cdots\wedge dz_{n}\wedge d\bar{z}_1\wedge\cdots\wedge d\bar{z}_{n}\in {C^k_0}(M,\Lambda^{n,n}T^*_M),$ and set $$U_{\al\bt,IJ}(h)={\rm Sgn}(IJPQRS)\int_M U_{\bt}(u_{\al})\wedge v_{\bt}=U(u_{\al},v_{\bt}),$$ where ${\rm Sgn}(IJPQRS)$ is the signature of the permutation
 $(1,\cdots,n,1\cdots,n)\longmapsto (P,Q,I,J,R,S).$
 Then there exist constants $C$ and $\bar{C}$ such that $$|U_{\al\bt,IJ}(h)|\leq C\sum_{r+s=k} p^r_K (h_1) p^s_K (h_2)\leq \bar{C} p^k_K (h)$$ for any compact subset $K\subset M.$ Hence $U_{\al\bt,IJ}$ is a $k$-order distribution. Let
 $$U_{\al\bt}=\sum_{|I|=p,|J|=q}U_{\al\bt,IJ}dz_I\wedge d\bar{z}_J\in \dpqd.$$
Now it is easy to check that \begin{align}F_{\bt}(u_{\al})=\int_M u_{\al} \wedge U_{\al\bt}\wedge v_{\bt}.\label{current3}\end{align}
By (\ref{current1}),(\ref{current2}),(\ref{current3}) we know $U(u,v)=\int_M u^t U v$ and $U$ is a matrix of currents of order $k$ and bidegree $(p,q).$
\end{proof}

\begin{corollary}A Hermitian matrix current $H\in {D'}^k_{p,p}(M,\End(\C^r))$  is a Hermitian matrix of currents of order $k$ and bidegree $(p,p).$
\end{corollary}

   The weak  differentials of a matrix  current $U\in {D'}^k_{p-1,q}(M,\End(\C^r)),$ considered as a bilinear current,  are defined by
  \begin{equation}d'U(\xi,\eta)=-(-1)^{f+g}U(d'\xi,\eta)+(-1)^{p+q}U(\xi,d'\eta)\label{diff1}\end{equation}
  \begin{equation}d''U(\xi,\eta)=-(-1)^{f+g}U(d''\xi,\eta)+(-1)^{p+q}U(\xi,d''\eta)\label{diff2}\end{equation}
  for any  $\xi\in\dfg$ and $\eta\in\dst$ with $f+s=n-p$ and $g+t=n-q.$  If $U\in {D'}^k_{p-1,q}(M,\End(\C^r))$ is  a sesquelinear current,  its weak differential are defined by 
  \begin{equation}d'U(\xi,\eta)=-(-1)^{f+g}U(d'\xi,\eta)+(-1)^{p+q}U(\xi,d''\eta)\label{diff3}\end{equation}
   \begin{equation}d''U(\xi,\eta)=-(-1)^{f+g}U(d''\xi,\eta)+(-1)^{p+q}U(\xi,d'\eta)\label{diff4}\end{equation}
  for any  $\xi\in\dfg$ and $\eta\in C^{\infty}_0 (M,\Lambda^{t,s}T^*M\otimes \C^r)$ with $f+s=n-p$ and $g+t=n-q.$
   Set $dU=d'U+d''U.$ A (Hermitian) matrix current $U$ is called {\it closed} if $dU=0.$

  \begin{proposition}\label{dq0} For any  matrix current $U\in D{'}_{p,q}(M,\End(\C^r)),$ $d'^2 U=d''^2U=d^2U=0.$
  \end{proposition}
  \begin{proof}For  any  $\xi\in\dfg$ and $\eta\in\dst$ with $f+s=n-p-2$ and $g+t=n-q.$
  $$\begin{array}{rcl}
  (d'^2U)(\xi,\eta)&=&-(-1)^{f+g} (d'U)(d'\xi,\eta)+(-1)^{p+q-1}(d'U)(\xi,d'\eta)\\
                   &=&-(-1)^{f+g+p+q}U(d'\xi,d'\eta)-(-1)^{f+g+p+q-1}U(d'\xi,d'\eta),
                   \end{array}$$
                   hence $d'^2=0.$ In the same way we have $d''^2 U =d^2 U=0.$
  \end{proof}

Identify a  matrix current $U\in D{'}_{p,q}(M,\End(\C^r))$ with a matrix $(U_{\al\bt})_{r\times r}$ of currents. For any
 $\xi\in\dfg$ and $\eta\in\dst$ with $f+r\leq n-p$ and $g+t\leq n-q,$
  define $$U(\xi,\eta)=\xi^t U \eta,$$
  then $U(\xi,\eta)\in D'^k(M,\Lambda^{a,b}T^*_M)$ is a current of $k$-order and bidegree $(a,b)$ with $f+r+a= n-p$ and $g+t+b= n-q.$

   \begin{definition}\label{matrix} A Hermitian matrix $(p,p)$-current $H\in {D'}^k_{p,p}(M, \Her(\C^r))$  is called positive if for any non zero  $\xi\in {C^k_0}(M,\C^n)$  the  current $H(\xi,\xi)$ is a positive $(p,p)$-current; which means
   for every choice of smooth $(1,0)$-forms $\al_1,\cdots,\al_p$ on $M$ the distribution
   $$H(\xi,\xi)\wedge i\al_1 \wedge\bar{\al}_1\wedge\cdots\wedge i\alpha_p\wedge\bar{\al}_p$$
   is a positive measure on $M.$ We denote the set of positive Hermitian matrix currents of $k$-order by  ${D'}^k_{p,p}(M, \Her^+ (\C^r)).$
\end{definition}

\begin{example} {\rm Let $\delta (z_j)$ denote the Dirac delta functions and
 \begin{eqnarray}
  H=\left(
    \begin{array}{cc}
    i(|z_1|^2+1)\delta(z_1) dz_1\wedge d\bar{z}_1 & z_1 \delta(z_1)\delta(z_2)dz_1\wedge dz_2\\
    \bar{z}_1 \delta(\bar{z}_1)\delta(\bar{z}_2) d\bar{z}_1\wedge d\bar{z}_2&    i(|z_2|^2+1)\delta(z_2)dz_2\wedge d\bar{z}_2 \\
        \end{array}
    \right)
    \end{eqnarray}
    Then $H\in  D{'}^0_{1,1}(\C^2,\Her^+ (\C^2))$ is a positive $(1,1)$-Hermitian matrix current of zero order. In fact,  for any $\xi=(\xi_1,\xi_2)^t\in {C^0_0}(\C^2,\C^2)$
     $$H(\xi,\xi)=|\xi_1 (0,z_2)|^2 idz_1\wedge d\bar{z}_1+|\xi_2 (z_1,0)|^2 idz_2\wedge d\bar{z}_2$$
     is a continuous  $(1,1)$-form. For any smooth $(1,0)$-form $\al=\al_1 dz_1 +\al_2dz_2$ we know
     $$H(\xi,\xi)\wedge (i\al\wedge \bar{\al})=(|\xi_1 (0,z_2)|^2|\al_2|^2+ |\xi_2 (z_1,0)|^2|\al_1|^2) idz_1\wedge d\bar{z}_1\wedge idz_2\wedge d\bar{z}_2$$
     is a positive measure.}
\end{example}

\section{Existence theorem for $\bar{\partial}$-equations and notions of nef vector bundles}\label{subs3}
 \subsection{Hermitian vector bundles} \label{hvb}

Let $M$ be a compact K\"ahler manifold of dimension $n$ with a K\"ahler form
 $\omega,$ and let $E$ be a holomorphic vector bundle of rank $r$
on $M$ with a smooth Hermitian metric $H$.
Let $(E^*, H^*)$ be the dual vector bundle.
Let $C^{\infty}(M, \pqme)$ be the space of $E$-valued smooth $(p,q)$-forms,
and ${C^k_0} (M, \pqme)$ be the space of $E$-valued smooth $(p,q)$-forms
with compact support.
Let $* : C^{\infty}(M, \pqme) \rw C^{\infty}(M,\Lambda^{n-q,n-p}T^*_M\otimes E)$ be the Hodge star-operator
with respect to $\omega$.
For any $f \in C^{\infty}(M, \pqme)$ and $g \in C^{\infty}(M,\Lambda^{a,b}T^*_M\otimes E)$,
we define $f \wedge H \bar{g} \in C^{\infty}(M,\Lambda^{p+b,q+a}T^*_M\otimes E)$ as following.
We take a local trivialization of $E$ on an open subset $U \subset X$,
and we regard $f = (f_1, \ldots, f_r)^t$ as a row vector
with $(p,q)$-forms $f_j$ on $U$.
The Hermitian metric $H$ is then a matrix valued function
$H = (h_{j\bar{k}})$ on $U$.
We define $f \wedge H \bar{g}$ locally on $U$ by
$$
	f^t\wedge H\bar{ g}=  \sum_{j,k} f_j \wedge h_{j\bar{k}} \bar{g}_k
	\in C^{\infty}(M,\Lambda^{p+b,q+a}T^*_M).
$$
In this manner, we can define anti-linear isomorphisms
$$\sharp_H : C^{\infty}(M,\Lambda^{p,q}T^*_M\otimes E) \rw C^{\infty}(M,\Lambda^{q,p}T^*_M\otimes E^*)$$ by
$\sharp_H u = H\bar{u}$, and
\begin{equation}\bar *_H = \sharp_H \circ * : C^{\infty}(M,\Lambda^{p,q}T^*_M\otimes E)\rw C^{\infty}(M,\Lambda^{n-p,n-q}T^*_M\otimes E^*)\label{hodgst1}\end{equation}
 by
$\bar *_H f = H\overline{*f}$.
There is a point-wise inner product on $C^{\infty}(M,\Lambda^{p,q}T^*_M\otimes E)$  defined by $\langle f,g\rangle =f\wedge \bar*_H g$ and
$|u|^2_H=\langle f,f\rangle.$ It induces an inner product and a  norm on $C^{\infty}(M,\Lambda^{p,q}T^*_M\otimes E)$  defined by
\begin{equation}(f, g)_H = \int_M \langle f,g\rangle ,\hskip 1cm~~~~~~~\Vert f\Vert^2_H = \int_M |f|^2_H.\label{hodgst2}\end{equation}
If the metric $\omega$ is complete we get a Hilbert space $L^2(M,\pqme),$ containing the space $C^{\infty}_0 (M,\pqme)$ as a dense subspace.  Denote by $D = D' + D''$ the Chern connection with $D''=d''.$ The Chern curvature form
\begin{equation}\Theta_H =id''(H^{-1}d'H) =id'(d''H^{-1})H)\label{curvature}\end{equation} is  smooth Hermitian matrix from.
The Hermitian vector bundle $(E, H)$ is said to be
{\it Nakano semi-positive} (resp.\ {\it Nakano positive}),
if the $\End (E)$-valued real $(1,1)$-from $i\Theta_H$
is positive semi-definite (resp.\ positive definite) quadratic form
on each fiber of the vector bundle $T_M \otimes E$.

We define ${D''}^* : C^{\infty}(M,\Lambda^{p,q}T^*_M\otimes E)\rw  C^{\infty}(M,\Lambda^{p,q-1}T^*_M\otimes E)$  by
${D''}^* = - * D' * = - \bar *_{H^*}D'' \bar *_H$, which is
the formal adjoint operator of $D'' : C^{\infty}(M,\Lambda^{p,q}T^*_M\otimes E) \rw C^{\infty}(M,\Lambda^{p,q+1}T^*_M\otimes E)$
with respect to the inner product $(\ ,\ )$. In the same way we denote ${D'}^*$ the formal adjoint operator of $D'$ with respect to the inner product $(\ ,\ )$.
We denote by $e(\al)$ the left exterior product acting on
$C^{\infty}(M,\Lambda^{p,q}T^*_M\otimes E)$ by a form $\al \in C^{\infty}(M,\Lambda^{a,b}T^*_M)$.
Then the adjoint operator $e(\al)^*$ with respect to
the inner product $(\ ,\ )_H$ is defined by
$e(\al)^* = (-1)^{(p+q)(a+b+1)}* e(\bar{\al})*$. For instance we set $\Lambda = e(\omega)^*$. In the following write $e(\al)$ and $e(\al)^*$ simply as $\al$ and $\al^*$ for brevity if without special explanations.

\begin{proposition}\label{weight}
If we change the metric $H$ by multiplying a positive function $\eta,$  the new adjoint operators of $\partial$  and $\bar{\partial}$ with respect to the new metric $H\eta,$ denoted by ${D'}^*_{\eta}$ and $\bar{\partial}^*_{\eta},$ are related to  ${D'}^*$ and ${D''}^*$ in the following way:
$${D'}^*_{\eta} ={D'}^*-\frac{1}{\eta} (d' \eta)^*,\hskip 1cm {D''}^*_{\eta} ={D''}^*-\frac{1}{\eta}(d''\eta)^*.$$
\end{proposition}
\begin{proof}
For $f\in C^{\infty}(M,\Lambda^{p,q}T^*_M\otimes E),$
$${D'}^*_{\eta}f=- \bar *_{{H\eta}^*}
D'( \bar *_{H\eta} f)=- (\frac{1}{\eta}\bar *_{{H}^*}){D'}( \eta\bar *_{H} f)=\bar *_{H^*}{D'} \bar *_H f-\frac{1}{\eta}\bar *_{H^*}d'\eta \bar *_H,$$
hence we get the first equality  and the second follows the same way.
\end{proof}

\subsection{Twisted Bochner-Kodaira-Nakano identity}

The following Bochner-Kodaira-Nakano identity is basic in  studying  vanishing theorem and $L^2$-estimates in K\"ahler geometry, we may find its proof in \cite{siu01},\cite{siu11},\cite{debo}. 
\begin{align}
\Vert D'' f\Vert_H ^2 +\Vert {D''}^* f\Vert_H^2=\Vert D'f \Vert_H ^2 +\Vert {D'}^* f \Vert_H^2  +( [ \Theta_H,\Lambda] f,  f)_H.\end{align}
Ohsawa and Takegoshi  obtained a twisted form of it in \cite{ot}, where it plays a critical role in establishing the extension theorem with their name.  In the following we will give a short proof of the twisted Bochner-Kodaira-Nakano identity from the point of view of 
deforming the Hermitian metric. 

\begin{proposition}\label{kno} Let  $M$ be a K\"ahler manifold  with a complete  K\"ahler metric $\omega$ and $E$ be a Hermitian holomorphic vector bundle equipped with smooth Hermitian metric $H.$ Let $\eta,\lambda$ be smooth positive function on $M.$ Then for any $f\in\lnqe,$
\begin{align}
\Vert (\eta +\lambda^{-1})^{\frac{1}{2}}D''^* f\Vert_H ^2 +\Vert\eta^{\frac{1}{2}}D''f\Vert_H^2 \geq ( [ \eta\sta -id'd''\eta-i\lmd d'\eta\wedge d''\eta,\Lm] f,  f)_H.\label{kodait} \end{align}
\end{proposition}

\begin{proof} Since $\dnqie$ is dense in $\lnqe,$ it suffices to prove (\ref{kodait}) for $u\in\dnqie.$ Change the metric by $H\eta$  and note $D'f=0$ when $f\in\dnqie$ we get  from the  Bochner-Kodaira identity the following
\begin{align}
\Vert D''f\Vert_{H\eta}^2 +\Vert(D''^*_{\eta}f\Vert_{H\eta}^2=\Vert(D'_{\eta})^*f\Vert_{H\eta}^2 +([\Theta_{H\eta},\Lambda]f,f)_{H\eta}.\label{cur1}
\end{align}
By Proposition \ref{weight},
 \begin{align}
\Vert D''^*_{\eta}f\Vert_{H\eta}^2 =\Vert\eta^{\frac{1}{2}}D''^* f\Vert_{H}^2 +\Vert \eta^{-\frac{1}{2}}(d''{\eta})^* f\Vert_{H}^2
-2{\mathfrak{Re}} (D''^* f, (d''{\eta})^* f)_H.\label{cur2}
\end{align}
Since $i[d''{\eta},\Lambda]=(d'\eta)^*$ and  $i[d'{\eta},\Lambda]=(d''\eta)^*,$ we have $(d''\eta )(d''{\eta})^*-(d'\eta)^*(d'{\eta})=i(d''\eta )[d'{\eta},\Lambda]+i[d''{\eta},\Lambda](d'{\eta})=[id'\eta\wedge d''\eta,\Lambda]$ and hence for $u\in\lnqe,$
 \begin{align}
\Vert (d''{\eta})^* f\Vert_{H}^2 =([id'\eta\wedge d''\eta,\Lambda]f,f)_H.\label{cur3}
\end{align}
Since $\Theta_{H\eta}=d''(\eta^{-1}H^{-1}d'(\eta H))=\Theta_H +i\eta^{-2}d'\eta\wedge d''\eta-i\eta^{-1}d'd''\eta$  we have
\begin{align}
-\Vert \eta^{-\frac{1}{2}}(d''{\eta})^* f\Vert_{H}^2+([\Theta_{H\eta},\Lambda]f,f)_{H\eta} =([\eta \Theta_H -id'd''\eta,\Lambda]f,f)_H.\label{cur4}
\end{align}
By (\ref{cur1})-(\ref{cur4}) we have

\begin{eqnarray}\begin{array}{rcl}
\Vert\eta^{\frac{1}{2}}D'' f\Vert_{H}^2 +\Vert\eta^{\frac{1}{2}}D''^* f\Vert_{H}^2&=&([\eta\Theta_H -id'd''\eta,\Lambda]f,f)_H \\
&&+\Vert D'^* f\Vert_{H\eta}^2
-2{\mathfrak{Re}} (D''^* f, (d''{\eta})^* f)_H.
\end{array}\label{cur5}\end{eqnarray}
By Cauchy Schwartz inequality we have
\begin{align}
-2{\mathfrak{Re}} (D''^* f, (d''{\eta})^* f)_H\geq
-\Vert\lambda^{-\frac{1}{2}}D''^* f\Vert_{H}^2 -\Vert \lambda^{\frac{1}{2}}(d''{\eta})^* f\Vert_{H}^2.\label{cur7}
\end{align}
Note  $\Vert\eta^{\frac{1}{2}}D''^* f\Vert_{H}^2 + \Vert\lambda^{-\frac{1}{2}}D''^* f\Vert_{H}^2 =\Vert (\eta+\lambda^{-1})^{\frac{1}{2}}D''^* f\Vert_{H}^2 $
and by (\ref{cur3}) we have  $ \Vert \lambda^{\frac{1}{2}}(d''{\eta})^* f\Vert_{H}^2 = ([i\lambda d'\eta\wedge d''\eta,\Lambda]f,f)_H.$ From these together
with (\ref{cur5}) and
 (\ref{cur7}) we conclude (\ref{kodait}).
\end{proof}

\subsection{Existence thorem for $\bar{\partial}$-equation}

We will use the  $L^2$ existence theorem for solving $\bar{\partial}$-equations.  Ohsawa and Takegoshi  \cite{ot}  established their extension theorem via solving the twisted form of $\bar{\partial}$-equations, and siu \cite{siu11} gave an elegant explanation  of  by a commutating diagram (page 1780 of \cite{siu11}). Let $Tf={D''}((\eta +\lambda^{-1})^{\frac{1}{2}}f)$ and $Sf={\eta}^{\frac{1}{2}}({D''}f)$ be the twisted $\bar{\partial}$ operators acting on $E$-valued differential form $f.$ Assume for some specific choice of the function $\lambda$ and $\eta$ we have
 \begin{equation} (\eta\sta -id'd''\eta-i\lmd d'\eta\wedge d''\eta,\Lm]f,f )_H \geq \Vert \sqrt{\gamma}({d''}\zeta)^*f\Vert^2_H-\delta\Vert f\Vert^2_H\label{delta}\end{equation}
 on $M,$ where $\zeta$ is a non zero smooth function, $\gamma$ is a positive function and $\delta\geq 0$ is a real number. Throughout this paper  the norms of the elements in $L^2(M,\Lambda^{n,k}T^*_M\otimes E)$ for $0\leq k\leq m$ are denoted by the same signal $||\cdot ||_H.$
  Then by (\ref{kodait}) we have
 $$||T^*f||^2_H+||Sf||^2_H+\delta ||f||^2_H\geq \Vert \sqrt{\gamma}({d''}\zeta)^*f\Vert^2_H.$$
 Assume $g, g_1\in W_1:=\lnqe$ are $(n,q)$ form with $Sg=0,$ and $g={d''}\zeta\wedge g_0 +g_1$ with $g_0\in W_2:=L^2(M,\Lambda^{n,q-1}T^*_M\otimes E).$  We want solve the equation $Tf=g.$ 
 In this paper we need to solve the equation in following two cases. 
 
 \subsubsection{\bf Case 1}  In this case $g_1=0$ and $\delta=0.$ For any $u\in W_1,$ write $u=u_1 +u_2$ with ${D''}u_1 =0$ and $u_2$ is in the orthogonal complement space of the kernel of ${D''}.$ Then $(g,u)_H=(g,u_1)_H$ and we get
  $$(g,u)_H =(({d''}\zeta)^*u,g_0)_H \leq ||T^*u||^2_H \cdot ||\frac{1}{\sqrt{\gamma}}g_0||^2_H,$$ 
 which means the functional $T^*u\rw \langle g,u\rangle$ defined on the Hilbert space $W_2$ is continuous, and hence by Riez Representation Theorem 
 the equation ${D''}\tilde{f}=g$ with $\tilde{f}=  (\eta +\lambda^{-1})^{\frac{1}{2}}f$ has a solution with estimate 
  $$\int_M (\eta +\lambda^{-1})^{-1}|\tilde{f}|^2_H\leq 2 \int_M  \frac{1}{\sqrt{\gamma}}|g_0|^2_H.$$

 \subsubsection{\bf Case 2}  The second case we will need is to solve the approximate equation by taking  $\delta \rw 0.$ So we assume $\delta>0.$ Furthermore  $g_1$ is not necessary assumed to be zero. In this case
 $$(g,u)_H =(({d''}\zeta)^*u,g_0)_H +(g_1,u)_H,$$
 by Cauchy-Schwartz inequality we have 
 $$\frac{1}{2}|(g,u)_H |^2\leq ||\sqrt{\gamma}({d''}\zeta)^*u||^2_H \cdot ||\frac{1}{\sqrt{\gamma}}g_0||^2_H +\delta ||u||^2_H \cdot \frac{1}{\delta}||g_1||^2_H.$$
Again we use the trick to write $u=u_1 +u_2\in W_1$ with ${D''}u_1 =0$ and $u_2$ is in the orthogonal complement space of the kernel of ${D''}.$ Then  we get
 $$\frac{1}{2}|(g,u)_H |^2\leq C(g,\gamma,\delta) (||T^* u||^2_H +\delta ||u||^2_H)$$
  with $C(g,\gamma,\delta)=||\frac{1}{\sqrt{\gamma}}g_0||^2_H +\frac{1}{\delta}||g_1||^2_H.$
  Define a function on the Hilbert space sum $W_1\oplus W_2$ by
  $$(T^*u,\sqrt{\delta}u)\rw \langle g,u\rangle,$$
  then the inequality we got means it is a continuous functional, hence using Riez Representation Theorem we can solve the approximate $\bar{\partial}$ equation
     $$D''\tilde{f}+\sqrt{\delta}h=\bar{\partial}\zeta\wedge g_0 +g_1$$ together with estimate
   $$\int_M (\eta +\lambda^{-1})^{-1}|\tilde{f}|^2_H +\int_M  |h|_H^2 \leq 2 \Big(\int_M  \frac{1}{\sqrt{\gamma}}|g_0|^2_H +\frac{1}{\delta}\int_M |g_1|^2_H\Big).$$ 
 As a summery we have the following existence theorem:

\begin{proposition}\label{exist}
 Let $M$ as in Proposition \ref{kno}, assume  (\ref{delta}) holds for any $f\in\lnqe.$ Then {\rm (i)}. If $\delta=0,$ then  for $g={d''}\zeta\wedge g_0 \in\lnqe$ such that $D''g=0$ and  $\int_M  \frac{1}{\sqrt{\gamma}}|g_0|^2_H<+\infty,$ there exits $f \in L^2 (M,\bigwedge^{n,q-1}T^*_M\otimes E)$ such that $D''f=g$ and
 $$\int_M (\eta +\lambda^{-1})^{-1}|{f}|^2_H\leq 2 \int_M  \frac{1}{\sqrt{\gamma}}|g_0|^2_H.$$ 
 {\rm (ii)}. If $\delta>0$, then for any $g={d''}\zeta\wedge g_0 +g_1\in\lnqe$ such that $D''g=0$ and  $\int_M  \frac{1}{\sqrt{\gamma}}|g_0|^2_H +\frac{1}{\delta}\int_M |g_1|^2_H <+\infty,$ there exits $f \in L^2 (M,\bigwedge^{n,q-1}T^*_M\otimes E)$  and  $h\in L^2 (M,\bigwedge^{n,q}T^*_M\otimes E)$
 such that $D''f+\sqrt{\delta}h=\bar{\partial}\zeta\wedge g_0 +g_1$ and
$$\int_M (\eta +\lambda^{-1})^{-1}|{f}|^2_H +\int_M  |h|_H^2 \leq 2 \Big(\int_M  \frac{1}{\sqrt{\gamma}}|g_0|^2_H +\frac{1}{\delta}\int_M |g_1|^2_H\Big).$$ \end{proposition}

\subsection{Singular Hermitian metric on holomorphic vector bundle}

Let $M$  be a complex manifold and $E$ a holomorphic vector bundle of rank $r$ on $M.$ For any $U\in D{'}^k_{p,q}(M,\End(E))$ and any   continuous and differentiable of $k$-order  matrix function $V\in C^k (M,\End(E)),$ we define the product
 $VU:{C^k_0}(M,E)\times {C^k_0}(M,E)\rightarrow D'_{p,q}(M)$ by ($VU)(\xi,\eta)=U(\xi,V\eta).$ In the same way we define $(UV)(\xi,\eta)=U(V\xi,\eta).$
 Let $U_k,U\in D{'}^k_{p,q}(M,\End(E))$ be matrix currents. The sequence $U_k$ is called weak*-convergent to $U$ if for any  $\xi,\eta\in {C^k_0}(M,E),$ the $(p,q)$-currents $U_k (\xi,\eta)$ is weak*-convergent to $U(\xi,\eta).$

 In this paper, we restrict to study the following type of singular Hermitian metric on the holomorphic vector bundles:
\begin{definition}\label{singdef} A {\it singular Hermitian metric of type I} on $E$ we means a measurable Hermitian
metric $H$ such that $d'H\in  D{'}^1_{1,0}(M,\End(E))$  and $H^{-1}\in C^0(M,\End(E))$ is a metric with continuous coefficients.
A {\it singular Hermitian metric of type II} on $E$ we means a continuous Hermitian
metric $H$ such that  $H^{-1}$ is measurable and
 $d''H^{-1}\in  D{'}^1_{0,1}(M,\End(E)).$ 
\end{definition}

By formula (\ref{curvature}), if $H$ and $H^{-1}$ are smooth then  $id''(H^{-1}d'H)=id'(d''H^{-1})H).$ 
For singular Hermitian metrics, the Chern connection is not necessary smooth and Chern curvature forms
 $id''(H^{-1}d'H)$  and $id'(d''H^{-1})H)$ are only meaningful in the sense of distribution.  For singular Hermitian metric of type I since $H^{-1}d'H$ is a well defined  matrix current, and  its weak derivative $d''(H^{-1}d'H)$ is also well defined. For singular Hermitian metric of type II the curvature is also well defined in the sense of distribution.
  We call   $\Theta_H=id''(H^{-1}d'H)$ (reap. $\Theta_H=id'(d''H^{-1})H)$) the {\it curvature currents} of the singular Hermitian metric of type I (reap. type II).

 \begin{proposition} The curvature current $\Theta_H $ is a closed Hermitian matrix current.
 \end{proposition}
 \begin{proof} Here we give proof for the singular Hermitian metric of type I.  The type II case is prove in the same way.
 This is a local property,  it suffices to prove the proposition on any small open set $U$ over which $E$ is a trivial bundle.
 For any $\xi\in{C^{\infty}_0}(U,\Lambda^{f,g}T^*_M|_U\otimes \C^n)$ and
 $\eta\in{C^{\infty}_0}(U,\Lambda^{s,t}T^*_M|_U\otimes \C^n)$ with $f+s=g+t=n-1.$ For simplicity we assume $f=g=0$ and the other cases are checked in the same way.
Then we have
 \begin{align}
{\Theta_H} (\xi,\eta)
=-i H^{-1}d'H(d''{\xi},{\eta})+iH^{-1}d'H({\xi},d'{\eta}).\label{herm0}
\end{align}
Using the definition of weak differential of current we rewrite the R.H.S. of (\ref{herm0}) as
\begin{align}
H^{-1}d'H(d''\bar{\xi},{\eta})=-(d'd''{\xi},{\eta})- H(d''{\xi}, d''(H^{-1}{\eta})),\label{herm1}\\
H^{-1}d'H(\bar{\xi},d'{\eta})=-(d'{\xi},d'{\eta})-H({\xi},d''(H^{-1}d'{\eta})),\label{herm2}
\end{align}
where $(d'd''{\xi},{\eta})=\int_M (d'd''{\xi})^t\wedge\bar {\eta}.$  The second term of R.H.S. of (\ref{herm1}) is
\begin{align}
H(d''{\xi},d(H^{-1}{\eta}))=H({d'\xi},{(d'\bar{H}^{-1})\eta})+(d''{\xi},d''{\eta}).\label{herm3}
\end{align}
The second term of R.H.S. of (\ref{herm2}) is
\begin{align}
H({\xi},d''(H^{-1}d'{\eta}))=H({\xi},{(d'{H}^{-1})d''\eta})+({\xi},d''d'{\eta}).\label{herm4}
\end{align}
By (\ref{herm0}) and (\ref{herm1})-(\ref{herm4}) we get
\begin{align}
\overline{\Theta_H} (\xi,\eta)=i\bar{H}(\overline{d'\xi},\overline{(d'{H}^{-1})\eta})-i\bar{H}(\bar{\xi},\overline{(d'{H}^{-1})d''\eta})
\end{align}
Since $H$ is a Hermitian matrix we have
\begin{align}
\overline{\Theta_H} (\xi,\eta)=i{H}({(d'{H}^{-1})\eta},d'\xi)-i{H}({(d'{H}^{-1})d''\eta},\xi),
\end{align}
hence we have
\begin{align}
\overline{\Theta_H} (\xi,\eta)=id''({H}(d'{H}^{-1}))(\eta,\xi).\label{herm5}
\end{align}
In the same way we may check that
\begin{align}
({\Theta_H})^t(\xi,\eta)=id''({H}(d'{H}^{-1}))(\eta,\xi). \label{herm6}
\end{align}
By (\ref{herm5}) and (\ref{herm6})) we  know $\Theta_H$ is closed and
 $\overline{\Theta_H}= (\Theta_H)^t,$  hence $\Theta_H$ is a closed Hermitian matrix current.
 \end{proof}

 \begin{example}\label{ex1}{\rm Assume that $s_1,\cdots,s_N$ are non zero holomorphic sections of the rank $r$ holomorphic vector bundle $E$
   and denote $S=(s_1,\cdots,s_N)$ with $s_j =(s_{1j},\cdots,s_{rj})^t \in E$ for $j=1,\cdots,N.$
    Let $ A=S\bar{S}^t=\sum_{j=1}^N s_j \bar{s}^t_j,$ then

  \begin{eqnarray}
  A=\left(
    \begin{array}{ccccc}
    \sum_{j=1}^N |s_{1j}|^2         &  \sum_{j=1}^N s_{1j}\bar{s}_{2j} &    \cdots   &   \sum_{j=1}^N s_{1j}\bar{s}_{rj}&\\
    \sum_{j=1}^N s_{2j}\bar{s}_{1j} &      \sum_{j=1}^N |s_{2j}|^2     &    \cdots   &   \sum_{j=1}^N s_{2j}\bar{s}_{rj}& \\
    \vdots                          &            \vdots                &    \ddots   &     \vdots &\\
    \sum_{j=1}^N s_{rj}\bar{s}_{1j} & \sum_{j=1}^N s_{rj}\bar{s}_{2j} &     \cdots  &     \sum_{j=1}^N |s_{rj}|^2&\\
    \end{array}
    \right)
    \end{eqnarray}
    is a $r\times r$ Hermitian matrix. We denote the Moore-Penrose pseudoinverse of $A$ by $H.$ Note that when $H$ is a non-degenerate square matrix we have $A=H^{-1}.$
 There is  a natural (possible singular) Hermitian metric on $E$ defined by
$$||\xi||^2 ={\xi}^t H\overline\xi,~~~~{\rm for}~~\xi \in E_z.$$
It is easy to check that this definition does not depend on the local trivialization of $E$ and is well-defined.
Note $A$ is a nonnegative definite matrix and the set of singularity points  of the metric  are the points where $A$ is degenerate,  so it is the set
$Z=\{z\in M| {\rm rank} S(z)<r\},$ which are exactly the points where $s_1,\cdots,s_N$ don't generate the stalk $E_z.$ Clearly $Z$ is an analytic subset of $M$ and hence we say that the Hermitian metric has analytic singularity.

If $(s_1,\cdots,s_N)$  generated the stalk $E_z$ for all $z\in M,$ then $Z=\emptyset.$ If $(s_1,\cdots,s_N)$   generated the stalk $E_z$ at least for one point $z\in M$ (then $(s_1,\cdots,s_N)$  generated the stalks over an open subset of $M$),  we call $(s_1,\cdots,s_N)$  is {\it generically generated.} In this case we must have $N\geq r.$ Let $Gr(N,r) $ be the complex Grassmannian manifold of $r$-planes in complex $N$-dimensional vector space. Then there is a holomorphic embedding $F:M\rw Gr(N,r),p\mapsto [s_1 (p),\cdots,s_N(p)].$  Then the metric defined above is smooth and it is exactly the pull back by $F$ of the Fubiny-Study metric on the universal vector bundle of $Gr(N,r).$ If  moreover $r=1,$ then $||\xi||^2=|\xi|_U^2e^{-\varphi}.$ Here the metric weight $\varphi=\log (|s_1|^2+\cdots +|s_N|^2)$ is  a plurisubharmonic function on $U.$ }
\end{example}

\begin{proposition} If $(s_1,\cdots,s_N)$  is  generically generated then the (possible) singular  Hermitian metric in Example \ref{ex1} has positive curvature current.
\end{proposition}
\begin{proof} On the Grassmannian manifold  $Gr(N,r),$  the universal bundle  $Q$, which is the quotient bundle of the trivial bundle $\C^N\rw Gr(N,r),$ and the quotient metric is given by $h(f)=(f\bar{f}^t)^{-1},$ where $f\in Gr(N,r)$ is a $r\times N$ matrix whose rows give the $r$-plane $f.$ Choose holomorphic transformation such that $f=(I_r,Z).$ In this holomorphic local coordinate we have $h=(I+Z\bar{Z}^t)^{-1}$ and $h(0)=I,dh(0)=0.$ Hence at the origin, the Chern curvature of the universal bundle is $i\Theta_h =id''d'h=idZ\wedge \overline {dZ}^t.$ Since  $Gr(N,r)$ is homogeneous under the unitary group and $h$ is an invariant Hermitian metric, we know $i\Theta_h$ is semi-positive (positive only when $r=1$) at origin and hence is semi-positive everywhere.

 If $(s_1,\cdots,s_N)$    globally generate the stalks of $E,$ then $F$ is smooth
 and hence $i\Theta_H =iF^*(\Theta_{h}) $ is semi-positive and smooth matrix of differential form.  If $(s_1,\cdots,s_N)$   doesn't generate the stalks of $E$
 then $H$ defines a singular Hermitian metric on $E.$
 If $(s_1,\cdots,s_N)$  is merely generically generated we still can define the map $F$  as above,  which is a meromorphic map and not necessarily holomorphic. The curvature current $i\Theta_H$ is the pull back of the curvature $i\Theta_{h}$ of  the Fubini-Study metric $h$ on the universal vector bundle $Q$ of $Gr(N,r).$
  We know $i\Theta_H =iF^*(\Theta_{h}) $ is a positive curvature current.
\end{proof}

\begin{remark}{\rm The line bundle $\wedge^r Q\rw Gr(N,r)$ has an induced metric $h_r(f) =\det (f\bar{f}^t) $ with positive curvature which gives an invariant K\"ahler metric $\omega_{Gr}=id''d'\log \det(I+Z\bar{Z}^t)$ on $Gr(N,r),$ which is exactly the pull-back of the Fubini Study metric via Pl\"ucker embedding
$Gr(N,r)\rw {\Bbb P}(\wedge^r \C^N).$
 Using that $d'\log \det(I+Z\bar{Z}^t)={\rm Tr}((I+Z\bar{Z}^t)^{-1}(dZ)\bar{Z}^t)$ we know
$$\omega_{Gr}=i{\rm Tr}\{\bar{Z}^t(I+Z\bar{Z}^t)^{-1}Z\overline{dZ}^t (I+Z\bar{Z}^t)^{-1}dZ-(I+Z\bar{Z}^t)^{-1}dZ\wedge \overline{dZ}^t\}.$$
Note that $\bar{Z}^t(I+Z\bar{Z}^t)^{-1}Z =I-(I+\bar{Z}^tZ)^{-1},$ we get the expression of $\omega_{Gr}$ as already given in \cite{wong}:
$$\omega_{Gr}=-i{\rm Tr} \{(I+\bar{Z}^t Z)^{-1}\overline{dZ}^t (I+Z\bar{Z}^t)^{-1}dZ\}=i{\rm Tr} \{ (I+Z\bar{Z}^t)^{-1}dZ(I+\bar{Z}^t Z)^{-1}\overline{dZ}^t \}.$$}
\end{remark}

\qed

 Given a Hermitian metric $H$ with measurable coefficients, we may smooth it via using convolution construction. To illustrate it we take the Euclid flat case as an example. Let 
  $H=(h_{\al\bar{\bt}})$ be a Hermitian matrix  of measurable functions, we could get a smooth matrix function $H_{\epsilon}=H\star \rho_{\epsilon}\stackrel{def}= (h_{\al\bar{\bt}}\star \rho_{\epsilon}),$   the convolution of its each entry with a smooth kernel function $\rho_{\epsilon}.$ Here $\rho_{\epsilon}=\frac{1}{\el^{2n}}\rho(\frac{z}{\el})$  and $$\rho(z)=\chi_{\{|z|<1\}}\cdot e^{-\frac{1}{1-|z|^2}}$$ is the standard positive mollifier.

 \begin{proposition}\label{smoot} Let $H\in L^1_{loc}(\C^r,\Her(\C^r)).$ If $H$ is  semi positive definite   then  $H_{\epsilon}$ is also  semi positive definite;
  if  $H$ is strictly positive definite then $H_{\epsilon}$ is strictly positive  definite too.
\end{proposition}
\begin{proof}  Note for any $a\in \C^n,$ the translation transformation $$i_a:C^{\infty}(\C^n)\rightarrow C^{\infty}(\C^n),~ \xi(x)\mapsto \xi(x+a)$$ is an algebra isomorphism of complex functions.
For any $\xi\in C^{\infty}(\C^n),$ we know $$\begin{array}{rcl}
(\xi^t H_{\epsilon}\bar\xi)(x) &=&\sum_{\al,\bt}\Big(\int_{\C^n}h_{\al\bar{\bt}}(x-y)\rho_{\epsilon}(y)dy\Big)\xi_{\al}(x)\bar{\xi}_{\bt}(x)\\
&=&\int_{\C^n}\sum_{\al,\bt}h_{\al\bar{\bt}}(x-y)(i_y\xi_{\al})(x-y)\overline{(i_y{\xi}_{\bt})}(x-y)\rho_{\epsilon}(y)dy.\\
\end{array}$$
 Hence if $H$ is semi positive definite,   then  $H_{\epsilon}$ is also semi positive definite; if  $H$ is strictly positive definite  then $H_{\epsilon}$ is strictly positive definite too.
 \end{proof}
 
Using Proposition \ref{smoot},   we may approximate a measurable Hermitian metric $H$ by smooth Hermitian metric $H_{\el}.$ However we
could not assure that the curvature $\Theta_{H_{\el}}$ converges to the curvature $\Theta_H,$ in some worse cases $\Theta_{H_{\el}}$ may have no any convergent sub sequences, in the topology that we usually use.
Moreover in application the curvature of $H$ usually has some positivity,
but the curvature  $\Theta_{H_{\el}}$ of the approximate metric may have no any positivity  related to that of $\Theta_H.$  In the next section we will introduce some better singular Hermitian metric to 
avoid these difficulties.

\subsection{Nef holomorphic vector bundle}
Let $E$ be a holomorphic vector bundle   on a Hermitian manifold $M.$
Recall that $E$ is called a nef vector bundle (in usual sense) \cite{dps} if the tautological line bundle $\ko_{\Pj (E^*)}(1)$ is a nef line bundle over $\Pj (E^*).$ Paper \cite{dps} gave a beautiful metric description of nefness when $E$ is a  line bundle. The metric description of nefness is very useful in application and generalization of algebraic geometric results.
 However we only know that a Hermitian metric on the tautological line bundle $\ko_{\Pj (E^*)}(1)$ naturally induces a Finsler metric on $E,$ and at present we don't know how to define  better Hermitian metrics on $E$ via using nefness metrics of the line bundle $\ko_{\Pj (E^*)}(1).$ In this paragraph we will define  some stronger nefness concepts, directly using the Hermitian metric  of $E.$

Let $H$ be a smooth Hermitian metric on $E$ and $G$ a smooth Hermitian metric $M.$
Let$(z_1,\cdots,z_n)$ be  local holomorphic coordinates of $M$ and  $\{e_1,\cdots,e_r\}$  a local
 orthogonal frame of $E$ with dual frame $\{e^1,\cdots,e^r\}.$ The Chern curvature form of  a given Chern connection has the following form
     $$\Theta_{H}=i\sum_{j,k,\al,\bt}R^{\bt}_{j\bk\al} dz_j\wedge d\bar{z}_k\otimes e^{\al}\otimes e_{\bt}\in \Gamma(M,\wedge^2T^*_M\otimes End(E)).$$
where $R^{\gamma}_{j\bk\al}=h^{\gamma\bar{\bt}}R_{j\bk\al\bar{\bt}}$ and $R_{j\bk\al\bar{\bt}}=-\frac{\partial^2 h_{\al\bar{\bt}}}{\partial z_j\partial \bar{z}_k}+h^{\gamma\bar{\delta}}\frac{\partial h_{\al\bar{\delta}}}{\partial z_j}\frac{\partial h_{\gamma\bar{\bt}}}{\partial \bar{z}_k}.$ Recall
that $E$ is called {\it Griffiths semipositive} if for any $u=\sum_j u^j\frac{\partial}{\partial z_j}$ and  $v=\sum_{\al}v^{\al}e_{\al},$ we have
   $$\Theta_{H}(u\otimes v,u\otimes v)=\sum_{j,k,\al,\bt} R_{j\bk\al\bar{\bt}} u^j\bar{u}^k v^{\al}\bar{v}^{\bt}\geq 0;$$
   $E$ is called {\it Nakano semipositive} if for any $u=\sum_{j,\al} u^{j\al}\frac{\partial}{\partial z_j}\otimes e_{\al},$ we have
   $$\Theta_{H}(u,u)=\sum_{j,k,\al,\bt} R_{j\bk\al\bar{\bt}} u^{j\al}\bar{u}^{k\bt}\geq 0.$$
  
   We will use the following fact many times in this paper. Let $u$ be a  $E$-valued $(n,q)$-from on $M.$ In holomorphic local coordinate write $u=\sum_{\al}\sum_{|K|=q}u^{\al}_K dz_1\wedge\cdots\wedge dz_n\wedge d\bar{z}^K\otimes e_{\al}$  with $K=\{j_1,\cdots,j_q\}$ and $d\bar{z}^K=d\bar{z}_{j_1}\wedge\cdots\wedge d\bar{z}_{j_q}.$ Then,
   \begin{equation}\langle [\Theta_H,\Lambda]u,u\rangle=\sum_{|S|=q-1}\sum_{j,k,\al,\bt}R_{j\bk\al\bar{\bt}} u^{\al}_{jS}{\bar{u}^{\bt}_{kS}}.\label{nakanoid}
   \end{equation}
   In particular, if $E$ is Nakano semi positive  then for any $u\in K_M\otimes E$ we have $\langle [\Theta_H,\Lambda]u,u\rangle\geq 0.$
   
\begin{definition}\label{nefff} A holomorphic vector bundle $E$ is said to be Griffiths nef, if for any $\epsilon>0$ there exists a smooth Hermitian metric $H_{\el}$ on $E$ such that its curvature satisfying  $\Theta_{H_{\el}}(u\otimes v,u\otimes v)\geq -\el ||u||^2_G||v||^2_{H_{\el}}$ for $u\in T_M$ and $v\in E.$ It is said to be Nakano nef, if for any $\epsilon>0$ there exist a smooth Hermitian metric $H_{\el}$ on $E$ such that its curvature satisfying  $\Theta_{H_{\el}}(u,u)\geq -\el ||u||^2_{G\otimes H_{\el}}$ for any $u\in T_M\otimes E.$\end{definition}

\begin{example}If $E=L_1\oplus\cdots\oplus L_r$ is direct sum of nef  holomorphic line bundles (in usual sense). Let $\varphi_j$ be the locally PSH function which gives the singular Hermitian metric $e^{-\varphi_j}$ on $L_j.$ Then $E$  with diagonal metric $H={\rm diag}(e^{-\varphi_1},\cdots,e^{-\varphi_r})$ is a Griffiths nef vector bundle, and also a Nakano nef vector bundle.
\end{example}

\begin{proposition} If $E$ is Griffiths nef  vector bundle, then it is  nef vector bundle (in usual sense).
\end{proposition}
\begin{proof}
Let $e_1,\cdots,e_r$ be the local holomorphic frame of $E$ and $e^1,\cdots,e^r$ the dual frame on $E^*.$ The corresponding holomorphic coordinate of $E^*$ is denoted by $w=(W_1,\cdots,W_r)$ and the homogeneous coordinate on fibers of $\Pj(E^*)$ is denoted by $[W_1:\cdots:W_r].$ We use $(w_1,\cdots,w_{r-1})$ to denote the holomorphic coordinate on fibers of $\Pj(E^*)$ and $z=(z_1,\cdots,z_{n})$ to denote the holomorphic coordinate on $X.$ Denote $W=(W_1,\cdots,W_r)^t$ as a column vector. The metrics $H_{\el}=(h_{j\bar{k}})$ naturally induce Hermitian metrics $H^L_{\el}$ on $L:=\ko _{\Pj(E^*)}(1)$
defined by $H^{L}_{\el}(z,w)=h_{j\bar{k}}W^j\bar{W}^k$ for $(z,w)\in \mathbb{P}(E^*).$ Fix a $\el,$ choose normal coordinates
$z^1,\cdots,z^n$  such that $H_{\el}(0)=I$ and $dH_{\el}(0)=0.$ Then $\Theta_{H^L_{\el}}
=id'd''\log H(z,w)$ and

$$\begin{array}{rcl}
\Theta_{H^L_{\el}}&=&-\frac{\sum_{j,k}(\partial \bar{\partial }h_{j\bar{k}})W_j \bar{W}_k} {\sum_{j,k}
h_{j\bar{k}}W_j\bar{W}_k}+ \frac{\sum_{j,k} h_{j\bar{k}}d W_j\wedge
d\overline{ W}_k} {\sum_{j,k} h_{j\bar{k}}W_j\bar{W}_k}\\
&&-\frac{
\sum_{j,k}h_{j\bar{k}}d W_j \bar{W}_k\wedge \sum_{j,k}h_{j\bar{k}}W_j
d\overline{ W}_k} {\sum_{j,k}(
h_{j\bar{k}}W_j\bar{W}_k)^2}.
\end{array}
$$
In matrix form, the curvature $$\Theta_{H^L_{\el}}= \left
(\begin{array}{ccc}\big(\frac{\Theta_{H_{\el}}(W,W) }{|W|^2_{H_{\el}}}\big)_{n\times n}&\Big|& 0\\
-\!\!\!-\!\!\!-\!\!\!-\!\!\!-\!\!-& &\!\!\!-\!\!\!-\!\!\!-\!\!\!-\!\!\!-\!\!\!-\!\!\!-\!\!\!-\!\!\!-\!\!\!-\!\!\!-\!\!\!-\!\!-\\
 0 &\Big|&
 \Big(\frac{(1+|w|^2)\delta_{jk}-w_k\bar{w}_j}{(1+|w|^2)^2}\Big)_{(r-1)\times
 (r-1)}
\end{array}
 \right), $$
where $w=(\frac{W_1}{W_j},\cdots,\frac{W_{j-1}}{W_j},
\frac{W_{j+1}}{W_j},\cdots,\frac{W_r}{W_j})$ is the local coordinate
of the open subset $\Omega_j=\{W_j\not=0\}$ of the fiber
$\mathbb{P}(E^*|_z).$ Note that
$\Big(\frac{(1+|w|^2)\delta_{ij}-w_j\bar{w}_i}{(1+|w|^2)^2}\Big)_{(r-1)\times
 (r-1)}$ is a positive definite matrix with eigenvalues
 $1/(1+|w|^2)$ of order $r-2$ and eigenvalue $1/(1+|w|^2)^2$ of
 order $1.$ Since $E$ is a Griffiths nef vector bundle we have $\frac{\Theta_{H_{\el}}(W,W) }{|W|^2_{H_{\el}}}\geq -\el I_{n\times n}.$ So $E$ is a nef vector bundle (in usual sense).
\end{proof}

\begin{proposition}\label{grna} If $E$ is  Griffiths nef, then $E\otimes \det E$  Nakano nef.
\end{proposition}
\begin{proof} We can prove this proposition by using discrete Fourier transformation, a method used by Demailly and Skoda \cite{desk}. Choose local coordinate on $M$ and local orthogonal frame $\{e_1,\cdots,e_r\}$ of $E\otimes \det E$ such that $G=(\delta_{jk})$ and $H_{\el}=\delta_{\al\bt}.$
Then the curvature of  $E\otimes\det(E)$ is expressed by  $R^{E\otimes\det(E)}_{j\bk\al\bar{\bt}}=R_{j\bk\al\bar{\bt}}+\delta_{\al\bt}\sum_{\gamma}R_{j\bk\gamma\bar{\gamma}}$ and
any $u=\sum_{j,\al} u^{j\al}\frac{\partial}{\partial z_j}\otimes e_{\al}$ we have
\begin{align}\sum_{j,k,\al,\bt} R^{E\otimes\det(E)}_{j\bk\al\bar{\bt}} u^{j\al}\bar{u}^{k\bt}=\sum_{j,k,\al,\bt}( R_{j\bk\al\bar{\bt}} u^{j\al}\bar{u}^{k\bt}+ R_{j\bk\al\bar{\al}} u^{j\bt}\bar{u}^{k\bt}). \label{nagrf}\end{align}
Fix a sufficiently large positive integer $N$ and let $S$ be the set of maps $\sigma:\{1,\cdots,r\}\rw\{1,e^{\frac{1}{N}{2{\prod} i}},\cdots,
e^{\frac{N-1}{N} 2{\prod} i}\}.$
$$\hat{u}_{\sigma}=\sum_{j=1}^n\sum_{\al =1}^r u^{j\al}\sigma(\al)\frac{\partial}{\partial z_j}\in T_M,\hskip 0.5cm \hat{v}_{\sigma}=\sum_{\al =1}^r \sigma(\al)e_{\al}\in E\otimes \det E.$$
Then
$$\begin{array}{rcl}
&&\frac{1}{N^r}\sum_{\sigma\in S}\Theta_{H_{\el}}(\hat{u}_{\al}\otimes \hat{v}_{\al},\hat{u}_{\al}\otimes \hat{v}_{\al})\\
&=&\frac{1}{N^r}\sum_{\sigma\in S}\sum_{j,k,\al,\bt}R_{j\bk\al\bar{\bt}}(\sum_{\gamma =1}^r u^{j\gamma}\sigma(\gamma)(\sum_{\delta =1}^r \overline{u^{k\delta}\sigma(\delta)} )\sigma(\al)\overline{\sigma(\bt)}\\
&=&\frac{1}{N^r}\sum_{\sigma\in S}\sum_{j,k,\al,\bt}R_{j\bk\al\bar{\bt}}\sum_{p=1}^{N-1}\sum_{\gamma,\delta=1}^r u^{j\gamma}\overline{u^{k\delta}}\sigma(\gamma)\overline{\sigma(\delta)} \sigma(\al)\overline{\sigma(\bt)}\\
\end{array}$$
Note   \begin{align}
\frac{\sum_{\sigma\in S}\sum_{\gamma,\delta=1}^r u^{j\gamma}\overline{u^{k\delta}}\sigma(\gamma)\overline{\sigma(\delta)}
 \sigma(\al)\overline{\sigma(\bt)}}{N^r}=\left\{
\begin{array}{cc}
\sum_{\gamma=1}^r u^{j\gamma}\overline{u^{k\gamma}}, &{\rm if}~~~~~~\al=\beta,\\
 u^{j\al}\overline{u^{k\bt}},&{\rm if}~~~~~~ \al\not=\beta.\\
 \end{array}
 \right.\label{four}
  \end{align}
  It together with (\ref{nagrf}) imply that
\begin{align}\sum_{j,k,\al,\bt} R^{E\otimes\det(E)}_{j\bk\al\bar{\bt}} u^{j\al}\bar{u}^{k\bt}=\frac{1}{N^r}\sum_{\sigma\in S}\Theta_{H_{\el}}(\hat{u}_{\al}\otimes \hat{v}_{\al},\hat{u}_{\al}\otimes \hat{v}_{\al}). \label{ident}\end{align}
 By the definition Griffiths nef  we  have
 \begin{align}
  \sum_{j,k,\al,\bt} R^{E\otimes\det(E)}_{j\bk\al\bar{\bt}} u^{j\al}\bar{u}^{k\bt}
  \geq -\el \frac{1}{N^r}\sum_{\sigma\in S} ||\hat{u}_{\al}||^2_{G}||\hat{v}_{\al}||_{H_{\el}}.
    \end{align}
  Since
 \begin{align}||\hat{u}_{\al}||^2_{G}||\hat{v}_{\al}||_{H_{\el}}=\sum_{\sigma\in S}\sum_{j,\gamma,\delta,\al}  u^{j\gamma}\overline{u^{j\delta}} \sigma(\gamma)\overline{\sigma(\delta)}\sigma(\al)\overline{\sigma(\al)}.\label{idents} \end{align}
Using (\ref{four}) again, we get from (\ref{ident}) to (\ref{idents}) that
 $$R^{E\otimes\det(E)}_{j\bk\al\bar{\bt}} u^{j\al}\bar{u}^{k\bt}\geq -r\el ||u||^2_{G\otimes H_{\el}},$$
 hence $E\otimes\det(E)$ is Nakano nef.
\end{proof}
\begin{definition} The curvature current $\Theta_E\in D{'}^0_{1,1}(M,\Her(E))$ of a singular Hermitian metric is called Griffith pseudoeffective
if $\Theta_E$ is a positive Hermitian current in the sense of Definition \ref{matrix}, i.e., if for any smooth local section $u$ of $E$ with compact support, $\Theta_E(u,u)$ is a positive $(1,1)$-current. It is called Nakano pseudoeffective
if $\Theta_E (v,v)$ is a positive distribution for any smooth local section $v=\sum_j u^j\frac{\partial}{\partial z_j}$ of $E\otimes T_X$ with compact support, where $u^j$ are local sections of $E.$
\end{definition}

Note if $E$ is Nakano pseudoeffective then it is  Griffith pseudoeffective. In fact write $\Theta_H =\sum_{j,k}T^E_{j\bk}dz_j\wedge d\bar{z}_{\fk},$
where $T^E_{j\bk}$ are generalized matrix function. then
\begin{align}\Theta_E(u,u)=\sum_{j,k}iT^E_{j\bk}(u,u)dz_j\wedge d\bar{z}_{\fk},\label{pseu1} \end{align}
 take $v=u\otimes w_j\frac{\partial}{\partial z_j}$ with $u$ a smooth local section of $E.$ Then
\begin{align}\Theta_E(v,v)=\sum_{j,k}iT^E_{j\bk}(u,u)w_j\bar{w}_k.\label{pseu2}\end{align}
 Clearly if the R.H.S. of (\ref{pseu2}) is positive distribution for any $u\in E$ and $\sum w_j\frac{\partial}{\partial z_j}\in T_X$ then the R.H.S. of (\ref{pseu1}) is a positive $(1,1)$-current  for any $u\in E$.

For a Hermitian metric $H$ on a holomorphic vector bundle $E$ of rank $r$ over a compact complex manifold $M$ with measurable entries $h_{\al\bar{\bt}},$
we have the following norms for $1 \leq p< \infty:$ for any contractible coordinate open subset $\Omega\subset M,$ choose local trivialization of $E$ on $\Omega$ and write the entries $h_{\al\bar{\bt}}$ as measurable function, and set
   $$||H||_{L^p  (\Omega)}=\max_{1\leq \al,\bt\leq r} ||h_{\al\bar{\bt}}||_{L^p  (\Omega)}.$$
   We  call  the metric $H$ is in $L^p (M,\Her^+ (E))$ if $||H||_{L^p  (M)}=\sum_{\al} ||\varphi_{\al}H||_{L^p  (\Omega_{\al})}<+\infty,$ where $\{\varphi_{\al}\}$ is a smooth partition of  unity of a coordinate open covering $\{\Omega_{\al}\}$ of $M.$  We say the metric $H$ is in $W^{k,p} (M,\Her^+ (E))$ if $H\in L^p (M,\Her^+ (E))$ and $\partial^{\gamma} H\in L^p (M,\Her^+ (E))$ for $0\leq |\gamma |\leq k.$ Here $\partial^{\gamma} H=(\partial^{\gamma} h_{\al\bar{\bt}})$ and $\partial^{\gamma} h_{\al\bar{\bt}}$ are weak derivatives of $ h_{\al\bar{\bt}}.$

\begin{proposition}\label{knab} Let  $M$ be a K\"ahler manifold  with a complete  K\"ahler metric $\omega$ and $E$ be a Hermitian holomorphic vector bundle equipped with singular Hermitian metric
 $H$ of type I. Suppose there is a family of smooth Hermitian metric $H_{\el}$ such that $\lim_{\el\rw 0}||H^{-1}_{\el}-H^{-1}||_{C^0}=0$
and  $\lim_{\el\rw 0}||H_{\el}-H||_{W^{1,2}}=0.$
Let $\eta,\lambda$ be smooth positive function on $M.$ Then for any $f\in W^{1,2}_0(M,\Lambda^{n,q}T^*_M\otimes E),$
\begin{align}
\Vert (\eta +\lambda^{-1})^{\frac{1}{2}}D''^* f\Vert_H ^2 +\Vert\eta^{\frac{1}{2}}D''f\Vert_H^2 \geq ( [ \eta\sta -id'd''\eta-i\lmd d'\eta\wedge d''\eta,\Lm] f,  f)_H.\label{kodt} \end{align}
\end{proposition}

\begin{proof} We may assume $f\in C^{\infty}_0(M,\Lambda^{n,q}T^*_M\otimes E)$ since $C^{\infty}_0(M,\Lambda^{n,q}T^*_M\otimes E)$ is dense in $ W^{1,2}_0(M,\Lambda^{n,q}T^*_M\otimes E),$
Using Proposition \ref{kno} for the smooth metric $H_{\el}$ we get
$$\Vert (\eta +\lambda^{-1})^{\frac{1}{2}}D''^* f\Vert_{H_{\el}} ^2 +\Vert\eta^{\frac{1}{2}}D''f\Vert_{H_{\el}}^2 \geq ( [ \eta\Theta_{H_{\el}} -id'd''\eta-i\lmd d'\eta\wedge d''\eta,\Lm] f,  f)_{H_{\el}}.$$
Clearly for any $f\in W^{1,2}(M,\Lambda^{n,q}T^*_M\otimes E),$
$$\begin{array}{rcl}
\lim_{\el\rw 0}\Vert (\eta +\lambda^{-1})^{\frac{1}{2}}D''^* f\Vert_{H_{\el}} ^2&=&\Vert (\eta +\lambda^{-1})^{\frac{1}{2}}D''^* f\Vert_H ^2,\\
\lim_{\el\rw 0}\Vert\eta^{\frac{1}{2}}D''f\Vert_{H_{\el}}^2 &=&\Vert\eta^{\frac{1}{2}}D''f\Vert_{H}^2.\\
\end{array}$$
Since $\eta$ and $\lambda$ are smooth, we also have
$$\lim_{\el\rw 0}( [ -id'd''\eta-i\lmd d'\eta\wedge d''\eta,\Lm] f,  f)_{H_{\el}} =( [  -id'd''\eta-i\lmd d'\eta\wedge d''\eta,\Lm] f,  f)_{H}.$$
Hence to obtain (\ref{kodt}) it suffices to prove that
  \begin{equation}\lim_{\el\rw 0}( [ \eta\Theta_{H_{\el}} ,\Lm] f,  f)_{H_{\el}}=( [ \eta\sta ,\Lm] f,  f)_{H}.\label{weaklim}\end{equation}
  Since $f$ is $E$-valued $(n,q)$-form we have $[ \Theta_{H_{\el}} ,\Lm] f=e(\Theta_{H_{\el}})e(\omega)^* f,$ this together the definition of the inner product
  in (\ref{hodgst2})  imply that
  $$( [ \eta\Theta_{H_{\el}} ,\Lm] f,  f)_{H_{\el}}=\int_M \eta e(\Theta_{H_{\el}})e(\omega)^* f\wedge  H_{\el}\overline{* f}.$$
  In the same way we rewrite R.H.S. of (\ref{weaklim}), and taking difference of both sides of (\ref{weaklim}), we get
$$\begin{array}{rcl}
&& |( [ \eta\Theta_{H_{\el}} ,\Lm] f,  f)_{H_{\el}}-( [ \eta\Theta_{H} ,\Lm] f,  f)_{H}|\\
&=&\underbrace{\Big |\int_M \eta e(\Theta_{H_{\el}})e(\omega)^* f\wedge  (H_{\el}-H)\overline{* f}\Big |}_{\underset{\rm def}{=}{I}_1(\el)}+\underbrace{\Big |\int_M \eta e(\Theta_{H_{\el}}-\Theta_{H})e(\omega)^* f\wedge  H \overline{* f}\Big |}_{\underset{\rm def}{=}{I}_1(\el)}.\\
\end{array}$$
In the following we will prove the limits of both terms $I_1(\el)$ and $I_2(\el)$ are zeros when $\el\rw 0.$

 In the term $I_1(\el),$ after  taking  out maximum of $\eta,$ then integrating by part, we know it is up bounded by the following three terms:
$$\begin{array}{rcl}
 I_1(\el)&\leq & \underset{M}{\max} |\eta|\Big{\{}\underbrace{\Big |\int_M  H^{-1}_{\el}e(dH_{\el}) d(e(\omega)^* f)\wedge  (H_{\el}-H)\overline{* f}\Big |}_{\underset{\rm def}{=}\hat{I}_1(\el)}\\
&&+\underbrace{\Big |\int_M  H^{-1}_{\el}e(dH_{\el}) e(\omega)^* f\wedge  (d(H_{\el}-H))\overline{* f}\Big |}_{\underset{\rm def}{=}\hat{I}_2(\el)}\\
&&+\underbrace{\Big |\int_M  H^{-1}_{\el}e(dH_{\el}) e(\omega)^* f\wedge  (H_{\el}-H)d(\overline{* f})\Big |}_{\underset{\rm def}{=}\hat{I}_3(\el)}\Big{\}}.\\
\end{array}$$
We claim that
\begin{align}
 &&\hat{I}_1(\el) \leq M_{\omega}\sup_{\el}||H^{-1}_{\el}||_{C^0} (||H_{\el}||_{W^{1,2}} +||f||_{W^{1,2}}) ||H_{\el}-H||_{L^2},\label{ineq1}\\
 &&\hat{I}_2(\el)\leq M_{\omega}\sup_{\el}||H^{-1}_{\el}||_{C^0} (||H_{\el}||_{W^{1,2}} +||f||_{L^2}) ||H_{\el}-H||_{W^{1,2}},\label{ineq2}\\
 &&\hat{I}_3(\el)\leq M_{\omega}\sup_{\el}||H^{-1}_{\el}||_{C^0} (||H_{\el}||_{W^{1,2}} +||f||_{W^{1,2}})  ||H_{\el}-H||_{L^2},\label{ineq3}
\end{align}
where $M_{\omega}$ is a positive constant depending only on the K\"ahler metric $\omega.$
Here we only give a proof of (\ref{ineq1}) since (\ref{ineq2}) and (\ref{ineq3}) follow in the same way.
Write $$f=\sum_{|J|=q}\sum_{\al=1}^r (f_{J\al} dz_1\wedge \cdots \wedge dz_n \wedge d\bar{z}_J)e_{\al}$$
By the definition of Hodge star operator in (\ref{hodgst1}) we have
$$\overline{*f}=\sum_{|J|=q,|I\cup J|=n}\sum_{\al=1}^r a_J(\overline{f_{J\al}} d\bar{z}_I)e_{\al},$$
where $a_J$ is a function depending only on the K\"ahler metric $\omega$ of the K\"ahler manifold $M.$
$$\begin{array}{rcl}
&&e(dH_{\el}) d(e(\omega)^* f)\\
=&&\sum_{|L|=q}\sum_{\al=1}^r \Big(\sum_{j,k=1}^n(b^{jk}_L\frac{\partial H_{\el}}{\partial {z}_j}\frac{\partial f_{L\al}}{\partial \bar{z}_k}+c^{jk}_L\frac{\partial H_{\el}}{\partial \bar{z}_j}\frac{\partial f_{L\al}}{\partial {z}_k} ) dz_1\wedge \cdots \wedge dz_n \wedge d\bar{z}_L\Big)e_{\al},
\end{array}$$
where $b^{jk}_L$ and $c^{jk}_L$ are functions depending only $\omega.$
$$ \begin{array}{rcl}
&&H^{-1}_{\el}e(dH_{\el}) d(e(\omega)^* f)\wedge  (H_{\el}-H)\overline{* f}\\
=&&\sum_{|J|=q}\sum_{j\in J}\sum_{\al,\bt,\gamma,\delta}(d^{jk}_J\frac{\partial (H_{\el})_{\bt\bar{\gamma}}}{\partial {z}_j}\frac{\partial f_{J\gamma}}{\partial \bar{z}_k}+e^{jk}_J\frac{\partial (H_{\el})_{\bt\bar{\gamma}}}{\partial \bar{z}_j}\frac{\partial f_{J\gamma}}{\partial {z}_k} ) (H_{\el}-H)_{\gamma\bar{\delta}}\overline{f_{J\delta}} dV
\end{array}$$
where  $d^{jk}_J$ and $e^{jk}_J$ are functions depending only $\omega$ and $dV={\prod}_{j=1}^n(\frac{i}{2}dz_j\wedge d\bar{z}_j).$
Now using the Cauchy-Schwartz inequality we have

$$ \begin{array}{rcl}
I_1(\el)&\leq &M_{\omega}||H^{-1}_{\el}||_{C^0}\sum_{|J|=q,j\in J}\sum_{\al,\bt,\gamma}\Big(\int_M\sum_{\gamma,\delta}|(H_{\el}-H)_{\gamma\bar{\delta}}|^2)dV\big)^{\frac{1}{2}}\\
&&\cdot \Big(\int_M(|\frac{\partial (H_{\el})_{\bt\bar{\gamma}}}{\partial \bar{z}_j}|^2+|\frac{\partial f_{J\gamma}}{\partial \bar{z}_j}|^2+|\overline{f_{J\delta}}|^2)dV\Big)^{\frac{1}{2}},
\end{array}$$
from this inequality it is easy to conclude (\ref{ineq1}).

Now form (\ref{ineq1}) to (\ref{ineq3}) we know
$$I_1(\el)\leq 3M_{\omega}||H^{-1}_{\el}||_{C^0} (||H_{\el}||_{W^{1,2}} +||f||_{W^{1,2}}) ||H_{\el}-H||_{W^{1,2}},$$
hence we have $\lim_{\el\rw 0}I_1(\el)=0. $

In the term $I_2(\el)$ write
$$\Theta_{H_{\el}}-\Theta_H =id''(H^{-1}_{\el}d'(H_{\el}-H))+id''((H_{\el}^{-1}-H^{-1})d'H),$$
then follow the same way as we done for $I_1(\el),$ we may get
$$I_2(\el)\leq \widetilde{M}_{\omega}[||H^{-1}_{\el}||_{C^0}  ||H_{\el}-H||_{W^{1,2}}+||H^{-1}_{\el}-H||_{C^0} ||H||_{W^{1,2}})(||H||_{W^{1,2}} +||f||_{W^{1,2}})]$$
hence we also have $\lim_{\el\rw 0}I_2(\el)=0. $
\end{proof}
The following version of Proposition \ref{exist} is  a direct conclusion of Proposition \ref{knab}:
\begin{proposition}\label{exist1}
 Let $M$ and $E$ as in Proposition \ref{knab},, assume  (\ref{delta}) is true for any $f\in  W^{1,2}_0(M,\Lambda^{n,q}T^*_M\otimes E).$ Then {\rm (i)}. If $\delta=0,$ then  for $g={d''}\zeta\wedge g_0 \in W^{1,2}_0(M,\Lambda^{n,q}T^*_M\otimes E)$ such that $D''g=0$ and  $\int_M  \frac{1}{\sqrt{\gamma}}|g_0|^2_H<+\infty,$ there exits $f \in W^{1,2}_0(M,\Lambda^{n,q-1}T^*_M\otimes E)$ such that $D''f=g$ and
 $$\int_M (\eta +\lambda^{-1})^{-1}|{f}|^2_H\leq 2 \int_M  \frac{1}{\sqrt{\gamma}}|g_0|^2_H.$$ 
 {\rm (ii)}. If $\delta>0$, then for any $g={d''}\zeta\wedge g_0 +g_1\in W^{1,2}_0(M,\Lambda^{n,q}T^*_M\otimes E)$ such that $D''g=0$ and  $\int_M  \frac{1}{\sqrt{\gamma}}|g_0|^2_H +\frac{1}{\delta}\int_M |g_1|^2_H <+\infty,$ there exits $f \in W^{1,2}_0(M,\Lambda^{n,q-1}T^*_M\otimes E)$  and  $h\in W^{1,2}_0(M,\Lambda^{n,q}T^*_M\otimes E)$
 such that $D''f+\sqrt{\delta}h=\bar{\partial}\zeta\wedge g_0 +g_1$ and
$$\int_M (\eta +\lambda^{-1})^{-1}|{f}|^2_H +\int_M  |h|_H^2 \leq 2 \Big(\int_M  \frac{1}{\sqrt{\gamma}}|g_0|^2_H +\frac{1}{\delta}\int_M |g_1|^2_H\Big).$$
 \end{proposition}   
   
The Proposition \ref{exist1} give the  solution of $\bar{\partial}$-equation for  holomorphic vector bundle with singular Hermitian metric of type I (we also may have a corresponding version for singular Hermitian metric of type II, for brevity  we omit it here), however the assumption  imposed on the Hermitian metric in Proposition \ref{knab} is still very strong. In application 
we need more singular Hermitian metrics.  Before introduce such singular metric we will use  we firstly introduce some notions.
We say that a Hermitian metric $H$ is
    equicontinuous and uniformly bounded if each entry  $h_{\al\bar{\bt}}$ is equicontinuous and uniformly bounded on any open set $\Omega\subset M.$

\begin{definition}A holomorphic vector bundle $E$ is said to be strong Griffiths ({\rm resp. }Nakano) nef, if it is  Griffiths nef  and the metrics $H_{\el}$ in Definition \ref{matrix} satisfying the following two conditions: (i)There exists a $p>r$  ({\rm resp. } $p>1$) such that $||H^{-1}_{\el}||_{L^p (M)}$ is uniformly bounded  with respect to $\el;$ (ii) $H_{\el}\in C^0 (M, \Her^+ (E))$ are equicontinuous and uniformly bounded  with respect to $\el.$
\end{definition}

\begin{proposition} \label{naeff} If $E$ is a strong Griffiths nef vector bundle on a compact complex manifold $M$, then $E$ is Griffiths pseudoeffective, $E\otimes\det E$
is strong Nakano nef and Nakano pseudoeffective.
\end{proposition}
\begin{proof}  Since $H_{\el}$ are equicontinuous, $H_{\el}$ and $||H^{-1}_{\el}||_{L^p (M)}$   are uniformly bounded, by selecting  subsequences one after another,one by one, on the finite entry positions of $H_{\el}$ and $H^{-1}_{\el}$ and on a given finite covering of contractible coordinate open subsets of $M$ respectively,  we at last get a subsequence  such that

(i).  $(H_{\el_k})_{\al\bar{\bt}}$ are uniformly convergent  to $g_{\al\bar{\bt}}$ for all $1\leq \al,\bt\leq r,$

(ii). $(H^{-1}_{\el_k})_{\al\bar{\bt}}$ are  weakly convergent to  $h_{\al\bar{\bt}}$ in $L^p (M)$   for all $1\leq \al,\bt\leq r.$

  \noindent  Let $G=(g_{\al\bar{\bt}})_{r\times r}$ and $H=(h_{\al\bar{\bt}})_{r\times r}.$ Firstly we will prove $H=G^{-1}$ almost everywhere.
  Note $\det (H_{\el_k})$ is uniformly convergent  to $\det (G)$ and $(H^{-1}_{\el_k})_{\al\bar{\bt}}$ is pointwise convergent to $(G^{-1})_{\al\bar{\emph{}\bt}}=\frac{M_{\al\bar{\bt}}}{\det (g_{\al\bar{\bt}})}$ on $\det (g_{\al\bar{\bt}})\not=0.$  Here $M_{\al\bar{\bt}}=(-1)^{\al+\bt}\det( \widetilde{M}_{\al\bar{\bt}})$ and $\widetilde{M}_{\al\bar{\bt}}$ is the 
   submatrix of $(g_{\al\bar{\bt}})$ by deleting its $\al$-th row and $\bt$-th column. So  $M_{\al\bar{\bt}}$
  is a homogeneous polynomial of $\{g_{\al\bar{\bt}}|1\leq \al,\bt\leq r\}.$ For any relative compact open subset $\Omega\subset \{z\in M|\det (g_{\al\bar{\bt}}(z))\not=0\},$ we have $\lim_{k\rw \infty}||(H^{-1}_{\el_k})_{\al\bar{\bt}}-\frac{M_{\al\bar{\bt}}}{\det (g_{\al\bar{\bt}})}||_{L^p(\Omega)}=0.$
  i.e., $(H^{-1}_{\el_k})_{\al\bar{\bt}}$ is strongly convergent to  $(G^{-1})_{\al\bar{\bt}}$ in $L^p (\Omega),$ so
  $H=G^{-1}$ on  $\{z\in M|\det (g_{\al\bar{\bt}}(z))\not=0\}.$ Note $\{z\in M|h_{\al\bar{\bt}}(z)=\infty\}\subseteq \cup_{k=1}^{\infty}\{z\in M|(H^{-1}_{\el_k})_{\al\bar{\bt}}(z)=\infty\}.$
 The set $ \{z\in M|(H^{-1}_{\el_k})_{\al\bar{\bt}}(z)=\infty\}$ is of measure zero since $H_{\el_k}\in L^p (M,\Her^+ (E)),$  hence $\{z\in M|h_{\al\bar{\bt}}(z)=\infty\}$ is of measure zero too. 
 
    Let $A= \{z\in M|\exists \al,\bt ~~{\rm such~~that}~~h_{\al\bar{\bt}}(z)=\infty\}$ and $B= \{z\in M|\det (g_{\al\bar{\bt}}(z))=0\} .$  We will prove that $A=B$ up to a subset of measure zero.
 For any $z_0$
   such $\det (g_{\al\bar{\bt}}(z_0))\not=0$ there is a small neighborhood $\Omega_{\delta}$ of $z_0$ such that $|\det (g_{\al\bar{\bt}}(z))|\geq \el_0>0$ on the closer $\overline{\Omega_{\delta}},$ hence 
   $\lim_{k\rw \infty}||(H^{-1}_{\el_k})_{\al\bar{\bt}}-\frac{M_{\al\bar{\bt}}}{\det (g_{\al\bar{\bt}})}||_{L^p(\overline{\Omega_{\delta}})}=0.$ Note $\frac{M_{\al\bar{\bt}}}{\det (g_{\al\bar{\bt}})}$ is bounded on $\overline{\Omega_{\delta}},$ hence
   $h_{\al\bar{\bt}}$ is bounded almost everywhere in the neighborhood $\Omega_{\delta}$ of $z_0.$  Therefore
   $A\subseteq B.$  Assume $A\subsetneq B$ and the measure of $B$ is positive,  then there exist small open subset $\Omega_{\delta}$ and a  subset $C\subset B\cap \Omega_{\delta}$  of positive measure,
  such that on the whole subset $C,$ $h_{\al\bar{\bt}}(z)$ is almost everywhere bounded for all $\al,\bt$ but $\det (g_{\al\bar{\bt}}(z))=0.$   Since $(H^{-1}_{\el_k})_{\al\bar{\bt}}$ is weakly convergent to  $h_{\al\bar{\bt}}$ in $L^p (M),$  we have $\int _M ((H^{-1}_{\el_k})_{\al\bar{\bt}}(z)-h_{\al\bar{\bt}}(z))\chi_{\Omega_{\delta}}(z)d\mu \rw 0,$  which means $\int_{\Omega_{\delta}}((H^{-1}_{\el_k})_{\al\bar{\bt}}-h_{\al\bar{\bt}})d\mu \rw 0$ for all $\al,\bt$ as $k\rw \infty.$ But  $(H^{-1}_{\el_k})_{\al\bar{\bt}}$ are uniformly convergent  to $g_{\al\bar{\bt}}$  on $\overline{\Omega}_{\delta}$ for all $1\leq \al,\bt\leq r.$  Since $\det (g_{\al\bar{\bt}})(z)=0$ on $C,$ hence $\det ((H^{-1}_{\el_k})_{\al\bar{\bt}})$ is divergent to $\infty$ on $C$ as $k\rw \infty.$ Therefore there exist $\al_0,\bt_0$ such that $((H^{-1}_{\el_k})_{\al_0\bar{\bt}_0})$ is divergent to $\infty$ on  a positive measure subset of $C\subset \Omega_{\delta},$  and it is a contradiction to $\int_{\Omega_{\delta}}((H^{-1}_{\el_k})_{\al\bar{\bt}}-h_{\al\bar{\bt}})d\mu \rw 0$ for all $\al,\bt$ as $k\rw \infty.$
  Hence $\{z\in M|\det (g_{\al\bar{\bt}}(z))=0\} $ is of measure zero and $H=G^{-1}$ almost everywhere.

  Now we claim that the curvatures $\Theta_{H_{\el_k}}=id'((d''H^{-1}_{\el_k})H_{\el_k})$ are weak$^*$-convergent to the curvature current $\Theta_{H}=id'((d''H^{-1})H).$ For any $u,v\in{C^{\infty}_0}(M,E),$ by formula (\ref{herm0}) we have
 \begin{align}
{\Theta_{H_{\el_k}}} (u,v)
=-i d''H^{-1}_{\el_k}(H_{\el_k}d'u,v)+id'H^{-1}_{\el_k}(H_{\el_k}u,d''v).\label{conver1}
\end{align}
  Since $d''H^{-1}_{\el_k}$ are weakly convergent to $d''H^{-1}$ and $H_{\el_k}v$ and $ H_{\el_k}d'v$ are uniformly convergent to
  $Hv$ and $ Hd'v$ respectively. Hence
     \begin{align}
\lim_{k\rw \infty}{\Theta_{H_{\el_k}}} (u,v)
=-i d''H^{-1}(Hd'u,v)+id'H^{-1}(Hu,d''v).\label{conver2}
\end{align}
  for any $u,v\in{C^{\infty}_0}(M,\C^n).$ So $\Theta_{H_{\el_k}}$ are weak$^*$-convergent to  $\Theta_{H}.$ It suffices to prove $\Theta_{H}$ is positive Hermitian current, but it is clear since $\Theta_{H_{\el_k}}(u\otimes v,u\otimes v)\geq -\el_k ||u||^2 ||v||^2.$

  Let $\widetilde{H_{\el}}=H_{\el} \det (H_{\el})$ be the Hermitian metrics on $E\otimes \det E.$ Then $H_{\el} \det (H_{\el})$ are  equicontinuous and uniformly bounded  with respect to $\el.$ For any $1<q<\frac{p}{r}$ and any contractible coordinate open subset $\Omega\subset M,$ by H\"older's inequality
  $$\begin{array}{rcl}
 && \int_{\Omega} |(H^{-1}_{\el})_{\al\bar{\bt}}  (H^{-1}_{\el})_{\al_1\bar{\bt}_1}\cdots (H^{-1}_{\el})_{\al_r\bar{\bt}_r}|^q d\mu\\
  &\leq&
 {\rm Vol}(\Omega)^{\frac{p-qr}{p}} (\int_{\Omega} |(H^{-1}_{\el})_{\al\bar{\bt}}|^p d\mu)^{\frac{q}{p}} (\int_{\Omega} |(H^{-1}_{\el})_{\al_1\bar{\bt}_1}|^p d\mu)^{\frac{q}{p}}\cdots
  (\int_{\Omega} |(H^{-1}_{\el})_{\al_r\bar{\bt}_r}|^p d\mu)^{\frac{q}{p}}\\
    \end{array}$$
Expand $\det (H^{-1}_{\el})$ as a sum of $r!$ monomials, we get $$||(\widetilde{H^{-1}_{\el}})_{\al\bar{\bt}}||_{L^q(\Omega)}\leq r!{\rm Vol}(\Omega)^{\frac{p-qr}{pq}} ||H^{-1}_{\el}||^{r+1}_{L^p(M)}\leq r!{\rm Vol}(M)^{\frac{p-qr}{pq}} ||H^{-1}_{\el}||^{r+1}_{L^p(M)},$$
hence $||\widetilde{H^{-1}_{\el}}||_{L^q(M)}$ are uniformly bounded with respect to $\el.$ Hence $\widetilde{H}_{\el}$ are uniformly convergent to
$\widetilde{H}=H\det H$ and $\widetilde{H^{-1}_{\el}}$ are weakly  convergent to
$\widetilde{H}^{-1}=H^{-1}(\det H^{-1})$ in $L^q(M).$ Hence $E\otimes \det E$ is strong Nakano nef.
 Furthermore the curvature current $\Theta_{\widetilde{H}}$ of Hermitian metric $\widetilde{H}$ is Nakano pseudoeffective since $E\otimes \det E$ is Nakano nef by Proposition \ref{grna}.
\end{proof}

 Let $E$ be a  strong Grffiths (Nakano) nef holomorphic vector bundle, then  by Proposition \ref{naeff} there is a family smooth Hermitian metric $H_{\el}$ on $E$ which convergent to a
Grifiths (Nakano) pseudo effective  singular Hermitian metric of type II, denoted by $H.$   We call  $H$ the {\it limit Hermitian metric} of the    strong Grffiths (Nakano) nef vector bundle.

\section{$L^2$-Extension theorems }\label{subs4}
Using the twisted Bochner-Kodaira-Nakano inequality,
Ohsawa-Takegoshi established a $L^2$-extension theorem for weighted $L^2$-integrable holomorphic functions on a bounded Stein domain in the 
fundamental paper \cite{ot}. Manivel get a more general $L^2$-extension theorem in the framework of vector bundles  using a more geometric approach in \cite{man}. 
Siu's work in algebraic geometry in \cite{siu98},\cite{siu01},\cite{siu02} gave  very important applications of the  Ohsawa-Takegoshi theorem. 
In the paper \cite{deex},  Demailly emphasized   other applications  Ohsawa-Takegoshi theorem in his own work and gave many inspiring comments on how to deal with small negative curvature term
when constructing $L^2$-extension theorems via
solving  approximating $\bar{\partial}$-equations, which Yi used in \cite{yi} in studying  $L^2$-extension theorems on compact K\"ahler manifolds.
  The lecture paper \cite{var} gave a good introduction to  Ohsawa-Takegoshi extension theorem and other $L^2$-techniques in complex geometry.
 
 In this section, we will consider the extension of holomorphic jet sections for high rank bundles whose supports are discrete points. Popovici  \cite{pop} firstly got a $L^2$-extension theorem for the holomorphic jet sections of line bundles whose supports are connected sub manifolds without  codimension  restriction. In our case the support may not be connected, however, its dimension is zero.  We only consider extending holomorphic jet sections  whose supports are sub schemes  
 of dimension zero since we only need such a case in proving our main theorem in the next section. With some more strength, it would be possible to get  the extension of holomorphic jet sections  whose supports are arbitrary subscheme
 whose underling reduced irreducible components are smooth complex sub manifolds.
Recently   the paper \cite{gz} obtained  improvements of the upper bounds  of the extension theorems founded before, using undetermined function via solving ordinary differential equations. We get our extension theorem via optimizing these techniques. Since we don't  need   an effective  upper bound constant in the next section, the techniques developed in \cite{gz} are used for a different purpose, rather for an optimal upper bound. 
Our extension theorem \ref{exthm} is sharper  than main theorem of \cite{pop} when the support is an isolated point.

 In Subsection \ref{subsub2},  we  establish the $L^2$-extension theorem for 
jet section for strong Nakano nef holomorphic vector bundles on bounded Stein domain by inductive method;  in Subsection \ref{subsub3} we generalize our $L^2$-extension theorem  further to 
any compact K\"ahler manifold. 
We cut the  compact K\"ahler manifold $M$ into small bounded Stein open subsets and to extend  a given jet section to each Stein open subset  in a compatible way, then glue these extended sections to a  global holomorphic section on $M,$  via solving approximate $\bar{\partial}$-equation.

\subsection{Intrinsic norm on jet sections}

Let $\fp=\{p_1,\cdots,p_l\}\subset M $ be a finite set of distinct points, and $E$ a holomorphic vector bundle  on a Stein manifold $M$ with  (maybe singular) Hermitian metric $H.$
A holomorphic section 
$f\in H^0 (M,E\otimes \ko _M /\kj^{k_1+1}_{p_1}\cdots \kj^{k_l+1}_{p_l})$ is a {\it $\fk$-jet section} of $E$ (passing through  $p_1,\cdots,p_l$ with $\fk=(k_1,\cdots,k_l)$).
 We use $J^{\fk}_{\fp} (E)$ to denote the set of all $\fk$-jet sections of $E$ with support $\fp.$ For any holomorphic section $F$ of $E,$ it defines a unique $\fk$-jet via quotient map   
 $E\rw E\otimes \ko _M /\kj^{k_1+1}_{p_1}\cdots \kj^{k_l+1}_{p_l},$ we denote it by $J^{\fk}_{\fp} (F).$ Conversely, for any jet section $f\in J^{\fk}_{\fp} (E),$ if there exits  $F\in H^0(M,E)$ such that $J^{\fk}_{\fp} (F)=f,$ we call $F$ a holomorphic {\it lifting section} of $f.$

Choose holomorphic local  frame $\{e_1,\cdots,e_r\}$ of $E$ on a contractible neighborhood $U_j$ of $p_j$ for $j=1,\cdots,l.$ For any holomorphic section $F$ of $E,$ write
$$F=\sum_{j=1}^r F_{j}e_{j},$$
Denote  
$$\partial^{\al} F(p_j)=\sum_{\al=1}^r \partial^{\al} F_j(p_j)e_{\al}$$
the $\partial$-partial derivative of $\al$-order of $F,$ where $\al=(\al_1,\cdots,\al_r).$ And define 
\begin{equation}|\partial^{\al} F|^2_H(p_j)=\sum_{j\bk} H_{j\bk}\partial^{\al_j} F_j(p_j)\overline{\partial^{\al_k} F_{k}(p_j)},\label{nnm}\end{equation}
where $H_{j\bk}=H(e_j,e_k).$  If we change to another holomorphic local  frame $\{\tilde{e}_1,\cdots,\tilde{e}_r\}$ of $E$
with transformation $e_j=\phi_{jk}\tilde{e}_k,$  then $F=\sum_{j=1}^r \tilde{F}_{j}\tilde{e}_{j}$ 
with $\tilde{F}_j=\sum_k F_k \phi_{kj}.$ Hence under frame $\{\tilde{e}_1,\cdots,\tilde{e}_r\},$ the number $|\partial^{\al} F|^2_H(p_j)$ may take a different value 
\begin{equation}\widetilde{|\partial^{\al} F|^2_H(p_j)}=\sum_{s,t,j,k}( H_{s\bar{t}}\phi^{sj}\overline{\phi^{tk}})\partial^{\al_j} (\sum_n F_n \phi_{nj})\overline{\partial^{\al_k}(\sum_n F_n \phi_{nj})}\Big|_{p_j},\label{nnmm}\end{equation}
where $(\phi^{jk})_{r\times r}$ is the inverse matrix of $(\phi_{jk})_{r\times r}.$
 But the difference will not affect our definition of jet norm up to a positive constant in the following Proposition \ref{intnorm}.

 Since $M$ is Stein there is a surjective map $H^0(M, K_M\otimes E)\rw H^0 (M,K_M \otimes E/\kj^{k_1+1}_{p_1}\cdots \kj^{k_l+1}_{p_l}).$ For any $\fk$-jet $f\in J^{\fk}_{\fp} (E),$ let $F$ be an arbitrary holomorphic lifting to $M.$  Now  for $f\in  J^{\fk}_{\fp} (E)$  define
 $$\Vert f\Vert ^2_{\fk} =\sum_{j=1}^l\sum_{|\al|\leq k_j}\frac{1}{(\al !)^2} |\partial^{\al} F|^2_H (p_j),$$
  $$\widetilde{\Vert f\Vert ^2_{\fk}} =\sum_{j=1}^l\sum_{|\al|\leq k_j}\frac{1}{(\al !)^2}\widetilde{ |\partial^{\al} F|^2_H (p_j)},$$
where $|\al|=\max \{\al_1,\cdots,\al_n\}$ and $\al !=\al_1 !\cdots \al_l !.$ 
By (\ref{nnmm}) it is easy to check that 

\begin{proposition}\label{intnorm} Fix the points $p_1,\cdots,p_l.$ Suppose the Hermitian metric $H$ is non degenerate  and bounded at $p_1,\cdots,p_l,$ in particular if the  metric is continuous, then
$\Vert f\Vert ^2_{\fk}$ is equal to  $\widetilde{\Vert f\Vert ^2_{\fk}}$  up to a finite positive constant which
is decided by the  coordinate transformation  (and their derivatives) of the local frame.  Hence up to a positive  finite constant $||f||_{\fk}$ is a well defined norm.\end{proposition}
\begin{proof}
Clearly $\Vert f\Vert ^2_{\fk}=0$ 
if and only if $\partial^{k}F_j (p_s)=0$ for all $0\leq k\leq k_j$ and $j=1,\cdots,r$ and $s=1,\cdots,l.$ 
 And $\widetilde{\Vert f\Vert ^2_{\fk}}=0$ 
if and only if $\sum_j[\phi^{sj}\partial^k (\sum_n F_n \phi_{nj})](p_t)=0$ for all $0\leq k\leq k_t$ and $s,t=1,\cdots,r.$ 
Since $\{\phi_{jk}\}$ are non degenerate  coordinate transformation   we have 
\begin{equation}\partial^k (\sum_n F_n \phi_{nj})(p_t)=0\label{nmn}\end{equation}
for all $0\leq k\leq k_t$ and $t=1,\cdots,l$
and $j=1,\cdots,r.$ Take $k=0$ we get  $\sum_n (F_n \phi_{nj})(p_t)=0$ for $t=1,\cdots,l$
and $j=1,\cdots,r.$  Hence we have $ F_n (p_t)=0$ for $n=1,\cdots,r$ and $t=0,\cdots,l.$ Then take these into (\ref{nmn}) then let $k=2$ in (\ref{nmn}) we get $\partial^2 F_n (p_t)=0$ for $n=1,\cdots,r $ and $t=1,\cdots,r.$
Go on this way we finally have  $\partial^{k}F_j (p_s)=0$ for all $0\leq k\leq k_j$ and  $j=1,\cdots,r$ and $s=1,\cdots,l.$  Hence ${\Vert f\Vert ^2_{\fk}}=0$  if and only if $\widetilde{\Vert f\Vert ^2_{\fk}}=0.$ 

If the  metric $H$ is bounded,  then $|\partial^{\al} F|^2_H(p_j)$ for all $\al_j\leq k_j$ are bounded if and only if $\widetilde{|\partial^{\al} F|^2_H(p_j)}$ are bounded for all $\al_j\leq k_j,$
 since $\{\phi_{jk}\}$ are smooth functions.  Hence ${\Vert f\Vert ^2_{\fk}}$ is bounded  if and only if $\widetilde{\Vert f\Vert ^2_{\fk}}$ is bounded.
 
 It is  also very easy to check the norms $||f||_{\fk}$ doesn't depend on its lifting $F,$ for brevity we omit the details here.  Hence up to a positive  finite constant $||f||_{\fk}$ is a well defined norm.
\end{proof}

Here we define norm on jet sections not using intrinsic derivatives because their
support is a very simple set  consisting of isolated points,  another benefit is that we could see that $||f||_{\fk}$ doesn't depend on the Chern connection of the Hermitian holomorphic vector bundle $E,$
which is bad behaved if the Hermitian metric $H$ is only continuous. 

\subsection{Extension theorems for jet sections on  bounded Stein domains}\label{subsub2}

Through out this subsection we assume $M$ is a Stein  domain in $\C^n$ with finite diameter.
 Choose holomorphic local coordinate charts of $M,$ and denote the coordinate of $p_1,\cdots,p_l$ by $w_1,\cdots,w_l.$ Let $D$ denote the diameter of $M.$  And  $s_j (z)=\frac{z-w_j}{27D}$ for
   $j=1,\cdots,l$ if without special explanation,  and 
   here $\log (\log \frac{1}{|s_j|})>1.$

\subsubsection{\bf Extension theorems for Nakano  semi positive vector bundles}

\begin{theorem}\label{exthm} Let $M$ be a  bounded Stein domain of diameter $D,$ and  $E$  a  holomorphic vector bundle   on $M,$  equipped with a smooth  Hermitian metric $H$ such that  $\Theta_H $ is Nakano semi-positive. Then for
 any holomorphic jet section $f\in J^{\fk}_{\fp} (K_M \otimes E)$ satisfying $$\Vert f\Vert ^2_{\fk}<\infty,$$ there exists holomorphic section $F\in H^0(M,K_M \otimes E) $  such that $J^{\fk}_{\fp} (F) =f$ and
$$\int_M \frac{|F|^2_H}{{\prod}_{j=1}^l \Big(|s_j|^{2n}\Big(\log \frac{1}{|s_j|}\Big)^{\delta_j}\Big)} \leq C\Vert f\Vert ^2_{\fk},\hskip 3cm  {\rm (A_{\fk})} $$
where $\delta_j >0 $ are  any positive constants and $C$ is a finite positive  constant depending only on $\delta_j$ and the minimal  distance between any two points among $p_1,\cdots,p_l.$
\end{theorem}

\begin{proof}  Firstly note that it suffices to prove 

\begin{lemma}\label{exlem}Under the assumptions in Theorem \ref{exthm}, then for
 any holomorphic jet section $f\in J^{\fk}_{\fp} (K_M \otimes E)$ satisfying $$\Vert f\Vert ^2_{\fk}<\infty,$$ there exists holomorphic section $F\in H^0(M,K_M \otimes E) $  such that $J^{\fk}_{\fp} (F) =f$ and for any $1\leq j\leq l$
 $$\int_M \frac{|F|^2_H}{\Big(|s_1|\cdots |s_l|)^{2n}\Big(\log \frac{1}{|s_j|}\Big)^{\delta_j}} \leq C\Vert f\Vert ^2_{\fk},\hskip 3cm {\rm (B_{\fk,j})}$$
where $\delta_j >0 $ are  any positive constants and $C$ is a finite positive  constant depending only on $\delta_j$ and the minimal  distance between any two points among $p_1,\cdots,p_l.$
\end{lemma}
Assume that Lemma \ref{exlem} is true, then by the inequality of arithmetic and geometric means we have
$$\int_M \frac{|F|^2_H}{{\prod}_{j=1}^l \Big(|s_j|^{2n}\Big(\log \frac{1}{|s_j|}\Big)^{\frac{\delta_j}{n}}\Big)} \leq \frac{1}{l}\sum_{j=1}^l \int_M \frac{|F|^2_H}{\Big(|s_1|\cdots |s_l|)^{2n}\Big(\log \frac{1}{|s_j|}\Big)^{\delta_j}} ,$$
 hence Theorem  \ref{exthm} follows. Note for the inequalities {\rm ($A_{\fk}$)} and {\rm ($B_{\fk,j}$)}  it suffices to prove the case where $0<\delta_j <1 $ since it implies the cases $\delta_j\geq 1$ via using the inequalities $\frac{1}{(\log \frac{1}{|s_j|})^{a}}\leq \frac{1}{(\log \frac{1}{|s_j|})^{b}}$ if $a\geq b> 0.$

 Due to the symmetry, clearly it suffices to establish the inequalities  {\rm ($B_{\fk,j}$)} in  Lemma \ref{exlem} for $j=1.$ We will establish {\rm ($A_{\fk}$)} and {\rm ($B_{\fk,1}$)} altogether  by 
by induction on $|\fk|=k_1 +\cdots+k_l.$ 

\vskip 0.5cm

The case $|\fk|=0$ is essentially a sharpening of a special case (the support of the section to be extended is of dimension zero) of the Demailly's version \cite{deex}  of the   Ohsawa-Takegoshi-Manivel extension theorem \cite{ot},\cite{oh1},\cite{man}.  
Since $M$ is Stein, the restriction map $H^0 (M,K_M \otimes E)\rw H^0 (M,K_M \otimes E/\kj_{p_1}\cdots \kj_{p_l})$ is surjective, for any $0$-jet section $f\in H^0 (M,K_M \otimes E/\kj_{p_1}\cdots \kj_{p_l}),$ we may write $f=f_1 +\cdots+f_n,$  with $f_j\in H^0 (M,K_M \otimes E/\kj_{p_1}\cdots \kj_{p_l})$ satisfying
 \begin{equation} f_j(p_j)=f(p_j),~~~~~~~~~~~~~~~~~{\rm and}~~~f_j(p_k)=0~~~{\rm for} ~~~k\not=j.\label{fid}\end{equation}
It suffices to extend each $f_j$ to a section $\hat{f}_j$ with the required properties and at the same time satisfying $\hat{f}_j(p_k)=0$ for $k\not=j.$ In the following we extend $f_1$ to $\hat{f}_1$ and
 the extension of $f_j$ follows in the same way.

Let $\rho_{\el}$ be the smooth kernel function such that its support is contained in $(-\el,\el)$ and $\int^{\el}_{-\el}\rho_{\el}(t)dt=1.$
Then
 \begin{equation} a_{\tau,\el}(t)=\frac{1}{1-\el}\chi_{\{\tau-(1-\el)/2<t<\tau+(1-\el)/2\}}\star \rho_{\frac{1}{8}\el}\label{a}\end{equation}
 is a smooth approximation of the characteristic function $\chi_{\{\tau-1/2 <t<\tau+1/2 \}}$ as $\el\rw 0.$ Let
 \begin{equation} b_{\tau,\el}(t)=\int_{\tau-1}^t a_{\tau,\el}(x)dx,\hskip 1cm v_{\tau,\el}(t)=\int_{0}^t b_{\tau,\el}(x)dx.\label{bc}\end{equation}
 Throughout this proof we assume that $0<\el <\frac{1}{2}.$  Then $v_{\tau,\el}$ is a family of convex increasing function with $0\leq v_{\tau,\el}'\leq 1$ and $0\leq v_{\tau,\el}''\leq 2.$  Let 

 \begin{equation}  \phi =\log |s_1|, \hskip 1cm
 \eta_{\el} =u(-v_{\tau,\el}\circ \phi).\label{pet}\end{equation}
  Here $u$ is a positive smooth function  which will be decided later. Then
 $$\begin{array}{rcl}
 -id'd''\eta_{\el} &=& u'[b_{\tau,\el}(\phi)id'd''\phi +a_{\tau,\el}(\phi)id'\phi\wedge d''\phi] \\
 &&-u''id'(v_{\tau,\el}(\phi))\wedge d''(v_{\tau,\el}(\phi));\\
 -\lambda_{\el} id'\eta_{\el} \wedge d''\eta_{\el} &=& -\lambda_{\el} u'^2 id'(v_{\tau,\el}(\phi))\wedge d''(v_{\tau,\el}(\phi)).
 \end{array}$$
Let $\lambda_{\el}=-\frac{u''_{\el}}{u'^2_{\el}}$  and assume that
\begin{align}
u>0,\hskip 0.3cm u'>1 \hskip 0.3cm {\rm and} \hskip 0.3cm u''<0. \label{ass}
\end{align}
 Then $$-id'd''\eta_{\el}-\lambda_{\el} id'\eta_{\el} \wedge d''\eta_{\el}=u'[b_{\tau,\el}(\phi)id'd''\phi +a_{\tau,\el}(\phi)id'\phi\wedge d''\phi] $$ and
 $$\eta_{\el}+\lambda_{\el}^{-1}=\Big(u -\frac{u'^2}{u''}\Big)\circ (-v_{\tau,\el}\circ \phi).$$
   Since $[id'\phi\wedge d''\phi,\Lambda]=d''\phi(d''\phi)^*$ is a semi-positive self-adjoint operator by (\ref{cur3}), the operators $[a_{\tau,\el}id'\phi\wedge d''\phi,\Lambda]$ and $[b_{\tau,\el}(\phi)id'd''\phi,\Lambda]$ are semi-positive. Let $$\widetilde{H}_1=He^{-2n\sum_{j=1}^l \log |s_j|}$$ be a new Hermitian metric on $K_M\otimes E,$ then its curvature 
   $$\Theta_{\widetilde{H}_1}=\Theta_{H} +n\sum_{j=1}^l id'd''\log |s_j|^2$$
    is semi-positive.  For any $E$-valued $(n,1)$-form $\vartheta$, by (\ref{nakanoid}), $\langle [\Theta_{\widetilde{H}_1},\Lambda]\vartheta,\vartheta\rangle\geq 0.$  Hence
   $$B_{\el}=[ \eta_{\el} \Theta_{\widetilde{H}_1} -id'd''\eta_{\el}-i\lmd d'\eta_{\el}\wedge d''\eta_{\el},\Lm]$$
   is a  semi-positive self-adjoint operator, and
$$\langle B_{\el}\vartheta,\vartheta\rangle  \geq \langle [a_{\tau,\el}(\phi)id'\phi\wedge d''\phi),\Lm]\vartheta,\vartheta\rangle \geq a_{\tau,\el}(\phi) |(d''\phi)^* \vartheta|^2. $$

 Let $\tilde{f}_1$ be a section of $H^0(M,K_M\otimes E)$ which is holomorphic lifting of $f_1,$ and set
 $$L_{1,\tau,\el}=D''((1-b_{\tau,\el}(\phi))\tilde{f}_1)=(-a_{\tau,\el}(\phi)d''\phi)\tilde{f}_1. $$
Here $1-b_{\tau,\el}(t)$ is a smooth cut off  function,
$1-b_{\tau,\el}(t)=1$ for $t<\tau -\frac{1-\el}{2}$ and $1-b_{\tau,\el}(t)=0$ for $t>\tau +\frac{1-\el}{2}.$ The support of $L_{1,\tau,\el}$ is contained in $\phi^{-1}([\tau-\frac{1-\el}{2},\tau+\frac{1-\el}{2}]).$ Hence $L_{1,\tau,\el}$ has its support
contained in the tubular neighborhood of $p_1$ as $\tau\rightarrow -\infty.$ We assume that $\tau\ll 0$ such that $\{p_2,\cdots,p_l\}\subset M\backslash \phi^{-1}([\tau-\frac{1-\el}{2},\tau+\frac{1-\el}{2}]),$ then
$$L_{1,\tau,\el}(p_2)=\cdots =L_{1,\tau,\el}(p_l)=0.$$
By (i) of Proposition \ref{exist}, there exist $K_{1,\tau,\el}\in L^2(M,K_M\otimes E)$ such that
$$D'' K_{1,\tau,\el}=L_{1,\tau,\el},$$ and

$$ \int_{M}(\eta_{\el} +\lambda_{\el}^{-1})^{-1}\frac{|K_{1,\tau,\el}|^2_H}{(|s_1|\cdots |s_l|)^{2n}}   \leq 2\int_{M} a_{\tau,\el}(\phi)\frac{|\tilde{f}_1 |^2_{H}}{(|s_1|\cdots |s_l|)^{2n}}  .$$
Note by (\ref{fk0}) the limit of the integral of R.H.S. of the above inequality is bounded after taking $\el\rw 0$ and then $\tau\rw -\infty.$ So the integral  of the L.H.S. is uniformaly  bounded.
Since $(\eta_{\el} +\lambda_{\el}^{-1})^{-1}$  is a smooth function, by counting the orders of the singularity at $p_1,\cdots,p_l$ of the denominator of the integral of the L.H.S. of the inequality above such that
$\frac{|K_{1,\tau,\el}|^2_H}{(|s_1|\cdots |s_l|)^{2n}}$ is locally integrable in  a neighborhood of $p_1,\cdots,p_l,$  
we must have
$$K_{1,\tau,\el}(p_1) =\cdots=K_{1,\tau,\el}(p_l) =0.$$
 Let $f_{1,\tau,\el}=((1-b_{\tau,\el}(\phi))\tilde{f}_1-K_{1,\tau,\el}.$
Then $D''f_{1,\tau,\el}=0$ and
$$\begin{array}{rcl}&& \int_{M}\frac{(\eta_{\el} +\lambda_{\el}^{-1})^{-1}}{(|s_1|\cdots |s_l|)^{2n}}|f_{1,\tau,\el}-(1-b_{\tau,\el}(\phi))\tilde{f}_1|^2_{H} \\
 &\leq & 2\int_{M} \frac{a_{\tau,\el}(\phi)}{(|s_1|\cdots |s_l|)^{2n}}|\tilde{f}_1 |^2_{H}.\\
 \end{array} $$ 
 
 Since  $K_{1,\tau,\el}$ 
 are uniformly $L^2$-bounded  (with respect to the Hermitian metric $\widetilde{H}_1$ on $E$),
 hence the holomorphic section family $f_{1,\tau,\el}$ are  uniformly $L^2$-bounded  and 
  have a weakly convergent subsequence $f_{1,\tau,\el_k}$ in
  $L^2$-topology.     It is not difficult to check the weak limit  $\hat{f}_{1,\tau}$  of the subsequence  $f_{1,\tau,\el_k}$ is a holomorphic section by using Cauchy integral formula.
   Hence by taking weak limit after passing to subsequence we have
$$\begin{array}{rcl} 
&&\int_{M }\frac{(\eta +\lambda^{-1})^{-1}}{(|s_1|\cdots |s_l|)^{2n}}|\hat{f}_{1,\tau}-(1-b_{\tau}(\phi))\tilde{f}_1|^2_{H}\\
 & \leq &
\underset{\el\rw 0}{\liminf}\int_{M}\frac{(\eta_{\el} +\lambda_{\el}^{-1})^{-1}}{(|s_1|\cdots |s_l|)^{2n}}|f_{1,\tau,\el}-(1-b_{\tau,\el}(\phi))\tilde{f}_1|^2_{H} \\
&\leq&2\int_{M } \frac{a_{\tau}(\phi)|\tilde{f}_1 |^2_{H}}{(|s_1|\cdots |s_l|)^{2n}}.\\
\end{array}$$
 where $b_{\tau}=\lim_{\el\rw 0} b_{\tau,\el}$ and  $v_{\tau}=\lim_{\el\rw 0} v_{\tau,\el}$ and $\eta =u(-v_{\tau}(\phi)).$
Note $b_{\tau}(t)=1$ for $t>\tau +\frac{1}{2}.$ In particular if $t>\tau +\frac{1}{2}$ then  $v_{\tau}(t)=\int^t_0 b_{\tau}(x) dx =t.$
Hence we have $\eta =u(\log \frac{1}{|s_1|})$ on any compact subset of $M\backslash \{p_1\}$ for $\tau\ll 0.$  We take the smooth function $u\in C^{\infty}[1,+\infty)$ defined by $u(x)=x+\frac{1}{\delta_1}x^{\delta_1},$ where $0<\delta_1 <1.$  Then $u$ satisfies assumption (\ref{ass}) and 
$$u(x)-\frac{u'(x)^2}{u''(x)}=x+\frac{1+2x^{\delta_1 -1}}{(1-\delta_1)x^{\delta_1 -2}}+\frac{x^{\delta_1}}{\delta_1 (1-\delta_1)}\geq \frac{x^{\delta_1}}{\delta_1 (1-\delta_1)}.$$Therefore we have
$(\eta +\lambda^{-1})\circ (-v_{\tau}(\phi))=\frac{1}{\delta_1 (1-\delta_1)}(\log\frac{1}{|s_1|})^{\delta_1}$ for $\tau\ll 0$ and
$$ \int_{M }\frac{|\hat{f}_{1,\tau}-(1-b_{\tau}(\phi))\tilde{f}_1|^2_{H}}{\Big({(|s_1|\cdots |s_l|)^{2n}}\Big)(\log\frac{1}{|s_1|})^{\delta_1}}  
\leq \frac{2}{\delta_1 (1-\delta_1)}\int_{M } \frac{a_{\tau}(\phi)|\tilde{f}_1 |^2_{H}}{{(|s_1|\cdots |s_l|)^{2n}}}, $$
The limit in the term of the R.H.S. of the inequality above is bounded
\begin{equation}
\lim_{\tau\rw -\infty} \int_{M} \frac{a_{\tau}(\phi)}{{(|s_1|\cdots |s_l|)^{2n}}}|\tilde{f}_1|^2_{H}
 \leq C |f|_H^2(p_1),\label{fk0}\end{equation}
where $C$ is a constant depending only on the diameter of $M$ and the   minimal distance among the points in $p_1,\cdots,p_l.$ This fact is a special case of Lemma \ref{claim} of page \pageref{claim} where we prove  the inductive process. Just like taking weak limits of $f_{1,\tau,\el}$ after passing to subsequence in necessary,
 as $\el\rw 0,$
 now we may likewise take weak limit of $\hat{f}_{1,\tau}$  by selecting a convergent subsequence,
 as $\tau\rw-\infty.$ 
 Denote $\hat{f}_1$ the weak limit holomorphic section,  then we  get
  $$
 \int_{M }\frac{|\hat{f}_1|^2_{H}}{{(|s_1|\cdots |s_l|)^{2n}}\Big(\log\frac{1}{|s_1|}\Big)^{\delta_1}}
\leq C |f|_H^2(p_1),
$$
 where $C$ is a constant depending only on $\delta_1,$ the diameter of $M,$ and the   minimal distance among the points in $p_1,\cdots,p_l.$ 
 
 In the following we directly assume $K_{1,\tau,\el}$ and $f_{1,\tau,\el}$ are weakly convergent (without passing to subsequnces) without loss generality.
Since $\tilde{f}_1(p_j)=K_{1,\tau,\el}(p_j)=0$ for $j\geq 2$ we have
$f_{1,\tau,\el}(p_j)=((1-b_{\tau,\el}(\phi))\tilde{f}_1(p_j)-K_{1,\tau,\el}(p_j)=0$ for $j\geq 2$  and hence $\hat{f}_1(p_j)=\lim_{\tau\rw -\infty}\lim_{\el\rw 0}f_{1,\tau,\el}(p_j)=0.$
Note $b_{\tau,\el}(-\infty)=0,$  therefore $f_{1,\tau,\el}(p_1)=((1-b_{\tau,\el}(\phi))\tilde{f}_1(p_1)-K_{1,\tau,\el}(p_1)=\tilde{f}_1(p_1)=f(p_1),$ and hence $\hat{f}_1(p_1)=\lim_{\tau\rw -\infty}\lim_{\el\rw 0}f_{1,\tau,\el}(p_j)=f_1(p_1).$ 

Now let $F=\hat{f}_1 +\cdots +\hat{f}_l.$ Then $F(p_j)=f(p_j)$ for $j=1,\cdots,l$ and $|F|_H^2\leq l(|\hat{f}_1|^2_H+\cdots+|\hat{f}_l|^2_H ).$
 By the inequality of arithmetic and geometric means we have
$$\int_{M }\frac{|F|^2_{H}}{\Big({\prod}_{j=1}^l(|s_j|^{2n}\log\frac{1}{|s_j|})^{\delta_j}\Big)} 
\leq C\sum|f|_H^2(p_j)=C\Vert f\Vert^2_0,
$$
 $C$ are finite positive  constants depending only on $\delta_j$ and the minimal  distance between any two points among $p_1,\cdots,p_l.$ This finishes the proof of the case $|\fk|=0.$

\vskip 0.5cm

 Assume Theorem \ref{exthm}  and Lemma \ref{exlem}  have been proved for $|\fk|-1.$ 
Without loss of generality we assume $k_1\geq 1.$ Consider the exact sequence of sheaves:
$$0\rw (S^{k_1} T^*_M)|_{p_1}\rw \ko/(\kj^{k_1+1}_{p_1}\cdots \kj^{k_l+1}_{p_l})\rw \ko/(\kj^{k_1 }_{p_1}\kj^{k_2+1 }_{p_2}\cdots \kj^{k_l +1}_{p_l})\rw 0$$
Given $f\in  J^{\fk}_{\fp} (K_M \otimes E),$  let $g$ be the image of $f$ in $H^0(M,\frac{K_M\otimes E} {\kj^{k_1 }_{p_1}\kj^{k_2+1 }_{p_2}\cdots \kj^{k_l+1}_{p_l}})$ under the induced cohomology group morphism. By the induction hypothesis, there exists a lifting section $F_{\fk-1}\in H^0(M,K_M\otimes E)$ such that $J^{\fk-1}_p (F_{\fk-1})=g$ and
\begin{equation}\int_M \frac{|F_{\fk-1}|^2_H}{{\prod}_{j=1}^l \Big(|s_j|^{2n}\Big(\log \frac{1}{|s_j|}\Big)^{\delta_j}\Big)} \leq C\Vert f\Vert ^2_{\fk-1}.\label{inda}\end{equation}
Here we denote $\fk-1=(k_1-1,k_2,\cdots,k_l)$ for brevity.
Moreover we can view $f-J^{\fk-1}_p (F_{\fk-1})$ as a global holomorphic section of $K_M\otimes E\otimes  (S^{k_1} T^*_M)|_{p_1}$ by using the induced cohomology group exact sequence 
induced by the following short exact sequence of sheaves 
$$0\rw K_M\otimes E\otimes (S^{k_1} T^*_M)|_{p_1}\rw \frac{K_M\otimes E}{\kj^{k_1+1}_{p_1}\cdots \kj^{k_l+1}_{p_l}}\rw \frac{K_M\otimes E}{\kj^{k_1 }_{p_1}\kj^{k_2+1 }_{p_2}\cdots \kj^{k_l+1}_{p_l}}\rw 0.$$

We will construct holomorphic lifting section of  $f-J^{\fk-1}_{\fp} (F_{\fk-1})$ via solving $\bar{\partial}$-equation
 by using H\"omander's $L^2$-theory.

Let $\rho_{\el},a_{\tau,\el},b_{\tau,\el}, v_{\tau,\el}$ defined as before and 
   let $\phi =\log |s_1|$
 and $\eta_{\el} =u(-v_{\tau,\el}\circ \phi)$  and $\lambda_{\el}=-\frac{u''_{\el}}{u'^2_{\el}}.$   Then $\eta_{\el}+\lambda_{\el}^{-1}=\Big(u -\frac{u'^2}{u''}\Big)\circ (-v_{\tau,\el}\circ \phi).$
    Let $$\widetilde{H}=He^{-2\sum_{j=1}^l (n+k_j)\log |s_j|}$$ be a new Hermitian metric on $K_M\otimes E,$ then  for any $E$-valued $(n,1)$-form $\vartheta,$ the operator
  $B_{\el}:=[ \eta_{\el} \Theta_{\widetilde{H}} -id'd''\eta_{\el}-i\lmd d'\eta_{\el}\wedge d''\eta_{\el},\Lm]$  satisfies
  $$\langle B_{\el}\vartheta,\vartheta\rangle  \geq \langle [a_{\tau,\el}(\phi)id'\phi\wedge d''\phi),\Lm]\vartheta,\vartheta\rangle \geq a_{\tau,\el}(\phi) |(d''\phi)^* \vartheta|^2. $$
Let  $\wf$ be a holomorphic lifting section of $f$ in $K_M\otimes E$ and put
 $$g_{\el}=D''((1-b_{\tau,\el}(\phi))(\wf-F_{\fk-1}))=(-a_{\tau,\el}(\phi)d''\phi)(\wf-F_{\fk-1}). $$
The support of $g_{\el}$ is contained in $\phi^{-1}([\tau-\frac{1-\el}{2},\tau+\frac{1-\el}{2}]).$ Hence $g_{\el}$ has its support
contained in the tubular neighborhood of $p_1$ as $\tau\rightarrow -\infty.$
By Proposition \ref{exist}, there exist $G_{\tau,\el}\in L^2(M,K_M\otimes E)$ such that
$$D'' G_{\tau,\el}=g_{\el},$$ and

$$ \int_{M}(\eta_{\el} +\lambda_{\el}^{-1})^{-1}\frac{|G_{\tau,\el}|^2_H}{{\prod}_{j=1}^l |s_j|^{2(n+k_j)}}   \leq 2\int_{M} a_{\tau,\el}(\phi)\frac{|\wf-F_{\fk-1}|^2_{H}}{{\prod}_{j=1}^l |s_j|^{2(n+k_j)}}  .$$
In the following Lemma \ref{claim} we will prove the right hand side  has a uniform bound not depending on $\tau$ as $\el\rw 0$ and $\tau\rw -\infty.$
Hence the integral of the L.H.S.  of the inequality above is uniformly $L^2$-bounded with respect to the metric  $He^{-2\sum_{j=1}^l (n+k_j)\log |s_j|}.$
Note by counting the order of the singularity at $p_1,\cdots,p_l$ of the denominator of the integral of the left hand side of the inequality above such that
$\frac{|G_{\tau,\el}|^2_H}{{\prod}_{j=1}^l |s_j|^{2(n+k_j)}} $ is locally integrable in  a neighborhood of $p_j,$  
we must have
\begin{equation}\partial ^{\al}G_{\tau,\el}|_{p_j} =0\hskip 0.2cm{\rm with}~~~ |\al|\leq k_j~~~~~ {\rm for} \hskip 0.2cm j=1,\cdots,l.\label{gg}\end{equation}
 Let $F_{\tau,\el}=((1-b_{\tau,\el}(\phi))(\wf-F_{\fk-1}))-G_{\tau,\el}.$
Then $D''F_{\tau,\el}=0$ and
$$\begin{array}{rcl}&& \int_{M}\frac{(\eta_{\el} +\lambda_{\el}^{-1})^{-1}}{{\prod}_{j=1}^l |s_j|^{2(n+k_j)}}|F_{\tau,\el}-(1-b_{\tau,\el}(\phi))(\wf-F_{\fk-1})|^2_{H} \\
 &\leq & 2\int_{M} \frac{a_{\tau,\el}(\phi)}{{\prod}_{j=1}^l |s_j|^{2(n+k_j)}}|\wf-F_{\fk-1}|^2_{H}.\\
 \end{array} $$
 Note $\{G_{\tau,\el}\}$ are uniformly
 $L^2$-bounded, hence  $\{F_{\tau,\el}\}$ 
  are uniformly
 $L^2$-bounded holomorphic sections of $K_M\otimes E,$ hence as $\el\rw 0,$ by taking weakly convergent subsequence if in need 
 we may assume  
 they  convergent weakly to a holomorphic section, denoted by $F_{\tau}.$ Let $b_{\tau}=\lim_{\el\rw 0} b_{\tau,\el}$ and  $v_{\tau}=\lim_{\el\rw 0} v_{\tau,\el}.$ Then we have
$$ \int_{M }\frac{(\eta +\lambda^{-1})^{-1}}{{\prod}_{j=1}^l |s_j|^{2(n+k_j)}}|F_{\tau}-(1-b_{\tau}(\phi))(\wf-F_{\fk-1})|^2_{H}  \leq 2\int_{M } \frac{a_{\tau}(\phi)|\wf-F_{\fk-1}|^2_{H}}{{\prod}_{j=1}^l |s_j|^{2(n+k_j)}}$$
where $\eta  =u(\log \frac{1}{|s_1|})$ on any compact subset of $M\backslash \{p_1\}$ for $\tau\ll 0.$
Take  $u(x)=x+\frac{1}{\delta_1}x^{\delta_1},$ where $0<\delta_1 <1$ and $x>1.$  Then $u$ satisfies assumption (\ref{ass}) and 
$(\eta +\lambda^{-1})\circ (-v_{\tau}(\phi))=\frac{1}{\delta_1 (1-\delta_1)}(\log\frac{1}{|s_1|})^{\delta_1}$ for $\tau\ll 0,$ hence
$$ \int_{M }\frac{|F_{\tau}-(1-b_{\tau}(\phi))(\wf-F_{\fk-1})|^2_{H}}{\Big({{\prod}_{j=1}^l |s_j|^{2(n+k_j)}}\Big)(\log\frac{1}{|s_1|})^{\delta_1}}  
\leq \frac{2}{\delta_1 (1-\delta_1)}\int_{M } \frac{a_{\tau}(\phi)|\wf-F_{\fk-1}|^2_{H}}{{{\prod}_{j=1}^l |s_j|^{2(n+k_j)}}}, $$
In the following Lemma \ref{claim} we will prove the right hand side has a uniform bound not depending on $\tau$ as $\tau\rw -\infty.$
 Hence we could take weak limit again. Let $F_k$ be the weak limit of $F_{\tau}$ (or its subsequence) as $\tau\rw -\infty.$ 
 Likewise   we have
$$\begin{array}{rcl}
&& \int_{M }\frac{|F_k|^2_{H}}{{{\prod}_{j=1}^l |s_j|^{2(n+k_j)}}\Big(\log\frac{1}{|s_1|}\Big)^{\delta_1}}\\
& \leq& \lim_{\tau\rw -\infty}\frac{1}{\delta_1 (1-\delta_1)}\int_{M }
 \frac{a_{\tau}(\phi)|\wf-F_{\fk-1}|^2_{H}}{{{\prod}_{j=1}^l |s_j|^{2(n+k_j)}}}. \\
 \end{array} $$

\begin{lemma}\label{claim}
 $$ \begin{array}{rcl}
&&\lim_{\tau\rw -\infty} \int_{M} \frac{a_{\tau}(\phi)|\wf-F_{\fk-1}|^2_{H}}{{{\prod}_{j=1}^l |s_j|^{2(n+k_j)}}}\\
 &\leq& C\Big( \sum_{|\al|= k_1} | \partial^{\al} f |_H^2 (p_1)+ \int_{M }\frac{|F_{\fk-1}|^2_{H}}{{\prod}_{1\leq j\leq l}(|s_j|^{2n}(\log \frac{1}{|s_j|})^{\delta_j})}\Big),\\
 \end{array}$$
 \end{lemma}

\noindent where $C$ is a constant depending only on the diameter of $M$ and the   minimal distance among the points in $p_1,\cdots,p_l.$
\vskip 0.5cm

\noindent Assume the Lemma \ref{claim},  we can finish the proof of Theorem \ref{exthm} and Lemma \ref{exlem}. In fact by the induction assumption   {\rm ($A_{\fk -1}$)}   and note $\frac{1}{|s_j|^{2(k_j +n)}}\geq \frac{1}{|s_j|^{2n}}$ 
we get  
 \begin{equation} \int_{M }\frac{|F_k|^2_{H}}{(|s_1|\cdots |s_l|)^{2n}(\log\frac{1}{|s_1|})^{\delta_1}} 
\leq  C ( \sum_{|\al|\leq k_1}| \partial^{\al} f|_H^2 (p_1)+||f||_{\fk-1}).\label{mast1} \end{equation}

 In the following we directly assume $G_{\tau,\el}$ and $F_{\tau,\el}$  are weakly convergent (without passing to subsequnces) without loss generality.
Now put $F=F_k +F_{\fk-1}.$   Note that $F_{\tau,\el}=(1-b_{\tau,\el}(\phi))(\wf-F_{\fk-1})-G_{\tau,\el}\rw F_k$ when $\el\rw 0$ and $\tau\rw -\infty.$ 
By definition of $b_{\tau,\el},$  when $p$ is inside a sufficiently small neighborhood  of $p_1$  we have $b_{\tau,\el}(\phi(p))=0$ and hence in a small neighborhood of $p_1$ we have $F_{\tau,\el}(p)=\wf(p)-G_{\tau,\el}(p)$ and
by (\ref{gg}) we have
$$\partial ^{\al} F(p_1)=\partial ^{\al} \wf (p_1)$$ for  any $\al$ with $|\al|\leq k_1.$
For any $p\not=p_1$ we have $$\lim_{\tau\rw -\infty}\lim_{\el\rw 0}(1-b_{\tau,\el}(\phi(p)))=0,$$ 
In fact if $\tau\ll 0$ then we have $b_{\tau,\el}(\phi(p))=1$ for any fixed $p\not=p_1.$  By (\ref{gg}) we have
$$\partial ^{\al} F(p_j)=-\lim_{\tau\rw -\infty}\lim_{\el\rw 0}\partial ^{\al} G_{\tau,\el}(p_j)+\partial ^{\al} F_{\fk-1}(p_j)=\partial ^{\al} \tilde{F}(p_j)$$ 
for any $\al$ with $|\al|\leq k_j$ and any $j\geq 2.$ 
Hence $F$ is the expected extension of jet section $f\in J^{\fk}_{\fp}(K_M\otimes E).$

By the induction assumption   {\rm ($B_{\fk -1,1}$)} 
 \begin{equation}  \int_M \frac{|F_{\fk-1}|^2_H}{ (|s_1|\cdots |s_l|)^{2n} \Big(\log \frac{1}{|s_1|}\Big)^{\delta_j}}\leq  C_1 \Vert f\Vert ^2_{\fk-1}.\label{mast2} \end{equation}
Note that $|F|^2_H\leq 2(|F_{k}|^2_H +|F_{\fk-1}|^2_H),$  by (\ref{mast1}) and (\ref{mast2}), 
$$\int_M \frac{|F|^2_H}{ (|s_1|\cdots |s_l|)^{2n} \Big(\log \frac{1}{|s_|}\Big)^{\delta_j}}\leq  C_2 \sum_{|\al|\leq k_1}| \partial^{\al} f |_H^2 (p_1)+C_1\Vert f\Vert ^2_{\fk-1}=C\Vert f\Vert ^2_{\fk},$$
where $C_1,C_2$ and $C$ are finite positive  constants depending only on $\delta_j$ and the minimal  distance between any two points among $p_1,\cdots,p_l.$ This establishes {\rm ($B_{\fk ,1}$)}   and hence {\rm ($A_{\fk }$)}. 

\vskip 0.7cm

\noindent {\it Proof of Lemma \ref{claim}:} Without loss of generality we assume the coordinate of $p_1$ is origin, i.e. $w_1=0.$  Take local trivialization of $K_M\otimes E$ on the small neighborhood of $p_1.$  Write
$$\wf-F_{\fk-1} =(\tilde{f}-f_{\fk-1})dz_1\wedge \cdots \wedge dz_n, $$
where $\tilde{f}$ and $f_{\fk-1}$ are $E$-valued holomorphic function.
 Since $\widetilde{F}$ is a lifting of $f$ and $J^{\fk-1}_{\fp}(F_{\fk-1})$ is the  image of $f$ in
  $$H^0(M,K_M\otimes E\otimes \ko/(\kj^{k_1 }_{p_1}\kj^{k_2 +1}_{p_2}\cdots \kj^{k_l+1}_{p_l})),$$ we know
    $$\partial^{\al} \tilde{f} (p_1)=\partial^{\al} f_{\fk-1}(p_1),\hskip 1cm \forall \hskip 0.5cm |\al|\leq k_1-1.$$
Let $dV=\wedge_{j=1}^n (idz_j\wedge d\bar{z}_j)$ be the volume element. Take coordinate transformation $z_j =e^{\tau/2} \tilde{z}_j$ and $\tilde{z}=(\tilde{z}_1,\cdots,\tilde{z}_n).$ Then   $dV=e^{n\tau}\wedge_{j=1}^n (id\tilde{z}_j\wedge d\bar{\tilde{z}}_j)=e^{n\tau}d{\sigma}$  with $d{\sigma}=\wedge_{j=1}^n (id\tilde{z}_j\wedge d\bar{\tilde{z}}_j).$ Under the coordinate transformation we have

\begin{equation}\frac{|\wf-F_{\fk-1}|^2_H }{|s_1|^{2k_1 +2n}}=h (27D)^{2k_1 +2n} \frac{|\tilde{f}-f_{\fk-1}|^2(e^{\tau/2} \tilde{z}_1,\cdots,e^{\tau/2} \tilde{z}_n) }
{e^{k_1 \tau}|\tilde{z}|^{2k_1}}d{\sigma},\label{cl1}\end{equation}
where $h$ is a function depending only on metric. Note
$$
\lim_{\tau\rw -\infty} \int_{M} \frac{a_{\tau}(\phi)|\wf-F_{\fk-1}|^2_{H}}{{\prod}_{j=1}^l |s_j|^{2(n+k_j)}}
=
\lim_{\tau\rw -\infty} \underset{e^{\tau-1/2}<|z|<e^{\tau+1/2}}{\int}  \frac{|\wf-F_{\fk-1}|^2_{H}}{{\prod}_{j=1}^l |s_j|^{2(n+k_j)}}.
$$
Using Taylor expansion $\tilde{f}-f_{\fk-1}=\sum_{\al}\frac{\partial^{\al}(\tilde{f}-f_{\fk-1})}{\al!}{\tilde{z}}^{\al}$
 together with (\ref{cl1})  we have
$$\begin{array}{rcl}
&&\underset{\tau\rw -\infty}{\lim} \int_{M} \frac{a_{\tau}(\phi)|\wf-F_{\fk-1}|^2_{H}}{{\prod}_{j=1}^l |s_j|^{2(n+k_j)}}\\
&=& \underset{\tau\rw -\infty}{\lim}\sum_{|\al|=k_1}\Big( \frac{|\partial^{\al}\tilde{f}-\partial^{\al}f_{\fk-1}|^2(0)}{|\al!|^2}\\
&&\cdot \underset{e^{\tau-1/2}<|z|<e^{\tau+1/2}}{\int} \frac{|\tilde{z}_1|^{2\al_1}\cdots |\tilde{z}_n|^{2\al_n}}{|\tilde{z}|^{2(k_1+n)}}\frac{h(27D)^{2k_1 +2n}}
{{{{\prod}_{j=2}^l |s_j|^{2(n+k_j)}}}}d{\sigma}\Big).\\
\end{array}$$
In the  integral  of the R.H.S. of the above equation 
$$\begin{array}{rcl}
&&\underset{\tau\rw -\infty}{\lim}\underset{e^{\tau-1/2}<|z|<e^{\tau+1/2}}{\int} \frac{|\tilde{z}_1|^{2\al_1}\cdots |\tilde{z}_n|^{2\al_n}}{{|\tilde{z}|^{2(k_1+n)}{\prod}_{j=2}^l |s_j|^{2(n+k_j)}}} d{\sigma}\\
&\leq& \underset{\tau\rw -\infty}{\lim}\Big(\frac{27D}{d}\Big)^{2(nl+|\fk|-k_1-n)}\underset{e^{\tau-1/2}<|z|<e^{\tau+1/2}}{\int} \frac{|\tilde{z}_1|^{
2\al_1}\cdots |\tilde{z}_n|^{2\al_n}}{|\tilde{z}|^{2(k_1+n)}}d{\sigma}\\
&=& 2^n\underset{\tau\rw -\infty}{\lim}\Big(\frac{27D}{d}\Big)^{2(nl+|\fk|-k_1-n)}\underset{e^{\tau-1/2}<|z|<e^{\tau+1/2}}{\int} \frac{1}{r}dr\cdot  \underset{|\tilde{z}|=1}{\oint}
\frac{|\tilde{z}_1|^{
2\al_1}\cdots |\tilde{z}_n|^{2\al_n}}{|\tilde{z}|^{2(k_1+n)}}dS\\
&=& 2^n \Big(\frac{27D}{d}\Big)^{2(nl+|\fk|-k_1-n)} \underset{|\tilde{z}|=1}{\oint}
\frac{|\tilde{z}_1|^{
2\al_1}\cdots |\tilde{z}_n|^{2\al_n}}{|\tilde{z}|^{2(k_1+n)}}dS\\
\end{array}$$
is a bounded constant. Here $dS$ is the area element of unit sphere in $\R^{2n}$
and $d$ is the minimal distance from $p_1$ to the point set $\{p_2,\cdots,p_l\}.$  Note $\partial^{\al}f (p_1)=\partial^{\al}\wf (p_1)=\partial^{\al}\tilde{f} (p_1)$ for $|\al|=k_1,$  so we have

\begin{equation}\lim_{\tau\rw -\infty} \int_{M} \frac{a_{\tau}(\phi)|\wf-F_{\fk-1}|^2_{H}}{{{\prod}_{j=1}^l |s_j|^{2(n+k_j)}}}
\leq B_1\sum_{|\al|=k_1}\Big(|\partial^{\al}f|_H^2(p_1)
+|\partial^{\al}f_{\fk-1}|_H^2(p_1)\Big),\label{latecite}
\end{equation}
where $B_1$ is a constant depending only on the diameter of $M$  and  minimal distance among the points $\{p_1,\cdots,p_l\}.$
 Expand $f_{\fk-1}(\tilde{z})=\sum_{\al}\frac{\partial^{\al}f_{\fk-1}(p_1)}{\al !}\tilde{z}^{\al}.$ By Parseval's formula
$$\begin{array}{rcl}
&&\int _{|\tilde{z}| <1} |F_{\fk-1}|^2_H \\
&=&\sum_{\al}\frac{|\partial^{\al} f_{\fk-1}|^2(p_1)}{(\al !)^2} \int _{|\tilde{z}| <1} |\tilde{z}_1|^{2\al_1} \cdots |\tilde{z}_n|^{2\al_n}d{\sigma},\\
\end{array}$$ we know
$$\sum_{|\al|=k_1}|\partial^{\al} f_{\fk-1}|^2(p_1)\leq\sum_{|\al|=k_1}\frac{(\al !)^2\int _{M} |F_{\fk-1}|^2_H}{\int _{\tilde{z}|<1} 
|\tilde{z}_1|^{2\al_1} \cdots |\tilde{z}_n|^{2\al_n}d{\sigma}}.$$ 
Hence
$$
\lim_{\tau\rw -\infty} \int_{M} \frac{a_{\tau}(\phi)|\wf-F_{\fk-1}|^2_{H}}{{{\prod}_{j=1}^l |s_j|^{2(n+k_j)}}}
\leq B_2\sum_{|\al|=k_1}\Big(|\partial^{\al}f|_H^2(p_1)
+\int_{M}|F_{\fk-1}|^2_H\Big),
$$
where $B_2$ is a constant depending only on the diameter of $M$  and  minimal distance among the points $\{p_1,\cdots,p_l\}.$
Since by our assumption that $|s_j|<1$ and $\log\frac{1}{|s_j|} >3,$    we have 
$$0<|s_j|^{2n}(\log\frac{1}{|s_j|})^{\delta_j}\leq |s_j|^{2n}(\log\frac{1}{|s_j|})\leq |s_j|^{2n-1}<1,$$ therefore
$$\begin{array}{rcl}
&&\underset{\tau\rw -\infty}{\lim} \int_{M} \frac{a_{\tau}(\phi)|\wf-F_{\fk-1}|^2_{H}}{{\prod}_{j=1}^l |s_j|^{2(n+k_j)} }\\
&\leq &C\sum_{|\al|=k_1}\Big(|\partial^{\al}f|_H^2(p_1)
+\int_{M}\frac{|F_{\fk-1}|^2_H}{{\prod}_{1\leq j\leq l}(|s_j|^{2n}(\log \frac{1}{|s_j|})^{\delta_j})}\Big),\\
\end{array}
$$
where $C$ is a constant depending only on the diameter of $M$  and  minimal distance among the points $\{p_1,\cdots,p_l\}.$ 
\end{proof}

\subsubsection{\bf Extension theorems for strong Nakano nef vector bundles}
Now assume $E$ is a strong Nakano nef vector bundle on a bounded Stein domain $M,$ which means for any positive number  $\rho >0,$ there exist a smooth Hermitian metric $H_{\rho}$ such that
 its curvature satisfying $\Theta_{H_{\rho}}(u,u)\geq -\rho ||u||^2$ for any $u\in T_M\otimes E,$ and $\{H_{\rho}\}$ have a subsequence convergent to a Nakano pseudo effective Hermitian metric $H,$
 called {\it limit Hermitian metric.}

\begin{theorem}\label{exth} Let $M$ be a  bounded Stein domain of diameter $D,$ and  $E$  a  strong Nakano nef holomorphic vector bundle   on $M,$  with a Nakano pseudo effective limit Hermitian metric $H.$  Then for
 any holomorphic jet section $f\in J^{\fk}_{\fp} (K_M \otimes E)$ satisfying $$\Vert f\Vert ^2_{\fk}<\infty,$$ there exists holomorphic section $F\in H^0(M,K_M \otimes E) $  such that $J^{\fk}_{\fp} (F) =f$ and
$$\int_M \frac{|F|^2_H}{{\prod}_{j=1}^l \Big(|s_j|^{2n}\Big(\log \frac{1}{|s_j|}\Big)^{\delta_j}\Big)} \leq C\Vert f\Vert ^2_{\fk}, $$
where $\delta_j >0 $ are  any positive constants and $C$ is a finite positive  constant depending only on $\delta_j$ and the minimal  distance between any two points among $p_1,\cdots,p_l.$
\end{theorem}

\begin{proof}
We may prove this Theorem using the same inductive
 process as  the proof of Theorem \ref{exthm},   the only difference is that we will now solve  an approximate $\bar{\partial}$-equation every time instead  when we solve a  $\bar{\partial}$-equations   
 during the proof of Theorem \ref{exthm}.  After taking limit we will get the same estimates as in Theorem \ref{exthm}.  For brevity  here we give a proof of the first step where $|\fk|=0,$ since the
 the rest procedure is a repeat modification process as the case $|\fk|=0.$
 
  Write a given $0$-jet section $f\in H^0 (M,K_M \otimes E/\kj_{p_1}\cdots \kj_{p_l})$  as a sum  $f=f_1 +\cdots+f_n,$  with $f_j\in H^0 (M,K_M \otimes E/\kj_{p_1}\cdots \kj_{p_l})$ satisfying (\ref{fid}).
  By symmetry positions of $f_1,\cdots,f_n,$ it suffices to extend $f_1$ and the others are extended in the same way.
  Let $\tilde{f}_1$ be a section of $H^0(M,K_M\otimes E)$ which is holomorphic lifting of $f_1$ 
  and let $a_{\tau,\el},b_{\tau,\el},v_{\tau,\el},\phi,\eta_{\el}$ be the functions as defined in (\ref{a}),(\ref{bc}) and (\ref{pet}) respectively, here in  (\ref{pet})
  we take $u(x)=x+\frac{1}{\delta_1}x^{\delta_1}$ with $0<\delta_1 <1.$  
     Let $$\widetilde{H}_1=H_{\rho}e^{-2n\sum_{j=1}^l \log |s_j|}$$ be a smooth Hermitian metric on $K_M\otimes E.$  Note here $H_{\rho}$ are a family smooth Hermitian metric on the strong Nakano nef bundle $E,$ and $$\langle [\Theta_{H_{\rho}},\Lambda]\vartheta,\vartheta\rangle \geq -\rho |\vartheta|_{H_{\rho}}^2$$ 
 for any $E$-valued $(n,1)$-form $\vartheta$ by (\ref{nakanoid}).  Hence the operator
   $B_{\el}=[ \eta_{\el} \Theta_{\widetilde{H}_1} -id'd''\eta_{\el}-i\lmd d'\eta_{\el}\wedge d''\eta_{\el},\Lm]$
   satisfy
$$\langle B_{\el}\vartheta,\vartheta\rangle  \geq \langle [a_{\tau,\el}(\phi)id'\phi\wedge d''\phi),\Lm]\vartheta,\vartheta\rangle \geq a_{\tau,\el}(\phi) |(d''\phi)^* \vartheta|_{H_{\rho}}^2-\rho\eta_{\el} |\vartheta|_{H_{\rho}}^2. $$
Note $b_{\tau,\el}(t)=0$ for $t<\tau -\frac{1-\el}{2}$ and $b_{\tau,\el}(t)=1$ for $t>\tau +\frac{1-\el}{2}.$ Hence
 if $\tau<-2-\el,$ then $v_{\tau,\el}(t)=t$ for $t\geq 0$   and $v_{\tau,\el}(t)\leq -(\tau-\frac{1-\el}{2})$ for $t<0.$  Hence for $\tau\ll 0$
 \begin{equation}\eta_{\el}=u(-v_{\tau,\el}(\phi))\leq [-(\tau-\frac{1-\el}{2})]-\frac{1}{\delta_1}[-(\tau-\frac{1-\el}{2})]^{\delta_1}\leq \tau^2.\label{taub}\end{equation}
Take $$\rho=\frac{1}{\tau ^4}.$$  
Then for $\tau\ll 0$ we have  
$$\rho\eta_{\el} \leq \frac{1}{\tau^2}.$$
Let
 $$L_{1,\tau,\el}=D''((1-b_{\tau,\el}(\phi))\tilde{f}_1)=(-a_{\tau,\el}(\phi)d''\phi)\tilde{f}_1. $$
If $\tau\ll 0$ such that $\{p_2,\cdots,p_l\}\subset M\backslash \phi^{-1}([\tau-\frac{1-\el}{2},\tau+\frac{1-\el}{2}]),$ then
$$L_{1,\tau,\el}(p_2)=\cdots =L_{1,\tau,\el}(p_l)=0.$$
By (ii) of Proposition \ref{exist}, there exist $K_{1,\tau,\el}\in L^2(M,\Lambda^{n,0} T^*_M\otimes E)$ and $ h_{1,\tau,\el}\in L^2(M,\Lambda^{n,1} T^*_M\otimes E)$ such that
$$D'' K_{1,\tau,\el}+\frac{1}{\tau}h_{1,\tau,\el}=L_{1,\tau,\el},$$ and
$$ \int_{M}\frac{(\eta_{\el} +\lambda_{\el}^{-1})^{-1}|K_{1,\tau,\el}|^2_{H_{\rho}}}{(|s_1|\cdots |s_l|)^{2n}} + \int_{M}\frac{|h_{1,\tau,\el}|^2_{H_{\rho}}}{(|s_1|\cdots |s_l|)^{2n}}  \leq 2\int_{M} \frac{a_{\tau,\el}(\phi)|\tilde{f}_1 |^2_{H_{\rho}}}{(|s_1|\cdots |s_l|)^{2n}} .$$
By Lemma \ref{claim} (using the similar proof)  just like the estimate (\ref{fk0}), the limit $\lim_{\tau\rw -\infty}\lim_{\el\rw 0}\int_{M} \frac{a_{\tau,\el}(\phi)|\tilde{f}_1 |^2_{H_{\rho}}}{(|s_1|\cdots |s_l|)^{2n}}$
is bounded, hence the two integrals of L.H.S. of the above inequality are uniformly bounded. 
Since $$D''h_{1,\tau,\el}=\tau\cdot D''( L_{1,\tau,\el}-D'' K_{1,\tau,\el}) =0,$$ and $h_{1,\tau,\el}$ is square integrable with respect to the metric $$H_{\rho}e^{-2n\sum_{j=1}^l \log |s_j|}$$
from the estimate in the above inequality. Hence by the existence theorem of $\bar{\partial}$-equation on the Stein domain, there exists $E$-valued  function $g_{1,\tau,\el}$ such that
 $$D''g_{1,\tau,\el}=h_{1,\tau,\el}$$ 
 with the estimate
  $$\int_{M}\frac{|g_{1,\tau,\el}|^2_{H_{\rho}}}{(|s_1|\cdots |s_l|)^{2n}}  \leq C\int_{M}\frac{|h_{1,\tau,\el}|^2_{H_{\rho}}}{(|s_1|\cdots |s_l|)^{2n}} $$
  by H\"ormander $L^2$-estimate.  Then
  $$D''((1-b_{\tau,\el}(\phi)\tilde{f}_1-K_{1,\tau,\el}-\frac{1}{\tau}g_{1,\tau,\el})=0.$$
 Let
 $$f_{1,\tau,\el}=(1-b_{\tau,\el}(\phi))\tilde{f}_1-K_{1,\tau,\el}-\frac{1}{\tau}g_{1,\tau,\el}.$$
Then $D''f_{1,\tau,\el}=0$ and  $f_{1,\tau,\el}$ is holomorphic section. It in particular means that $K_{1,\tau,\el}+\frac{1}{\tau}g_{1,\tau,\el}$ is a smooth section.

$$\begin{array}{rcl}&& \int_{M}\frac{(\eta_{\el} +\lambda_{\el}^{-1})^{-1}}{(|s_1|\cdots |s_l|)^{2n}}|f_{1,\tau,\el}-(1-b_{\tau,\el}(\phi))\tilde{f}_1 -\frac{1}{\tau}g_{1,\tau,\el}|^2_{H_{\rho}} \\
 &\leq & 2\int_{M} \frac{a_{\tau,\el}(\phi)}{(|s_1|\cdots |s_l|)^{2n}}|\tilde{f}_1 |^2_{H_{\rho}}.\\
 \end{array} $$
So  as $\el\rw 0,$ $f_{1,\tau,\el},K_{1,\tau,\el},g_{1,\tau,\el}$ and  $h_{1,\tau,\el}$ have subsequences
which have weak limits in $L^2$, denoted by  $\hat{f}_{1,\tau},\hat{K}_{1,\tau}, \hat{g}_{1,\tau}$ and $\hat{h}_{1,\tau}$ respectively,  note $\hat{f}_{1,\tau}=(1-b_{\tau}(\phi))\tilde{f}_1
-\hat{K}_{1,\tau}-\frac{1}{\tau}g_{1,\tau}$ is  a holomorphic 
  section.
 
  Let $b_{\tau}=\lim_{\el\rw 0} b_{\tau,\el}$ and  $v_{\tau}=\lim_{\el\rw 0} v_{\tau,\el},$ then 
  \begin{equation}\lim_{\el\rw 0}(\eta_{\el} +\lambda_{\el}^{-1})^{-1}=\frac{1}{\delta_1 (1-\delta_1)}(\log\frac{1}{|s_1|})^{\delta_1}\label{btau}\end{equation}
   for $\tau\ll 0.$ 
  Taking weak limits of the above two inequalities after passing to subsequences,
   just like in the proof of Theorem \ref{exthm},
  we have
\begin{equation} \int_{M }\frac{|\hat{K}_{1,\tau}|^2_{H_{\rho}}} {\Big({(|s_1|\cdots |s_l|)^{2n}}\Big)(\log\frac{1}{|s_1|})^{\delta_1}}   \leq \frac{2}{\delta_1 (1-\delta_1)}\int_{M } \frac{a_{\tau}(\phi)|\tilde{f}_1 |^2_{H_{\rho}}}{(|s_1|\cdots |s_l|)^{2n}}\label{repat}\end{equation}
and $$\int_{M}\frac{|g_{1,\tau}|^2_{H_{\rho}}}{(|s_1|\cdots |s_l|)^{2n}}  \leq C\int_{M}\frac{|h_{1,\tau}|^2_{H_{\rho}}}{(|s_1|\cdots |s_l|)^{2n}} \leq 2C\int_{M } \frac{a_{\tau}(\phi)|\tilde{f}_1 |^2_{H_{\rho}}}{(|s_1|\cdots |s_l|)^{2n}}.$$
Since $\log\frac{1}{|s_1|}>1$  we know the weighted $L^2$-integrals
$$\int_{M}\frac{|\hat{K}_{1,\tau,\el}+\frac{1}{\tau}g_{1,\tau,\el}|^2_{H_{\rho}}}{(|s_1|\cdots |s_l|)^{2n}}$$
 with resect to the weight function $\frac{1}{(|s_1|\cdots |s_l|)^{2n}}$ 
are uniformly bounded.
by counting the orders of the singularity at $p_1,\cdots,p_l$ of the denominator we have
$$(\hat{K}_{1,\tau,\el}+\frac{1}{\tau}g_{1,\tau,\el})(p_1) =\cdots=(\hat{K}_{1,\tau,\el}+\frac{1}{\tau}g_{1,\tau,\el})(p_l) =0.$$

Note $\rho=\frac{1}{\tau^4}.$ By definition of strong Nakano nef vector bundle $\{H_{\rho}\}$ have a subsequence convergent to a limit Hermitian metric $H$ with
Nakano pseudo effective curvature current. After further taking subsequences  we may assume  $\hat{f}_{1,\tau}$ is weakly convergent to $\hat{f}_{1}.$
Note $\hat{f}_{1,\tau}=(1-b_{\tau}(\phi))\tilde{f}_1
-\hat{K}_{1,\tau}-\frac{1}{\tau}g_{1,\tau}$ and  any subseqences of $(1-b_{\tau}(\phi))\tilde{f}_1$ and $\frac{1}{\tau}g_{1,\tau}$ are always weakly convergent to zeros.  Hence  take weak limit as $\tau\rw -\infty$
of the inequality (\ref{repat}) we obtain 
$$
 \int_{M }\frac{|\hat{f}_1|^2_{H}}{{(|s_1|\cdots |s_l|)^{2n}}\Big(\log\frac{1}{|s_1|}\Big)^{\delta_1}}
\leq C_1 |\tilde{f}_1|_H^2(p_1).
$$
 where $C_1$ is a constant depending only on $\delta_1,$ the diameter of $M,$ and the   minimal distance among the points in $p_1,\cdots,p_l.$ 
 
 Without loss of generality,   in the following we assume $f_{1,\tau,\el},K_{1,\tau,\el},g_{1,\tau,\el}$ and  $h_{1,\tau,\el}$ themselves (not passing to subsequences) are  weakly convergent
  as $\el\rw 0$ and $\tau\rw -\infty.$
 
Since $\tilde{f}_1(p_j)=K_{1,\tau,\el}(p_j)+\frac{1}{\tau}g_{1,\tau,\el}(p_j)=0$ for $j\geq 2$ we have
$f_{1,\tau,\el}(p_j)=((1-b_{\tau,\el}(\phi))\tilde{f}_1(p_j)-K_{1,\tau,\el}(p_j)-\frac{1}{\tau}g_{1,\tau,\el}(p_j)=0$ for $j\geq 2$  and hence $\hat{f}_1(p_j)=\lim_{\tau\rw -\infty}\lim_{\el\rw 0}f_{1,\tau,\el}(p_j)=0.$
Note $b_{\tau,\el}(-\infty)=0,$  therefore $f_{1,\tau,\el}(p_1)=((1-b_{\tau,\el}(\phi))\tilde{f}_1(p_1)-K_{1,\tau,\el}(p_1)-\frac{1}{\tau}g_{1,\tau,\el}(p_1)=\tilde{f}_1(p_1)=f(p_1),$ and hence $\hat{f}_1(p_1)=\lim_{\tau\rw -\infty}\lim_{\el\rw 0}f_{1,\tau,\el}(p_j)=f_1(p_1).$ 

Now let $F=\hat{f}_1 +\cdots +\hat{f}_l.$ Then $F(p_j)=f(p_j)$ for $j=1,\cdots,l$ and $|F|_H^2\leq l(|\hat{f}_1|^2_H+\cdots+|\hat{f}_l|^2_H ).$
 By the inequality of arithmetic and geometric means we have
$$\int_{M }\frac{|F|^2_{H}}{\Big({\prod}_{j=1}^l(|s_j|^{2n}\log\frac{1}{|s_j|})^{\delta_j}\Big)} 
\leq C_2\sum|f|_H^2(p_j)=C_2\Vert f\Vert^2_0,
$$
 $C_2$ are finite positive  constants depending only on $\delta_j$ and the minimal  distance between any two points among $p_1,\cdots,p_l.$ This finishes the proof of the case $|\fk|=0.$

\end{proof}

\begin{corollary}\label{excor1}
Let $M$ be a  bounded Stein domain of diameter $D,$ and  $E$  a  strong Nakano nef holomorphic vector bundle   on $M,$  with a Nakano pseudo effective limit Hermitian metric $H.$   Let $s_j (z)=z-w_j$ for $j=1,\cdots,l.$  Then for
 any holomorphic jet section $f\in J^{\fk}_{\fp} (K_M \otimes E)$ satisfying $$\Vert f\Vert ^2_{\fk}<\infty,$$ there exists holomorphic section $F\in H^0(M,K_M \otimes E) $  such that $J^{\fk}_{\fp} (F) =f$ and
$$\int_M \frac{|F|^2_H}{{\prod}_{j=1}^l (1+|s_j|^2)^{n}} \leq C\Vert f\Vert ^2_{\fk}$$
 where $C$ is a finite positive  constant depending only on the diameter of $M$ and the minimal  distance between any two points among $p_1,\cdots,p_l.$
\end{corollary}

\begin{corollary}\label{excor2}
Let $M$ be a  bounded Stein domain of diameter $D,$ and  $E$  a  strong Nakano nef holomorphic vector bundle   on $M,$  with a Nakano pseudo effective limit Hermitian metric $H.$  Then for
 any holomorphic jet section $f\in J^{\fk}_{\fp} (K_M \otimes E)$ satisfying $$\Vert f\Vert ^2_{\fk}<\infty,$$ there exists holomorphic section $F\in H^0(M,K_M \otimes E) $  such that $J^{\fk}_{\fp} (F) =f$ and
$$\int_M |F|^2_H \leq C\Vert f\Vert ^2_{\fk}$$
 where $C$ is a finite positive  constant depending only on the diameter of $M$ and the minimal  distance between any two points among $p_1,\cdots,p_l.$
\end{corollary}

\noindent {\it Proof of Corollary \ref{excor1} and \ref{excor2}}.
If $s_j(z)=z-w_j$ and $D$ is diameter of $M,$ then $|\frac{s_j}{27D}|<1$ and $ \log \frac{1}{\frac{|s_j|}{27D}}\geq 1.$
 Set $h(x)=(1+x^2)^{n}-x^{2n}\Big(\log \frac{1}{|\frac{x}{27D}|}\Big)^{\delta_j}$  with $0<x<D.$ Since $\lim_{x\rw 0^+} h(x) =1,$ there exist $\kappa>0$ such that $x\leq \kappa$ then $h(x)>1/2$ for any $\delta_j \geq 0.$ Now assume $x\geq \kappa,$ then $\frac{27D}{x}-1\leq \frac{27D}{\kappa}-1$ and
 $h(x)\geq x^{2n}\Big((1+\frac{1}{x^2})^{n}-(\frac{27D}{\kappa}-1)^{\delta_j}\Big).$ Take
 $$\delta_j =\inf_{D>x\geq \kappa} \frac{m\log(1+\frac{1}{x^2})}{\log (\frac{27D}{\kappa}-1)}=\frac{m\log(1+\frac{1}{D^2})}{\log (\frac{27D}{\kappa}-1)}$$
 then $h(x)\geq 0$ for any $0\leq x\leq D.$  Now for each $j$ we take  a fixed $\delta_j$ as defined above for $j=1,\cdots,l$ in Theorem \ref{exthm}, then
 we have
$$\frac{|F|^2_H}{{\prod}_{j=1}^l (\frac{1}{27^2D^2}+|\frac{s_j}{27D}|^2)^{n}}\leq \frac{|F|^2_H}{{\prod}_{j=1}^l \Big(|\frac{s_j}{27D}|^{2n}\Big(\log \frac{1}{|\frac{s_j}{27D}|}\Big)^{\delta_j}\Big)},$$
 from it Corollary \ref{excor1}  follows immediately from Theorem \ref{exthm},  and Corollary \ref{excor2} follows directly from Corollary \ref{excor1}.
 \hfill  $\Box$
 
\subsection{Extension theorems  on compact K\"ahler manifolds}\label{subsub3}

In this subsection we denote $M$ a compact K\"ahler manifold and 
$E$  a strong Nakano nef vector bundle on  $M.$ The family  of  a smooth Hermitian metrics on $E$ is denoted by $\{H_{\rho}|\rho>0\}$  and its limit Hermitian metric is denoted by $H.$

We cut the  compact K\"ahler manifold $M$ into small bounded Stein open subsets and to extend  a given jet section $f$ to each Stein open subset which insect the support of $f$ in a compatible way via using the Theorem \ref{exth}  and Lemma \ref{exlem}  established in Section \ref{subsub2},
then glue these extended sections to a smooth global section on $M.$ After some small modification of this globally defined smooth section via solving $\bar{\partial}$-equation, we will obtain the expected  holomorphic section defined on $M$ which is an extension of $f.$

Suppose the support of the jet section $f$ is the  set $\{p_1,\cdots,p_l\}$  contained in a Stein open ball $U$ of $M$ whose diameter is  $D.$
Choose holomorphic local coordinate charts of $M,$ and denote the coordinate of $p_1,\cdots,p_l$ by $w_1,\cdots,w_l.$ Let $(z_1,\cdots,z_n)$ be the holomorphic coordinate on $U.$ Denote  $s_j (z)=\frac{z-w_j}{27D}$ for
   $j=1,\cdots,l$  as at the beginning of  Section \ref{subsub2}.

 \begin{theorem}\label{exh} Let $M$ be a  compact K\"ahler manifold and  $E$  a  strong Nakano nef holomorphic vector bundle   on $M,$  with a Nakano pseudo effective limit Hermitian metric $H.$  Let
 $f\in J^{\fk}_{\fp} (K_M \otimes E)$  be a holomorphic jet section whose support is a set of finite distinct points locate in a  stein open ball $U$ of $M,$ and
  satisfying $$\Vert f\Vert ^2_{\fk}<\infty.$$ Then there exists holomorphic section $F\in H^0(M,K_M \otimes E) $  such that $J^{\fk}_{\fp} (F) =f$ and there is a  finite positive  constant $C$ with
$$\int_M \frac{|F|^2_H}{{\prod}_{j=1}^l \Big(|s_j|^{2n}\Big(\log \frac{1}{|s_j|}\Big)^{\delta_j}\Big)} \leq C\Vert f\Vert ^2_{\fk}, \hskip 3cm {\rm (C_{\fk,j})}$$
where $\delta_j >0 $ are  any positive constants.
\end{theorem}

\begin{proof}  Like the proof of Theorem \ref{exthm},  it suffices to prove the following

\begin{lemma}\label{exle}Under the assumptions in Theorem \ref{exh}, then for
 any holomorphic jet section $f\in J^{\fk}_{\fp} (K_M \otimes E)$ satisfying $$\Vert f\Vert ^2_{\fk}<\infty,$$ there exists holomorphic section $F\in H^0(M,K_M \otimes E) $  such that $J^{\fk}_{\fp} (F) =f$ and for any $1\leq j\leq l$,  there is a  finite positive  constant $C$ such that
 $$\int_M \frac{|F|^2_H}{\Big(|s_1|\cdots |s_l|)^{2n}\Big(\log \frac{1}{|s_j|}\Big)^{\delta_j}} \leq C\Vert f\Vert ^2_{\fk},\hskip 3cm {\rm (D_{\fk,j})}$$
where $\delta_j >0 $ are  any positive constants.
\end{lemma}

We will use  almost  the same induction procedure  as  the proof of Theorem \ref{exthm},  
but solve different approximate $\bar{\partial}$-equations from that in the proof of Theorem \ref{exth}, then take limits.
Since in the proof
of Theorem \ref{exth} we have done the inductive step  $|\fk|=0$ a second time,
here we skip the step $|\fk|=0.$ The essential technical ideas for $|\fk|=0$ and the proof we give here for the inductive step from $|\fk|-1$ to $|\fk|$ is the same.

Assume  Lemma \ref{exle} have been proved for $|\fk|=k_1+\cdots+k_l$ varies from $0$ to $|\fk|-1.$ 
Without loss of generality we assume $k_1\geq 1.$  Let $g$ be the image of $f$ in $H^0(M,\frac{K_M\otimes E} {\kj^{k_1 }_{p_1}\kj^{k_2+1 }_{p_2}\cdots \kj^{k_l+1}_{p_l}}).$
 By the induction hypothesis, there exists a lifting section $F_{\fk-1}\in H^0(M,K_M\otimes E)$ such that $J^{\fk-1}_p (F_{\fk-1})=g$ and
\begin{equation}\int_M \frac{|F_{\fk-1}|^2_H}{\Big(|s_1|\cdots |s_l|)^{2n}\Big(\log \frac{1}{|s_j|}\Big)^{\delta_j}} \leq C_1\Vert f\Vert ^2_{\fk-1}.\label{ind23}\end{equation}
Here we denote $\fk-1=(k_1-1,k_2,\cdots,k_l)$ for brevity.
We will construct holomorphic lifting section of  $f-J^{\fk-1}_{\fp} (F_{\fk-1})$ via solving $\bar{\partial}$-equation
 by using H\"omander's $L^2$-theory. Without loss generality we assume $f-J^{\fk-1}_{\fp} (F_{\fk-1})\not=0$ throughout this section, in particular $||f||_{\fk}\not=0.$
 
Let $\tilde{F}$ be a holomorphic lifting of $f$ on the Stein open subset $U.$  We choose  a covering of $M$ by a finite number of bounded Stein open subsets $\{U_{\al}|\al =1,\cdots, N\}$ such that $U$ is a member of $\{U_{\al}\},$ and  the maximal diameter of $\{U_{\al}\}$ is less than $D.$ 
Let $f_{\al}\in H^0(U_{\al}, K_M\otimes E)$  be the extended holomorphic section of  the jet section $f-J^{\fk}_{\fp} (F)\in J^{k_1}_{p_1}(K_M\otimes E)$ on $U_{\al},$
 obtained  by using Lemma \ref{exlem} (note  Lemma \ref{exlem} is still true if the bundle $E$ is only strong Nakano nef by the proof of Theorem \ref{exth}). Then  we have
  \begin{equation} \int_{U_{\al}} \frac{|f_{\al}|^2_H}{\Big(|s_1|\cdots |s_l|)^{2n}\Big(\log \frac{1}{|s_j|}\Big)^{\delta_j}} \leq C_2\Vert f-J^{\fk-1}_{\fp} (F_{\fk-1})\Vert ^2_{\fk}.\label{ualpha}
    \end{equation}
We will glue theses holomorphic sections $\{f_{\al}\}$ together to get a global extension of $\tilde{F}- F_{\fk-1}$ on $M.$
 Let $\{\theta_{\al}\}$ be a smooth partition of unity corresponding to the covering $\{U_{\al}\},$   and   $$\tilde{f}=\sum_{\al}\theta_{\al}f_{\al}.$$
    is a smooth patching of $\{f_{\al}\}.$  By definition of jet section  we have $$\partial ^{\gamma}(\tilde{F}-F_{\fk-1})|_{p_1}=0,~~~{\rm for~~any}~~~~|\gamma|\leq k_1-1.$$
    For any index $\bt$  we have
    \begin{equation}
    D''\tilde{f}|_{U_{\bt}}=\sum_{\al}f_{\al}\bar{\partial}\theta_{\al}|_{U_{\bt}}=\sum_{\al}(f_{\al}-f_{\bt})\bar{\partial}\theta_{\al}|_{U_{\bt}};
    \label{compa1}
    \end{equation}
    here we use $\sum_{\al}\theta_{\al} =1$ and hence  $\sum_{\al}\bar{\partial}\theta_{\al}=0.$ Since both $f_{\al}$ and $f_{\bt}$ 
    are extensions of the jet section  $f-J^{\fk-1}_{\fp} (F_{\fk-1})$, we have 
     \begin{equation}
   \partial ^{\gamma}f_{\al}|_{p_1}=\partial^{\gamma}f_{\bt}|_{p_1}=0,~~~{\rm for~~any}~~~~|\gamma|\leq k_1-1.
    \label{compa2}
    \end{equation} and
      \begin{equation}
   \partial ^{\al}f_{\gamma}|_{p_1}=\partial^{\gamma}f_{\bt}|_{p_1}=\partial^{\gamma}\tilde{F}|_{p_1},~~~{\rm for~~any}~~~~|\gamma|=k_1.
    \label{compa3}
    \end{equation}

 Let $a_{\tau,\el},b_{\tau,\el},v_{\tau,\el},\phi,\eta_{\el}$ be the functions as defined in (\ref{a}),(\ref{bc}) and (\ref{pet}) respectively, here in  (\ref{pet})
  we take $u(x)=x+\frac{1}{\delta_1}x^{\delta_1}$ with $0<\delta_1 <1.$  
     Let $$\widetilde{H}=H_{\rho}e^{-2\sum_{j=1}^l (n+k_j) \log |s_j|}$$  be a smooth Hermitian metric on $K_M\otimes E.$  Note here $H_{\rho}$ are a family smooth Hermitian metric on the strong Nakano nef bundle $E,$ and $$\langle [\Theta_{H_{\rho}},\Lambda]\vartheta,\vartheta\rangle \geq -\rho |\vartheta|_{H_{\rho}}^2$$ 
 for any $E$-valued $(n,1)$-form $\vartheta$ by (\ref{nakanoid}).  Hence $$\langle [ \eta_{\el} \Theta_{\widetilde{H}_1} -id'd''\eta_{\el}-i\lmd d'\eta_{\el}\wedge d''\eta_{\el},\Lm]\vartheta,\vartheta\rangle  \geq a_{\tau,\el}(\phi) |(d''\phi)^* \vartheta|_{H_{\rho}}^2-\rho\eta_{\el} |\vartheta|_{H_{\rho}}^2. $$
For $\tau\ll 0$ we have the same estimate $\eta_{\el}\leq \tau^2$   as in the equality (\ref{taub}).
Take $\rho=\frac{1}{\tau ^4}$   then
$$\rho\eta_{\el} \leq \frac{1}{\tau^2}.$$
     Let
 $$g_{\tau,\el}=D''((1-b_{\tau,\el}(\phi))\tilde{f})=(-a_{\tau,\el}(\phi)d''\phi)\tilde{f}+(1-b_{\tau,\el}(\phi))D''\tilde{f}. $$
By (ii) of Proposition \ref{exist}, there exist $f_{\tau,\el}\in L^2(M,\Lambda^{n,0} T^*_M\otimes E)$ and $ h_{\tau,\el}\in L^2(M,\Lambda^{n,1} T^*_M\otimes E)$ such that
\begin{equation}D'' f_{\tau,\el}+\frac{1}{\tau}h_{\tau,\el}=g_{\tau,\el},\label{cpxeq}\end{equation}
and

\begin{eqnarray}\begin{array}{rcl}
&&\int_{M}(\eta_{\el} +\lambda_{\el}^{-1})^{-1}\frac{|f_{\tau,\el}|^2_{H_{\rho}}}{{\prod}_{j=1}^l |s_j|^{2(n+k_j)}}  +\int_{M}\frac{|h_{\tau,\el}|^2_{H_{\rho}}}{{\prod}_{j=1}^l |s_j|^{2(n+k_j)}}\\
 &\leq &2\Big(\int_{M} \frac{a_{\tau,\el}(\phi)|\tilde{f}|^2_{H_{\rho}}}{{\prod}_{j=1}^l |s_j|^{2(n+k_j)}}+
 \tau^2\int_M \frac{(1-b_{\tau,\el}(\phi))^2}{{{\prod}_{j=1}^l |s_j|^{2(n+k_j)}}}|D''\tilde{f}|^2_{H_{\rho}}\Big).\end{array}\label{mainineq}
 \end{eqnarray}
 
\begin{lemma} For $\tau\ll 0$ there is a bounded positive constant depending only on $D$ and $|\partial^{\gamma}\tilde{F}(p_1)|_H$ with $|\gamma|=k_1+1$ such that
$$\lim_{\el\rw 0}\int_M \frac{(1-b_{\tau,\el}(\phi))^2}{{{\prod}_{j=1}^l |s_j|^{2(n+k_j)}}}|D''\tilde{f}|^2_{H_{\rho}}\leq Ce^{2\tau}.$$
\end{lemma}

\begin{proof}Take a small contractible open neighborhood $U_1$ of $p_1$ and trivialized $E$ over $U_0$ then we may write
$$f_{\al}-f_{\bt}=(\hat{f}_{\al}-\hat{f}_{\bt})dz_1\wedge\cdots \wedge dz_n,$$
where $\hat{f}_{\al},\hat{f}_{\bt}$ are local vector valued holomorphic functions.
 Let $z=(z_1,\cdots,z_m) $ be the local holomorphic coordinate of $U_0$ such that $p_1$ lies at the origin, i.e., $w_1(p_1)=0.$ Then $s_1=\frac{1}{27D}z.$
 Since the support of $1-b_{\tau,\el}(\phi)$ is contained in the subset $\phi^{-1}((-\infty,\tau+\frac{1-\el}{2}))\subset M,$ for $\tau\ll 0$  the support of $1-b_{\tau,\el}(\phi)$ is contained  in $U_{\al}\cap U_{\bt}.$
  Note here $\phi=\log |s_1|.$ We have 
 $$\lim_{\el\rw 0}\frac{(1-b_{\tau,\el}(\phi))^2}{{{\prod}_{j=1}^l |s_j|^{2(n+k_j)}}}\leq 4\Big(\frac{27D}{d}\Big)^{2(nl +|\fk|-k_1-n)} \frac{1}{|s_1|^{2(k_1+n)}}.$$
 By (\ref{compa2}) and (\ref{compa3}),
on $U_1$ we have Taylor expansion
$$\hat{f}_{\al}-\hat{f}_{\bt}=\sum_{|\gamma |\geq k_1 +1}\frac{\partial^{\gamma}\hat{f}_{\al}-\partial^{\gamma}\hat{f}_{\bt}}{\gamma!}z^{\gamma}, $$
where $\gamma=(\gamma_1,\cdots,\gamma_n)$ and $|\gamma|=\gamma_1+\cdots+\gamma_n.$  Let $d{\sigma}=\wedge^n_{j=1} (idz_j \wedge d\bar{z}_j)=2^n r^{2n-1}drdS$
be the volume element and $dS$ is area element of the unit sphere in $\R^{2n}.$
Note for $0<\el<1,$
$$\begin{array}{rcl}
&&\sum_{|\gamma|=k_1+1}\int_M \frac{(1-b_{\tau,\el}(\phi))^2|z|^{2\gamma}}{|s_1|^{2(k_1+n)}}d{\sigma}\\
&=&\sum_{|\gamma|=k_1+1}\int_{|s_1|\leq e^{\tau+\frac{1-\el}{2}}} \frac{(1-b_{\tau,\el}(\phi))^2|z|^{2\gamma}}{|s_1|^{2(k_1+n)}}d{\sigma}\\
&\leq & 4 (27D)^{|\gamma|}\underset{|\gamma|=k_1 +1}{\sum}\int_{|z|\leq e^{\tau+\frac{1-\el}{2}}} \frac{|z|^{2\gamma}}{|z|^{2(k_1+n)}}d{\sigma}\\
&= & 2^{(n+2)} (27D)^{|\gamma|}(\int^{e^{\tau+\frac{1-\el}{2}}}_0 rdr )\underset{|\gamma|=k_1 +1}{\sum}\oint_{|z|=1}\frac{|z|^{2\gamma}}{|z|^{2(k_1+n)}}dS\\
&=& C_De^{2\tau},
\end{array}$$
where $C_D$ is a constant depending only on $D.$ Hence

$$\begin{array}{rcl}
&&\underset{\el\rw 0}{\lim}\int_M \frac{(1-b_{\tau,\el}(\phi))^2}{{{\prod}_{j=1}^l |s_j|^{2(n+k_j)}}}|D''\tilde{f}|^2_{H_{\rho}}\\
&\leq&\frac{4}{(\gamma!)^2}\Big(\frac{27D}{d}\Big)^{2(nl+|\fk|-k_1-n)}\underset{|\gamma|=k_1 +1}{\sum}\int_M \frac{|z|^{2\gamma}(|\partial^{\gamma}\hat{f}_{\al}|^2_{H_{\rho}}+|\partial^{\gamma}\hat{f}_{\bt}|^2_{H_{\rho}})|\bar{\partial}\theta |^2}{|z|^{2(k_1+n)}}
d\sigma\\
&\leq &\underset{\substack{1\leq \al\leq N\\ p\in M}}{\max}(|\bar {\partial}\theta_{\al}|^2(p))\underset{\substack{|\gamma|=k_1 +1\\ p\in M}}{\max}(|\partial^{\gamma}\hat{f}_{\al}|^2_{H_{\rho}}(p)+|\partial^{\gamma}\hat{f}_{\bt}|^2_{H_{\rho}}(p))C_De^{2\tau}+O(e^{2\tau})\\
&\leq & Ce^{2\tau}.
\end{array}$$
Since $\partial^{\gamma}\hat{f}_{\al}(z)\rw \partial^{\gamma}\hat{F}(p_1) $ and $\partial^{\gamma}\hat{f}_{\bt}(z)\rw \partial^{\gamma}\hat{F}(p_1) $ as $\tau\rw -\infty,$  hence for $\tau\ll 0$
$C$ is a bounded positive constant depending only on $D$ and $|\partial^{\gamma}\tilde{F}(p_1)|_H$ with $|\gamma|=k_1+1.$
\end{proof}

  Now we continue  to estimate the R.H.S. of  (\ref{mainineq}).  We may  decompose the first term of the R.H.S. of  (\ref{mainineq}) as sum of integral of each bounded Stein open set $U_{\al},$
  
  $$\int_{M} \frac{a_{\tau}(\phi)|\tilde{f}|^2_{H_{\rho}}}{{\prod}_{j=1}^l |s_j|^{2(n+k_j)}}\leq \sum_{\al} \int_{U_{\al}} \frac{a_{\tau}(\phi)|\tilde{f}|^2_{H_{\rho}}}{{\prod}_{j=1}^l |s_j|^{2(n+k_j)}}.$$
 Recall that each $f_{\al}$  is a extension of $f-J^{\fk-1}_{\fp}F_{\fk -1}$ on $U_{\al}.$ If $U_{\al}\cap U=\emptyset$ then  $\int_{U_{\al}} \frac{a_{\tau}(\phi)|\tilde{f}|^2_{H_{\rho}}}{{\prod}_{j=1}^l |s_j|^{2(n+k_j)}}=0$  when $\tau\ll 0;$   if $U_{\al}\cap U\not=\emptyset$ then  by  Lemma \ref{claim} (in fact (\ref{latecite}) in the proof of  Lemma \ref{claim} is enough),
 $$\lim_{\tau\rw -\infty}\int_{U_{\al}} \frac{a_{\tau}(\phi)|\tilde{f}|^2_{H_{\rho}}}{{\prod}_{j=1}^l |s_j|^{2(n+k_j)}}\leq C_2\Vert f-J^{\fk-1}_{\fp} (F_{\fk-1})\Vert_{\fk}\leq 2C_2\Vert f\Vert_{\fk},$$ 
 here the length $\Vert f-J^{\fk-1}_{\fp} (F_{\fk-1})\Vert_{\fk}$ and $\Vert f\Vert_{\fk}$ are firstly with respect to $H_{\rho},$ since (each entry of) $H_{\rho}$ is continuous and uniformly convergent 
 to $H,$ so finally  the length $\Vert f-J^{\fk-1}_{\fp} (F_{\fk-1})\Vert_{\fk}$ and $\Vert f\Vert_{\fk}$ in the above inequality are with respect to $H$ after taking the limit of the metric when $\tau\rw -\infty.$  
 As a result   
  $$\lim_{\tau\rw -\infty}\int_{M} \frac{a_{\tau}(\phi)|\tilde{f}|^2_{H_{\rho}}}{{\prod}_{j=1}^l |s_j|^{2(n+k_j)}}\leq C_3\Vert f\Vert_{\fk}$$ 
for a sufficiently large and  positive number $C_3<+\infty.$  Hence the  limit of the R.H.S. of  (\ref{mainineq}) 

\begin{equation}\lim_{\tau\rw -\infty}\Big(\int_{M} \frac{a_{\tau}(\phi)|\tilde{f}|^2_{H_{\rho}}}{{\prod}_{j=1}^l |s_j|^{2(n+k_j)}}+
 \tau^2\int_M \frac{(1-b_{\tau}(\phi))^2}{{{\prod}_{j=1}^l |s_j|^{2(n+k_j)}}}|D''\tilde{f}|^2_{H_{\rho}} \Big)\leq C_3\Vert f\Vert_{\fk}, \label{rhs}\end{equation}
since $\tau^2 e^{2\tau}\rw 0$ if $\tau\rw -\infty.$  So the integrals of the L.H.S. of (\ref{mainineq}) are uniformly bounded.

By equation (\ref{cpxeq}), $D''h_{\tau,\el}=\tau\cdot D''( g_{\tau,\el}-D''f_{\tau,\el}) =0,$ and $h_{\tau,\el}$ is square integrable with respect to the metric $H_{\rho}e^{-2\sum_{j=1}^l(n+k_j) \log |s_j|}$
from the estimate in (\ref{mainineq}) and (\ref{rhs}).  Hence by the existence theorem of $\bar{\partial}$-equation on  Stein manifold, there exists $E$-valued  function $l^{\al}_{\tau,\el}$ defined on $U_{\al}$
such that
 $$D''l^{\al}_{\tau,\el}=h_{\tau,\el}$$ 
 with the estimate
  $$\int_{U_{\al}}\frac{|l^{\al}_{\tau,\el}|^2_{H_{\rho}}}{{\prod}_{j=1}^l |s_j|^{2(n+k_j)}}  \leq C''\int_{U_{\al}}\frac{|h_{\tau,\el}|^2_{H_{\rho}}}{{\prod}_{j=1}^l |s_j|^{2(n+k_j)}}$$
  by H\"ormander $L^2$-estimate.  
   Let
 $$I_{\tau,\el}=(1-b_{\tau,\el}(\phi))\tilde{f}-f_{\tau,\el}+\frac{1}{\tau}l^{\al}_{\tau,\el}.$$
Then $D''I_{\tau,\el}=0$ and  $I_{\tau,\el}$ is holomorphic section. It in particular means that $f_{\tau,\el}-\frac{1}{\tau}l^{\al}_{\tau,\el}\in C^{\infty}(U_{\al})$ is a smooth section. Moreover

\begin{equation}\liminf_{\tau\rw -\infty} \liminf_{\el\rw 0}\int_{U_{\al}}\frac{(\eta_{\el} +\lambda_{\el}^{-1})^{-1}}{{\prod}_{j=1}^l |s_j|^{2(n+k_j)}}|f_{\tau,\el}|^2_{H_{\rho}} 
 \leq  C_3\Vert f\Vert_{\fk}.
\end{equation}
So  as $\el\rw 0,$ $I_{\tau,\el},f_{\tau,\el},h_{\tau,\el}$ and  $h^{\al}_{\tau,\el}$ have subsequences
which have weak limits in weighted $L^2$-space on each $U_{\al}$, denoted by  $I_{\tau}, f_{\tau}, h_{\tau}$ and $l^{\al}_{\tau}$ respectively.  Note $I_{\tau}$ is a holomorphic section and 
 $I_{\tau}=(1-b_{\tau}(\phi))\tilde{f}-f_{\tau}+\frac{1}{\tau}l^{\al}_{\tau}.$ 
 Taking weak limit as $\el\rw 0 $ after further taking subsequence, and using the formula (\ref{btau})  we know
\begin{equation} \int_{U_{\al}} \frac{|{f}_{\tau}|^2_{H_{\rho}}} {{{\prod}_{j=1}^l |s_j|^{2(n+k_j)}}(\log\frac{1}{|s_1|})^{\delta_1}}   
 <+\infty.
\label{repat2}\end{equation}
In the same way we  know  $\int_{U_{\al}}\frac{|l^{\al}_{\tau}|^2_{H_{\rho}}}{{\prod}_{j=1}^l |s_j|^{2(n+k_j)}}   <+\infty.$
Since $\log\frac{1}{|s_1|}>1$ we know that  on every $U_{\al}$
$$ \int_{U_{\al}}\frac{|\frac{1}{\tau}l^{\al}_{\tau}-f_{\tau}|^2_{H_{\rho}}}{{\prod}_{j=1}^l |s_j|^{2(n+k_j)}}   < +\infty.$$
Hence the integral
$$ \int_{U_{\al}}\frac{|\frac{1}{\tau}l^{\al}_{\tau,\el}-f_{\tau,\el}|^2_{H_{\rho}}}{{\prod}_{j=1}^l |s_j|^{2(n+k_j)}}  $$ is uniformly bounded for each $U_{\al},$
since $f_{\tau,\el}-\frac{1}{\tau}l^{\al}_{\tau,\el}\in C^{\infty}(U_{\al})$ is a smooth  we must have 

\begin{equation}\partial^{\gamma}(f_{\tau,\el}-\frac{1}{\tau}l^{\al}_{\tau,\el})(p_j)=0,~~~{\rm for}~~~ |\gamma|\leq k_j, ~~~~{\rm with}~~j=1,\cdots,l.\label{dbzero}  \end{equation}

Note $\rho=\frac{1}{\tau^4}.$ By definition of strong Nakano nef vector bundle $\{H_{\rho}\}$ have a subsequence convergent to a limit Hermitian metric $H$ with
Nakano pseudo effective curvature current.  Hence  take limit as $\tau\rw -\infty$
of the inequality (\ref{mainineq}) and use (\ref{rhs}) we obtain

\begin{equation} \liminf_{\tau\rw -\infty}\liminf_{\el\rw 0}\int_{M}\frac{(\eta_{\el} +\lambda_{\el}^{-1})^{-1}}{{\prod}_{j=1}^l |s_j|^{2(n+k_j)}}|f_{\tau,\el}|^2_{H_{\rho}} 
 \leq  C_3\Vert f\Vert_{\fk}.
\end{equation}
Since $f_{\tau}$ and $l^{\al}_{\tau}$ are weighted $L^2$-bounded, hence $I_{\tau}$ are weighted $L^2$-bounded.
After further taking subsequences we may assume the holomorphic sections $I_{\tau}=(1-b_{\tau}(\phi))\tilde{f}-f_{\tau}+\frac{1}{\tau}l^{\al}_{\tau}$  is weakly convergent to a holomorphic section, and the limit  we denote by $F_k.$
Note any subseqence of $(1-b_{\tau}(\phi))\tilde{f}$ and $\frac{1}{\tau}h^{\al}_{\tau}$ are always weakly  convergent to zeros on each open set $U_{\al}.$ Hence by (\ref{btau}) and noting 
$|s_j|<1$ for $j=1,\cdots,l$ we get
\begin{equation}
 \int_{M}\frac{|F_k|^2_{H}}{{(|s_1|\cdots |s_l|)^{2n}}\Big(\log\frac{1}{|s_1|}\Big)^{\delta_1}}
\leq C_4\Vert f\Vert_{\fk},\label{ind24}
\end{equation}
  where $C_4$ is a bounded positive constant.

  Now put $F=F_k +F_{\fk-1}.$   Note that $I_{\tau,\el}=(1-b_{\tau,\el}(\phi))\tilde{f}-f_{\tau,\el}+\frac{1}{\tau}l^{\al}_{\tau,\el}\rw F_k$ when $\el\rw 0$ and $\tau\rw -\infty.$ 
  Note by definition,  $(1-b_{\tau,\el}(\phi))$ is a globally  defined smooth section on $M.$
  since $f_{\tau,\el}$ is obtained via solving $\bar{\partial}$-equations on $M$ it is globally defined on the whole $M$ too.  
   Though  $l^{\al}_{\tau,\el}$ is obtained by solving $\bar{\partial}$-equations on $U_{\al},$ they may not be able to glue up to a global section, but
    $\frac{1}{\tau}l^{\al}_{\tau,\el}$ is weakly convergent to zero as $\el\rw 0$
  and $\tau\rw -\infty.$ Hence the weak limit $F_k$ of $I_{\tau,\el}$ is well defined a global holomorphic section on $M.$    
  
  Without loss of generality,   in the following we assume $I_{\tau,\el},f_{\tau,\el},h_{\tau,\el}$ and  $h^{\al}_{\tau,\el}$ themselves (not passing to subsequences) are  weakly convergent
  as $\el\rw 0$ and $\tau\rw -\infty.$

Now we check $F$ is the expected holomorphic extension.  By definition of $b_{\tau,\el},$  when $p$ is inside a sufficiently small neighborhood  of $p_1$  we have $b_{\tau,\el}(\phi(p))=0$ and hence in a small neighborhood of $p_1$ we have
 $I_{\tau,\el}(p)=\tilde{f}-f_{\tau,\el}+\frac{1}{\tau}l^{\al}_{\tau,\el}.$ By (\ref{dbzero})
 we have
$$\partial ^{\gamma} F(p_1)=\partial ^{\gamma} \tilde{f}(p_1)$$ for  any $\gamma$ with $|\gamma|\leq k_1.$
For any $p\not=p_1$ we have $$\lim_{\tau\rw -\infty}\lim_{\el\rw 0}(1-b_{\tau,\el}(\phi(p)))=0,$$ 
In fact if $\tau\ll 0$ then we have $b_{\tau,\el}(\phi(p))=1$ for any fixed $p\not=p_1.$  In particular for any $j\geq 2$ by (\ref{dbzero})  we have
$$\partial ^{\gamma} F(p_j)=-\lim_{\tau\rw \infty}\lim_{\el\rw 0}\partial ^{\gamma} (\frac{1}{\tau}l^{\al}_{\tau,\el}-f_{\tau,\el})(p_j)+\partial ^{\al} F_{\fk-1}(p_j)=\partial ^{\al} \tilde{F}(p_j)$$ 
for any $\gamma$ with $|\gamma|\leq k_j.$  
Hence $F$ is the expected extension of jet section $f\in J^{\fk}_{\fp}(K_M\otimes E).$
Note that $|F|^2_H\leq 2(|F_{k}|^2_H +|F_{\fk-1}|^2_H),$
By the induction assumption   (\ref{ind23}) and
  (\ref{ind24}), 
$$\int_M \frac{|F|^2_H}{ (|s_1|\cdots |s_l|)^{2n} \Big(\log \frac{1}{|s_|}\Big)^{\delta_1}}\leq  C_4\Vert f\Vert ^2_{\fk}+C_1\Vert f\Vert ^2_{\fk-1}=C\Vert f\Vert ^2_{\fk},$$
where  $C$ is a finite positive  constant. This establishes {\rm ($D_{\fk ,1}$)}   and hence {\rm ($C_{\fk }$)} by using arithmetic-geometric inequality as what we have done in the proof of Theorem \ref{exthm}.
  
  \end{proof}
  
  The following Corollaries are proved in the same way with Corollary \ref{excor1} and \ref{excor2}, as they are concluded from Theorem \ref{exthm}.
\begin{corollary}\label{excor3}
Let $M$ be a  compact K\"ahler manifold and  $E$  a  strong Nakano nef holomorphic vector bundle   on $M,$  with a Nakano pseudo effective limit Hermitian metric $H.$ 
Let $\{p_1,\cdots,p_l\}$ be a finite set locate in a  stein open ball $U$ of $M.$ Let $z$ be holomorphic coordinate of $U$ with $z(p_j)=w_j$  and
    $s_j (z)=z-w_j$ for $j=1,\cdots,l.$  Then for
 any holomorphic jet section $f\in J^{\fk}_{\fp} (K_M \otimes E)$ satisfying $$\Vert f\Vert ^2_{\fk}<\infty,$$ 
 there exists holomorphic section $F\in H^0(M,K_M \otimes E) $  such that $J^{\fk}_{\fp} (F) =f$ and there is a  finite positive  constant $C$ with
$$\int_M \frac{|F|^2_H}{{\prod}_{j=1}^l (1+|s_j|^2)^{n}} \leq C\Vert f\Vert ^2_{\fk},$$
 where $\delta_j >0 $ are  any positive constants.
\end{corollary}

\begin{corollary}\label{excor4}
Let $M$ be a  compact K\"ahler manifold and  $E$  a  strong Nakano nef holomorphic vector bundle  on $M,$  with a Nakano pseudo effective limit Hermitian metric $H.$ 
Then for any holomorphic jet section $f\in J^{\fk}_{\fp} (K_M \otimes E)$ whose support locate is a finite subset contained a Stein open ball  $U$ in $M,$ 
satisfying $$\Vert f\Vert ^2_{\fk}<\infty,$$ there exists holomorphic section $F\in H^0(M,K_M \otimes E) $  such that $J^{\fk}_{\fp} (F) =f$ and there is a  finite positive  constant $C$ with
$$\int_M |F|^2_H \leq C\Vert f\Vert ^2_{\fk},$$
 where $\delta_j >0 $ are  any positive constants.
 \end{corollary}

\section{Proof of the main theorem}\label{subs5}

To prove the main theorem  we will show $M$ has many global holomorphic vector fields if its tangential bundle is strong Griffiths positive.   Geometrically, if the follows of the holomorphic 
vector fields could connect with one another such that a point start one place and go along the flows  it could arrive everywhere of $M$ then $M$ is homogeneous.
 Unfortunately the point may stop at some places, which 
are exactly the zero points of the vector fields.
If all extended holomorphic vector fields have no zero points, on the Lie group level which means the follows defined by the  holomorphic vector fields has no fixed points, then the holomorphic automorphic group $\Aut(M)$ will act almost freely on $M.$ If moreover the extended holomorphic vector fields could span the  tangent space of $M$  everywhere, 
  which means the action of $\Aut(M)$ has open and dense orbit,
  then $M$ is holomorphic homogeneous. Unfortunately, we could not guarantee the extended holomorphic vector fields have no zero points  outside the points where they already have the given nonzero values. To overcome this difficulty, we consider extending holomorphic jet sections.  If we take enough many holomorphic jet vector fields defined on a Stein neighborhood, and arrange their jet values such that they linearly spanned  the tangent space in a neighborhood,  and if these jet vector fields are all extendable on the whole manifold $M,$ then it implies that $\Aut(M)$ acts transitively in the neighborhood. If we could repeat this process  everywhere on $M,$ then we may prove $M$ is homogenous.

Though as a Fano manifold, $M$ is projective and we may cut a hyper surface $L$ out of $M$ such that the rest $M\backslash L$ is a Stein manifold.  Unfortunately  it is very difficult to 
construct a complete K\"ahler metric on   $M\backslash L$ such that the diameter of $M\backslash L$ is bounded.  Since in  Theorem \ref{exthm},  Corollary \ref{excor1} and \ref{excor2},  the upper bounds of the $L^2$-estimate of the extended sections
depend on the diameter of the Stein manifold and is not necessary uniformly bounded, even we could extend a holomorphic section to defined on the whole of $M\backslash L,$ we could not extend it further
to $M.$   So we need using the extension theorem established on the K\"ahler manifolds.

        \begin{theorem} \label{autm} If $M$ is a compact K\"ahler manifold with strong Griffiths  nef tangent bundle, then $M$ is a homogeneous complex manifold under action of its holomorphic automorphism group.
\end{theorem}

\begin{proof}

Let $\Aut(M)$ denote the holomorphic automorphism group of $M.$ The natural holomorphic action $\Aut(M)\times M\rw M$ is denoted by mapping $(g,p)\longmapsto g\cdot p.$
Firstly we will prove the $\Aut(M)$ action is locally homogeneous. That means, for any point $p\in M$ and for an open ball $B_p (r)$  of  sufficient small radius $r$ with center $p,$   for any $q\in B_p (r)$ there exist $g\in \Aut(M)$ such that $g\cdot q=p.$

 Since the quotient map $$\ko_M\rw\frac{\ko _M} {\kj^2_{p_1}\kj^2_{p_2}\kj^2_{p_3}}\cong\oplus_{j=1}^3 (\C\oplus T^*_M|_{p_j})$$ is defined by 
 $f\longmapsto (f(p_1)+df(p_1),f(p_2)+df(p_2),f(p_3)+df(p_3))$ is surjective, hence  we have a surjective quotient map
 $$T_M\rw \frac{T_M}{\kj^2_{p_1}\kj^2_{p_2}\kj^2_{p_3}}\cong \oplus_{j=1}^3(T_M|_{p_j}\oplus \End (T_M|_{p_j})).$$ Take a jet $X\in H^0 (M, T_M\otimes\ko _M /{\kj^2_{p_1}\kj^2_{p_2}\kj^2_{p_3}})$ 
with $X(p_j)\not=X(p_k)$ and  $dX(p_j)\not=dX(p_k)$ for $1\leq j\not=k\leq 3.$ Hence without loss generality we assume $X(p_1)\in TM$ is not a zero vector and at the same time $dX(p_1)\not=0.$ Now we claim that we could $X$ to a global holomorphic vector field defined everywhere on $M.$
Since the tangent bundle $T_M$ of $M$ is strong Griffiths nef, by Proposition \ref{naeff}, $$E:=K^*_M\otimes T_M=\det(T_M)\otimes T_M$$
 is strong Nakano nef. Hence there is a limit Hermitian metric $H$ on $E$ whose curvature current is Nakano pseudo effective. Moreover every entry of $H$ is continuous on $M$ by Proposition \ref{naeff}, hence $||X||_{\fk} <+\infty$ with $\fk=(2,2,2).$
Note  $T_M=K_M\otimes E.$  By Corollary \ref{excor4},  $X$ could be extended to a global holomorphic vector field defined everywhere on $M$ and still denoted by $X\in H^0(M,T_M).$

Take a smooth curve $C_{pq}\subset B_p(r)$ connected $p$ and $q.$ For any $p'\in C_{pq},$  we could
 take jet vectors $X_1(p'),\cdots,X_m(p')$ such that they span the tangent space $T_M|_{p'},$ at the same time  
 we may assume  $dX_1(p'),\cdots, dX_n(p') $ are all nonzero.  Now extend all of them to global holomorphic vector fields $X_1,\cdots,X_m$ on $M,$ just in the same way we extend $X$ in the above
 paragraph. 
 Then $X_1,\cdots,X_m$ will span the tangent space $T_M$ in a neighborhood $U_{p'}$ of $p'.$  So the flows of $X_1,\cdots,X_n$ sweep out the open  neighborhood $U_{p'}$ by the Frobenius theorem. 
 Now for each $p'\in C_{pq}$ we repeat the process, we get an open cover $\{U_{p'}|p'\in C_{p,q}\}$ of $C_{pq}.$ By Heine-Borel's theorem, we get a finite subcover. Hence there are finite number of global holomorphic vector fields whose flow  may sweep out an open   neighborhood  of $C_{pq}.$ It in particular means that there exist $g\in\Aut(M)$  such that $g\cdot q=p.$

 Since for each $p\in M$ there is a open neighborhood $B_p(r)$ which is homogeneous  under the action of $\Aut (M)$ and $M$ is compact, using  Heine-Borel's theorem once again, we know $M$ is homogeneous  under the holomorphic action of $\Aut (M).$  Hence $M$ is a homogeneous complex manifold.\end{proof}
 
 \begin{corollary} If $M$ is a compact K\"ahler manifold with Griffith semipoistive tangent bundle, then $M$ is a homogeneous complex manifold under action of its holomorphic automorphism group.
 \end{corollary}
 
 \vskip 0.8cm
    
\noindent{\it Proof of the main theorem}.
By Theorem \ref{autm}, $M$ is a homogeneous complex manifold.  Since $M$ is meanwhile a projective K\"ahler manifold, by Borel-Remmert's theorem \cite{br}, $M$ is s the direct product of a complex torus and a projective-rational manifold. Note  any Fano manifold is simply connected, hence $M$ is a homogeneous  projective-rational manifold.
\qed

\vskip 0.8cm

\noindent{\bf Acknowledgments} I'm grateful to professor Yum-Tong Siu for his enlightening and stimulating discussions  and for explaining to me many of his important works and their backgrounds,  in particular the works related to the extensions theorem and positive closed currents; during the past year he not only taught me mathematics but also guided me to think mathematics in a simple way.
I would like to thank professor Shing-Tung Yau for encouragements and the friends in Yau's students siminar for helping me in many ways. I would take this chance to thank Professors Alan, Huckleberry, Peter, Heinzner and Takeo, Ohsawa, I studied under their guidance before and their works on complex Lie group actions and $L^2$-techniques in several complex variable
 inspired  me to think the problem studied in this paper for a long time.
 Finally I would like to thank  the department of mathematics of Harvard university for its hospitality. My visit to Harvard university is supported by the Lingnan foundation of Sun Yat-Sen university.

 \end{document}